\documentclass[11pt, reqno]{amsart}
\usepackage[dvipsnames]{xcolor}
\usepackage{lscape}
\definecolor{gB}{HTML}{2b83ba}
\definecolor{gR}{HTML}{d7191c}
\definecolor{gP}{HTML}{902fa3}
\definecolor{gG}{HTML}{00A64F}

\definecolor{OIred}{rgb}{0.90,0.20,0.20}
\definecolor{OIorange}{rgb}{0.90,0.60,0.0}
\definecolor{OIblue}{rgb}{0.0, 0.45, 0.70}
\definecolor{OIgreen}{rgb}{0.0,0.60,0.50}


\usepackage[margin=0cm]{caption}

\usepackage{hyperref}
\hypersetup{
	colorlinks=true,
	citecolor=OIblue,
	linkcolor=OIblue,
	filecolor=magenta,      
	urlcolor=OIblue,
}
\usepackage[capitalise,noabbrev, nameinlink]{cleveref}

\usepackage[a4paper,hmarginratio=1:1, bottom=1.0in]{geometry}
\usepackage{amsmath, amsthm, amssymb, graphicx, dsfont, stmaryrd, extarrows}
\usepackage[T1]{fontenc}
\usepackage{textcomp}
\usepackage{lmodern}
\usepackage{microtype}

\usepackage[color,matrix,arrow,all,cmtip]{xy}
\providecommand{\1}{}
\renewcommand{\1}{\mathds{1}}
\newcommand{\unit}{\mathds{1}}

\usepackage{amscd}
\usepackage{mathrsfs}
\usepackage{tikz-cd}
\usetikzlibrary{matrix,patterns}
\usetikzlibrary{decorations.pathmorphing}
\usetikzlibrary{shapes.geometric}
\usepackage[colorinlistoftodos]{todonotes}
\usepackage{color}
\usepackage{bbm}
\usepackage{verbatim}
\usepackage[mathscr]{euscript}
\usepackage[shortlabels]{enumitem}

\numberwithin{equation}{section}
\theoremstyle{plain}
\newtheorem{thm}[equation]{Theorem}

\newtheorem*{thm*}{Theorem}
\newtheorem*{prop*}{Proposition}
\newtheorem*{mthm*}{Meta Theorem}
\newtheorem*{cor*}{Corollary}

\newtheorem{prop}[equation]{Proposition}

\newtheorem{cor}[equation]{Corollary}       
\newtheorem{lem}[equation]{Lemma}

\theoremstyle{definition} 
\newtheorem{defn}[equation]{Definition} 
\crefname{defn}{Definition}{Definitions}
\newtheorem{hyp}[equation]{Hypothesis}
\newtheorem{ex}[equation]{Example}

\newtheorem{rem}[equation]{Remark} 
\crefname{rem}{Remark}{Remarks}  

\newtheorem{nota}[equation]{Notation}
\crefname{nota}{Notation}{Notations}
\newtheorem{con}[equation]{Construction}

\newcommand{\A}{\mathscr{A}}
\newcommand{\C}{\mathscr{C}}
\newcommand{\D}{\mathscr{D}}

\newcommand{\g}{\mathfrak{g}}
\newcommand{\p}{\mathfrak{p}}
\newcommand{\q}{\mathfrak{q}}

\newcommand{\m}{\mathfrak{m}}

\newcommand{\T}{\mathsf{T}}
\newcommand{\Sp}{\mathrm{Sp}}
\newcommand{\cubeL}{\mrm{Cb}}
\newcommand{\cubeR}{\mrm{Cb}^\msf{m}}
\newcommand{\cubeLR}{\mrm{CbCb}^\msf{m}}
\newcommand{\cubeLRplus}{\mrm{CbCb}^\msf{m}_+}
\newcommand{\Rwithout}[1]{R_{\backslash{#1}}}
\newcommand{\equals}[1]{\mrm{Pc}({#1})}
\newcommand{\equalsmix}[1]{\mrm{Pc}^\msf{m}({#1})}
\newcommand{\glue}[1]{\mrm{PcPc}^\msf{m}({#1})}
\newcommand{\glueplus}[1]{\mrm{PcPc}^\msf{m}_+({#1})}
\newcommand{\biggerthan}[2]{\C^{\geqslant {#1}}({#2})}
\newcommand{\biggerthanf}[2]{\C_\mrm{c}^{\geqslant {#1}}({#2})}

\newcommand{\s}[1]{(#1)}

\renewcommand{\mod}[2][]{\mathrm{Mod}_{#1}(#2)}

\newcommand{\mc}[1]{\mathcal{#1}}

\renewcommand{\phi}{\varphi}

\newcommand{\Q}         {{\mathbb{Q}}}
\newcommand{\Z}         {{\mathbb{Z}}}

\newcommand{\Hom}       {\operatorname{Hom}}
\newcommand{\uHom}{\underline{\Hom}}
\newcommand{\Map}       {\operatorname{Map}}

\newcommand{\supp}{\operatorname{supp}}

\newcommand{\cF}{\mathcal{F}}

\newcommand{\cK}{\mathcal{K}}

\newcommand{\tensor}{\otimes}

\renewcommand{\dim}{\mathrm{dim}}

\newcommand{\tors}{\mathrm{tors}}
\newcommand{\ad}{\mathrm{ad}}
\newcommand{\X}{\mathscr{A}}
                                
\newcommand{\Spc}{\mathrm{Spc}}                                
\newcommand{\downcl}{\land}
\newcommand{\upcl}{\lor}
\newcommand{\Lambdale}[1]{\Lambda_{\leqslant #1}} 
\newcommand{\Gammale}[1]{\Gamma_{\leqslant #1}} 
\newcommand{\Lge}[1]{L_{\geqslant #1}} 
\newcommand{\Elr}[1]{E\langle  #1\rangle}

\newcommand{\fp}{\mathfrak{p}}

\newcommand{\mrm}[1]{\mathrm{#1}}
\newcommand{\msf}[1]{\mathsf{#1}}

\newcommand{\Jordan}[1]{\textcolor{red}{#1}}

\newcommand{\cad}{\C_{\mrm{ad}}}

\newcommand{\ct}{\C^\X_{\mrm{t}}}
\newcommand{\cf}[1]{\C_{\mrm{c}}^{\geqslant {#1}}}

\newcommand{\ext}[3]{\mathsf{opr}_{#1}^{#2}(#3)}
\newcommand{\res}[2]{\mathsf{laxr}^{#1}(#2)}

\newcommand{\Cat}{\mathsf{Cat}}    % infty-category of small infty-categories
\DeclareRobustCommand{\SkipTocEntry}[5]{}
\title{Torsion models for tensor-triangulated categories}
\begin{document}
\author{Scott Balchin}
\address[Balchin]{Queen's University, Belfast, BT7 1NN, UK}
\email{s.balchin@qub.ac.uk}
\author{J.P.C. Greenlees}
\address[Greenlees]{Warwick Mathematics Institute, Zeeman Building, Coventry, CV4 7AL, UK}
\email{john.greenlees@warwick.ac.uk}
\author{Luca Pol}
\address[Pol]{Fakult\"{a}t f\"{u}r Mathematik, Universit\"{a}t Regensburg, Universit\"{a}tsstra{\ss}e 31, 93053 Regensburg, Germany}
\email{luca.pol@ur.de}
\author{Jordan Williamson}
\address[Williamson]{Department of Algebra, Faculty of Mathematics and Physics, Charles University in Prague, Sokolovsk\'{a} 83, 186 75 Praha, Czech Republic}
\email{williamson@karlin.mff.cuni.cz}
\maketitle

\begin{abstract}
Given a rigidly-compactly generated tensor-triangulated category
whose Balmer spectrum is finite dimensional and Noetherian, we
construct a \emph{torsion model} for it, which
is equivalent to the original tensor-triangulated category. The torsion model is
determined in an adelic fashion by objects with singleton
supports. This categorifies the Cousin complex from algebra, and the
process of reconstructing a spectrum from its monochromatic layers in chromatic stable homotopy theory. This model is inspired by work of the second author in rational equivariant stable homotopy theory \cite{Greenlees23}, and extends previous work~\cite{torsion1} of the authors from the one-dimensional setting. 
\end{abstract}
\setcounter{tocdepth}{1}

\makeatletter
\def\@tocline#1#2#3#4#5#6#7{\relax
  \ifnum #1>\c@tocdepth % then omit
  \else
    \par \addpenalty\@secpenalty\addvspace{#2}%
    \begingroup \hyphenpenalty\@M
    \@ifempty{#4}{%
      \@tempdima\csname r@tocindent\number#1\endcsname\relax
    }{%
      \@tempdima#4\relax
    }%
    \parindent\z@ \leftskip#3\relax \advance\leftskip\@tempdima\relax
    \rightskip\@pnumwidth plus4em \parfillskip-\@pnumwidth
    #5\leavevmode\hskip-\@tempdima
      \ifcase #1
       \or\or \hskip 1em \or \hskip 2em \else \hskip 3em \fi%
      #6\nobreak\relax
    \hfill\hbox to\@pnumwidth{\@tocpagenum{#7}}\par% <---- \dotfill -> \hfill
    \nobreak
    \endgroup
  \fi}
\makeatother

\tableofcontents

\section{Introduction}
Tensor-triangulated categories appear naturally in many contexts: 
in algebra through the derived category of a commutative ring, in 
representation theory via the stable module category, and in topology, 
via the category of spectra or its chromatic, equivariant, and motivic 
variants. In each case there is a corresponding support theory:
local cohomology in algebra and algebraic geometry, support
varieties in modular representation theory~\cite{BCR,Carlson},
geometric isotropy in equivariant topology and chromatic
support~\cite{HS} in chromatic stable homotopy
theory. Objects with small support can be viewed as building
blocks. Accordingly, one may analyse complicated objects by breaking
them into pieces with small support, and recording the data  necessary
for reassembly. In this paper, we implement this procedure, showing
that a nice enough
tensor-triangulated category has a
model built from  objects with singleton supports and assembly data.

We consider a rigidly-compactly generated tensor-triangulated category $\T$, and write $\T^\omega$ for the full subcategory of compact objects. 
Commutative rings and their modules can be studied by geometric
methods using  the Zariski spectrum. More generally, from the tensor-triangulated
category $\T$, one may form the  Balmer spectrum 
$\Spc(\T^\omega)$~\cite{Balmer05},  whose points are the thick
$\otimes$-ideals which are prime. This is designed to provide a
universal support theory
for compact objects in $\T$,  and in fact it  classifies the thick
$\otimes$-ideals of compact objects. 

The geometry is simplest when the Balmer spectrum is Noetherian and
finite dimensional. Just as in algebra, one may associate to any
Balmer prime $\p$  a localization functor $L_\p$ which picks out the
part of objects supported above $\p$, a derived torsion functor
$\Gamma_\p$ which picks out the part supported below $\p$, and a
derived completion functor $\Lambda_\p$. We note that objects
$\Gamma_\p L_\p X$ with support precisely at $\p$ play a distinguished
role both conceptually, and in applications. For instance, understanding the localizing subcategories generated by such objects controls the behaviour of arbitrary localizing subcategories as in the theory of stratifications~\cite{BHS, BIK}.

In~\cite{adelicm}, the first two authors constructed an \emph{adelic model} for a
tensor-triangulated category $\T$ which makes precise how to
reconstruct $\T$ from localized complete pieces. However, from the point of view of support, the objects in the adelic model are large, as they have support spread across many primes, so one may seek an alternative model in which the atomic pieces are supported at single primes only. 

In order to construct decompositions of objects $X$ of $\T$ into pieces with small support, one filters the Balmer spectrum by dimension
\[\varnothing = \Spc(\T^\omega)_{\leqslant -1} \subseteq \Spc(\T^\omega)_{\leqslant 0} \subseteq \Spc(\T^\omega)_{\leqslant 1} \subseteq \cdots \subseteq \Spc(\T^\omega).\] This filtration yields localization functors $\Lge{i}$ supported at dimension $i$ and above, and torsion functors $\Gammale{i}$ supported at dimension $i$ and below, for each $i$. Applying these at the level of objects produces a tower
\[X = \Lge{0}X \to \Lge{1}X \to \Lge{2}X \to \cdots \]
and each fibre $\mrm{fib}(\Lge{i}X \to \Lge{i+1}X) \simeq \Gammale{i}\Lge{i}X$ splits into its contributions from each prime of dimension $i$, namely \[\Gammale{i}\Lge{i}X \simeq \bigoplus_{\mrm{dim}(\p) = i}\Gamma_\p L_\p X.\] In algebra, these small pieces are the local cohomology modules $H_\p^*(X_\p)$ with the above filtration giving the Cousin complex, and in chromatic homotopy theory, these building blocks are the monochromatic layers.

The goal of this paper is to promote the previous discussion to an
equivalence of categories, and to do so in a way that is as algebraic
as possible in well-behaved examples.
For now, we state an informal version of our main theorem. We will motivate this and provide a more formal version in the remainder of the introduction.
\begin{thm*}[Informal version]
A rigidly-compactly generated tensor-triangulated category with finite dimensional Noetherian
Balmer spectrum 
has a {\em torsion model} $\T_\mrm{t}$, in the sense that there is an
equivalence of categories $\T \simeq \T_\mrm{t}$
where the objects of $\T_\mrm{t}$ are diagrams whose vertices are
determined by objects of the form $\Gamma_\p L_\p X$ together with
gluing data of an adelic nature. 
\end{thm*}

We now turn to motivating the construction of the torsion model. For this, we will need to assume that $\T$ admits an $\infty$-categorical enhancement $\C$.
This permits us to make two
essential constructions: 
firstly, we can take limits of diagrams of categories, and
secondly, we can consider categories of modules over highly
structured rings. The irreducible 1-dimensional case is described in
\cite{torsion1}. In that case there is a single generic point $\g$ and
potentially infinitely many closed points $\m$. The adelic model is based on
the diagram
$$\xymatrix{
   &L_\g\1\ar[d] \\
  \prod_\m \Lambda_\m \unit \ar[r] &  L_\g\prod_\m \Lambda_\m \unit
}$$
whose pullback is the unit object $\1$,
and the torsion model is based on the diagram
$$\xymatrix{
L_\g\1\ar[d] &\\
L_\g\prod_\m \Lambda_\m \unit \ar[r] &\Sigma\bigoplus_\m\Gamma_{\m}\1
}$$
which is obtained by taking the horizontal cofibre in the previous diagram.
Here we reconstruct $\1$ from the torsion data $\Gamma_\m\1$ at the closed
points and the generic point $L_\g\unit$. We note in passing 
that we view the objects as modules over completed rings, which makes the process more algebraic. 

The focus of the present paper is on the higher dimensional case where
 the combinatorics is more complicated and we need more sophisticated
 machinery to permit us to implement our reconstruction scheme. 
We explain the construction in the case when $\Spc(\T^\omega)$ is two-dimensional: this is sufficiently large to demonstrate some of the complexity in defining a torsion model in the general case, whilst being small enough to draw the required diagrams in a legible fashion. We refer the reader to \cref{sec:thetorsionmodel} for the precise constructions in the general case.

Let us consider the unit $\1$ of $\C$, and explain the necessary
gluing data required to reconstruct $\1$ from the pieces $\Gamma_\p
L_\p \1$. The  adelic model provides the first step towards this. For
simplicity,
we suppose that the Balmer spectrum has a unique generic point $\g$ (i.e., it is irreducible), and write $\p$ for arbitrary primes of dimension $1$, and $\m$ for primes of dimension $0$. The adelic cube $\1_\ad$ in this setting takes the form 
\[
\xymatrix@R=1em@C=0em{
& \prod_\p L_\p \Lambda_\p \1 \ar[dd] \ar[rr] & & L_\g \prod_\p L_\p \Lambda_\p \1 \ar[dd] & \\ & & \textcolor{OIred}{L_\g \1} \ar[ur] \ar[dd] \\
& \prod_\p L_\p \prod_\m \Lambda_\m \1 \ar[rr]|-\hole & & L_\g \prod_\p L_\p \prod_\m \Lambda_\m \1 & \\
\prod_\m \Lambda_\m \1 \ar[rr] \ar[ur] & & L_\g \prod_\m \Lambda_\m \1  \ar[ur] 
}
\]
and the limit of this cube is $\1$, so that the data of any object $X \in \C$ may be recovered from local and complete data. One notes that the adelic cube already contains the information from the top dimensional building block $L_\g\1$ which we indicate in red. In order to replace some data by information concentrated only at single one-dimensional primes, mirroring the objectwise decomposition above, we take cofibres towards the right to obtain the diagram
%\[
%\xymatrix@=1em{
%& L_\g \prod_\p L_\p \Lambda_\p \1 \ar[dd] \ar@[gB][rr] && \Sigma \Gammale{1} \Lge{1}\1 \ar[dd] \\
% L_\g \1 \ar[ur] \ar[dd] & & \\
% & L_\g \prod_\p L_\p \prod_\m \Lambda_\m \1 \ar@[gB][rr] && \Sigma \Gammale{1} \prod_\p L_\p \prod_\m \Lambda_\m \1 \\
% L_\g \prod_\m \Lambda_\m \1  \ar[ur] \ar@[gB][rr] && \Sigma \Gammale{1} \prod_\m \Lambda_\m \1 \ar[ur]
%}
%\]
%As $\Gammale{1}$ focuses on the part supported at 1 and below, when it is applied to objects with support at 1 and above, the result splits, see \cref{spltting} for a precise statement. As such, the above diagram simplifies to
\[
\xymatrix@R=1em@C=0em{
& L_\g \prod_\p L_\p \Lambda_\p \1 \ar[dd] \ar@[OIblue][rr] && \textcolor{OIred}{\Sigma \bigoplus_\p \Gamma_\p L_\p\1} \ar[dd] \\
 \textcolor{OIred}{L_\g \1} \ar[ur] \ar[dd] & & \\
 & L_\g \prod_\p L_\p \prod_\m \Lambda_\m \1 \ar@[OIblue][rr] && \Sigma \bigoplus_\p \Gamma_\p L_\p \1 \otimes \prod_\m \Lambda_\m \1 \\
 L_\g \prod_\m \Lambda_\m \1  \ar[ur] \ar@[OIblue][rr] && \Sigma \Gammale{1} \prod_\m \Lambda_\m \1 \ar[ur]
}
\]
in which the blue maps are the new maps obtained by taking cofibres. We note that to obtain the above forms for the cofibres, we appeal to some general splitting results which we prove in \cref{sec:splitting}.

Now we have the building blocks at dimensions 1 and 2, it remains to replace some data by the 0-dimensional building blocks. In order to do this, we take the cofibre of the map $\Sigma\Gammale{1}\prod_\m \Lambda_\m\1 \to \Sigma\bigoplus_\p \Gamma_\p L_\p \1 \otimes \prod\Lambda_\m\1$ to obtain the diagram
\[
\xymatrix@R=1em@C=0em{
& L_\g \prod_\p L_\p \Lambda_\p \1 \ar[dd] \ar@[OIblue][rr] &&  \textcolor{OIred}{\Sigma \bigoplus_\p \Gamma_\p L_\p\1} \ar[dd] \\
\textcolor{OIred}{L_\g \1} \ar[ur] \ar[dd] & & & & \\
 & L_\g \prod_\p L_\p \prod_\m \Lambda_\m \1 \ar@[OIblue][rr] && \Sigma \bigoplus_\p \Gamma_\p L_\p \1 \otimes \prod_\m \Lambda_\m \1 \ar@[OIblue][rr]  & & \textcolor{OIred}{\Sigma^2 \bigoplus_\m \Gamma_\m \1} \\
 L_\g \prod_\m \Lambda_\m \1  \ar[ur] \ar@[OIblue][rr] && \Sigma \Gammale{1} \prod_\m \Lambda_\m \1 \ar[ur] & &
}
\]
in which we have highlighted the building blocks with singleton
supports. Applying this recipe to any object $X$ of $\C$, we obtain a similar diagram $X_\tors$. Reversing this procedure, that is, taking fibres to the left and then taking the limit of the resulting diagram, one may recover any object in $\C$ from a diagram of the above form.

By extracting the salient features of the above reconstruction process, one is led to consider the category $\C_\mrm{t}$ whose objects $M$ are the diagrams
\[
\xymatrix@=1em{
& M(21^2) \ar@{-->}[dd] \ar[rr] && M(1^1) \ar@{-->}[dd] \\
 M(2^2) \ar@{-->}[ur] \ar@{-->}[dd] & & & & \\
 & M(210^2) \ar[rr] && M(10^1) \ar[rr] && M(0^0) \\
 M(20^2)  \ar@{-->}[ur] \ar[rr] && M(0^1) \ar@{-->}[ur]
}
\]
where:
\begin{enumerate}
\item $M(A^i)$ is a module over the adelic ring $\1_\ad(A)$;
\item $M(i^i)$ has support consisting only of primes of dimension $i$, so \[M(i^i) \simeq \bigoplus_{\dim(\p)=i} M(\p)\] where for all $\p$, $M(\p)$ is either 0 or has support precisely $\p$;
\item a dashed map $M(A^i) \dashrightarrow M(B^i)$ indicates an equivalence after extending scalars, that is, a $\1_\ad(B)$-module map $\1_\ad(B) \otimes_{\1_\ad(A)} M(A^i) \to M(B^i)$ which is moreover an equivalence;
\item a solid map $M((A \cup i)^i) \to M(A^{i-1})$ indicates a map after restricting scalars, that is, a $\1_\ad(A)$-module map $\mrm{res}^{A \cup i}_A M((A \cup i)^i) \to M(A^{i-1})$;
\item $M(0^1) \to M(10^1) \to M(0^0)$ is a cofibre sequence.
\end{enumerate}
We refer the reader to \cref{sec:thetorsionmodel} for the precise detailed construction of the category $\C_\mrm{t}$.

One sees that for any $X \in \C$, the diagram $X_\tors$ is an object of $\C_\mrm{t}$ by construction. Our main theorem says the converse is in fact true, so that diagrams of the above form correspond precisely to objects of $\C$.
\begin{thm*}[{\ref{torsionmodel}}]
Let $\C$ be a rigidly-compactly generated, symmetric monoidal, stable $\infty$-category whose Balmer spectrum is finite dimensional and Noetherian. The functor \[(-)_\tors\colon \C \xrightarrow{\sim} \C_\mrm{t}\] is an equivalence of $\infty$-categories.
\end{thm*}

We give some examples in \cref{sec:examples} which demonstrate that the above theorem may be viewed as a categorification of Cousin complex from algebra, and of monochromatic layers in chromatic stable homotopy theory.

In fact, we also prove variants of this theorem by introducing the
concept of \emph{assembly data}, which provides a compatible partition
of the primes of each dimension. This gives less extreme localization,
completion, and torsion functors by collecting together primes in the
same set of the partition, and then builds models based on these functors. 
There are two main motivating examples here. Firstly, one may consider
the coarsest partition, with all primes of the same dimension
considered together. This gives a very crude approach to modelling objects from $\C$, which splits up
objects according to the dimension filtration, so that the terms in
the diagram concern $\Lge{i}$, $\Gammale{i}$ and
$\Lambdale{i}$. Secondly, in the key example of rational torus
equivariant spectra, we can partition subgroups according to
their identity component. This adapts the approach taken in Greenlees--Shipley~\cite{GreenleesShipley18} to the general tensor-triangular setting, and provides a bridge between the results of~\cite{GreenleesShipley18} and the form of the adelic model in~\cite{adelicm}, see \cref{sec:assembly} and \cref{subsec:torusspec} for more details and discussion. The advantage of working over adelic rings in the torsion model is highlighted by the example of rational torus equivariant spectra: the adelic rings are well-understood, so that one may use the torsion model as a basis for building an \emph{algebraic} model. 

\addtocontents{toc}{\SkipTocEntry}
\subsection*{Other categorical decompositions}
The idea of splitting a category into smaller pieces is well established in mathematics, and in recent years, there have been several offerings in this direction. This paper generalises the one-dimensional case established by the authors in~\cite{torsion1}, although the strategy here also gives a new perspective and approach to the one-dimensional version. Most closely related to the torsion model, are the adelic models in the sense of~\cite{adelicm} and their one-dimensional variants in~\cite{adelic1}. However, there are other abstract reconstruction theorems appearing in~\cite{AMGR, prismatic}.

\addtocontents{toc}{\SkipTocEntry}
\subsection*{Structure of the paper}
We begin in \cref{sec:background} by firstly recalling the necessary background from tensor-triangular geometry, before introducing the notion of assembly data. In \cref{sec:cubes} we introduce categories of cubes and punctured cubes, which lays the groundwork for establishing the adelic model in \cref{sec:adelic}. We note that this is not just a recapitulation of~\cite{adelicm}; we transport the argument into the setting of $\infty$-categories and also show that the model applies relative to any assembly data. In \cref{sec:cofibres,catsofcofs,sec:punctured} we construct various cofibre functors between categories of cubes, which provides the framework for passing from the adelic to the torsion model. 
\cref{sec:global} iterates the process of taking cofibres, before we prove the main theorem of the paper in \cref{sec:thetorsionmodel}. Finally, in \cref{sec:examples} we illuminate the theory with some explicit examples in algebra, stable homotopy theory, and rational equivariant stable homotopy theory.

\addtocontents{toc}{\SkipTocEntry}
\subsection*{Conventions}
Throughout the paper we work in the setting of $\infty$-categories. We will frequently refer to $\infty$-categories instead as categories, unless we wish to make emphasis.

\addtocontents{toc}{\SkipTocEntry}
\subsection*{Acknowledgements}
The first named author was supported by the European Research Council (ERC) under Horizon Europe Grant No.~101042990. The second author is grateful for support under EPSRC Grant number EP/P031080/2. The third author is supported by the SFB 1085 Higher Invariants in Regensburg. The fourth named author is supported by the project PRIMUS/23/SCI/006 from Charles University, and by Charles University Research Centre
program No. UNCE/24/SCI/022. 

The authors would like to thank the Hausdorff Research Institute for Mathematics for the hospitality in the context of the Trimester program ``Spectral Methods in Algebra, Geometry, and Topology'', funded by the Deutsche Forschungsgemeinschaft (DFG, German Research Foundation) under Germany’s Excellence Strategy – EXC-2047/1 – 390685813. We are grateful to the Mathematisches Forschungsinstitut Oberwolfach for its support and hospitality during a Research in Pairs stay in 2022, which made for an engaging mathematical environment during the duration of the stay. The authors would like to thank the Isaac Newton Institute for Mathematical
Sciences for the support and hospitality during the programme `Topology, representation theory and higher structures' when work on this paper was undertaken. This was supported by EPSRC Grant Number EP/R014604/1.

\section{Localization functors in tensor-triangular geometry}\label{sec:background}

In this section, we begin by recalling the necessary formalism of tensor-triangular geometry in the sense of \cite{Balmer05}, and then prove various important preliminary results which we will require throughout the course of the paper. 

In what follows we will always assume that $\C$ is a rigidly-compactly generated, stable, symmetric monoidal $\infty$-category such that $\tensor$ commutes with colimits separately in each variable. We write $\1$ for the tensor unit, $\uHom$ for the internal hom functor of $\C$, and $\C^\omega$ for the full subcategory spanned by the compact objects. Under these assumptions, the homotopy category $h\C$ is a rigidly-compactly generated tensor-triangulated category. We will sometimes abuse notation and write $\C$ when we intend to work with the tensor-triangulated homotopy category.

Recall that a full subcategory of $\C$ or $\C^\omega$ is \emph{thick} if it is stable and closed under retracts. A full subcategory of $\C$ is said to be \emph{localizing} if it is a thick subcategory which is moreover closed under filtered colimits in $\C$. A thick subcategory of $\C^\omega$ is a \emph{thick $\otimes$-ideal} if it closed under tensor products with arbitrary objects of $\C^\omega$, and a localizing subcategory of $\C$ is moreover a \emph{localizing $\otimes$-ideal} if it is closed under tensor products with arbitrary objects of $\C$. A thick $\otimes$-ideal $\p$ of $\C^\omega$ is said to be \emph{prime} if whenever $X \otimes Y \in \p$, then either $X$ or $Y$ is in $\p$. 

Associated to any such $\C$ is a topological space $\Spc(\C^\omega)$ called the Balmer spectrum~\cite{Balmer05}. This is a spectral space whose points are the prime thick $\otimes$-ideals of $\C^\omega$. The closed subsets of this space are generated by
\[
\supp(X) := \{\fp \in \Spc( \C^\omega) \mid X \not\in \fp \}
\]
as $X$ runs through all objects of $\C^\omega$. We will see how to extend this support theory to arbitrary objects of $\C$ later on in this section.

We can now introduce the main hypothesis that we will require of $\C$.

\begin{hyp}\label{hyp:hyp}
Henceforth, $\C = (\C , \tensor, \1)$ will be a rigidly-compactly generated, stable, symmetric monoidal $\infty$-category such that $\tensor$ commutes with colimits separately in each variable, and such that $\Spc( \C^\omega)$ is a finite-dimensional Noetherian topological space.
\end{hyp}

The Noetherian hypothesis on the Balmer spectrum affords us some benefits as we now recall, after giving a preliminary definition.
For $\p,\q \in \Spc(\C^\omega)$, we say that $\q$ is a \emph{specialization of $\p$} if $\q \in \overline{\{\p\}}$, where $\overline{\{\p\}}$ denotes the closure of $\{\p\}$ in $\Spc(\C^\omega)$. A subset $S \subseteq \Spc(\C^\omega)$ is \emph{specialization closed} if $\p \in S$ implies that $\overline{\{\p\}} \subseteq S$. We consider the poset structure on the Balmer spectrum given by the specialization ordering, that is $\q \leqslant \p$ if and only if $\q$ is a specialization of $\p$. In particular, closed points are minimal with respect to this choice.

\begin{thm}\label{balmerres}
Let $\C$ be as in \cref{hyp:hyp} and $\mathcal{I}$ be a thick $\otimes$-ideal of  $\C$. Then the assignment 
\[
\mathcal{I} \mapsto \supp(\mathcal{I}):=\bigcup_{X\in \mathcal{I}}\supp(X)
\]
defines a bijection between the set of thick $\otimes$-ideals of $\C^\omega$  and the set of specialization closed subsets of $\Spc( \C^\omega)$.
\end{thm}
\begin{proof}
This is a special case of \cite[Theorem 4.10]{Balmer05}, using~\cite[Remark 4.11]{Balmer05} and~\cite[Proposition 2.4]{BalmerFiltrations} to obtain the above form.
\end{proof}

With the above definitions and result, we can now record an important consequence of the Noetherian assumption.

\begin{lem}[{\cite[Corollary 7.14]{BalmerFavi11}}]\label{lem:cons}
Let $\C$ be as in \cref{hyp:hyp}. For every $\fp \in \Spc(\C^\omega)$, there is a compact object $K_\fp \in \C^\omega$ such that $\supp(K_\fp) = \overline{\{\fp\}}$. Such an object is called a \emph{Koszul object for $\fp$}.
\end{lem}

We note that such Koszul objects for a fixed prime $\p$ are not unique, but by \cref{balmerres} any two generate the same thick $\otimes$-ideal.

%We shall say that $\C$ is \emph{finite dimensional Noetherian} if $\Spc( \C^\omega)$ is a finite-dimensional Noetherian space. We will write $\dim(\C)$ for the Krull dimension of this topological space. In the case that $\Spc( \C^\omega)$ is indeed a Noetherian space, one can see that the topology is determined uniquely by the specialization order.\todo{I wonder if we have been too brief here, and if it is worth being a little more rigorous, especially with how we have ordered things also do we mention Thomason subsets etc}

\subsection{Torsion, localization, and completion}

We will now move towards introducing torsion, localization, and completion functors on $\C$ with respect to a Balmer prime (or family thereof). Much of the theory that we discuss here works in a far larger generality, but we have chosen to extract just the features that we need using the Noetherian hypothesis. 

Given a specialization closed subset $V \subseteq \Spc( \C^\omega)$, we let $\cK(V) = \{K_\p \mid \p \in V\}.$ We write $\Gamma_V\C$ for the localizing $\otimes$-ideal of $\C$ generated by $\cK(V)$. 
We then define $L_{V^c} \C = (\Gamma_V\C)^\perp$ and $\Lambda_V \C = (L_{V^c}\C)^\perp$, where \[\mc{X}^\perp := \{Y \in \C \mid \Hom(X,Y) \simeq 0 \text{ for all $X \in \mc{X}$}\}.\] There are corresponding inclusion functors
\[
\iota_{\Gamma}\colon \Gamma_V \C \hookrightarrow \C, \quad
\iota_{L} \colon L_{V^c}\C \hookrightarrow \C, \quad
\iota_{\Lambda} \colon \Lambda_V \C \hookrightarrow \C.
\]
The functors $\iota_\Gamma$ and $\iota_L$ have right adjoints denoted by $\Gamma_V$ and $\Delta_{V^c}$ respectively, and the functors $\iota_L$ and $\iota_\Lambda$ have left adjoints denoted $L_{V^c}$ and $\Lambda_V$ respectively.  We will often write $\Gamma_V$ for the composite $\iota_\Gamma\Gamma_V$ and similarly for the other functors. There are natural cofibre sequences
	 \[\Gamma_V X \to X \to L_{V^c}X \qquad \text{and} \qquad
    \Delta_{V^c}X \to X \to \Lambda_V X\]
for all $X \in \C$. We refer to $\Gamma_V$, $L_{V^c}$, and $\Lambda_V$ as torsion, localization, and completion respectively.

We now record some key properties of these functors.
\begin{prop}\label{prop:localduality}
Let $\C$ be as in \cref{hyp:hyp} and $V$ be a specialization closed subset of $\Spc(\C^\omega)$.
\begin{enumerate}[label=(\arabic*)]
\item \label{item:smashing} The localization $L_{V^c}$ is smashing, and the colocalization $\Gamma_V$ is smashing.
\item \label{item:mgm} There are natural equivalences of endofunctors on $\C$
\[\Lambda_V \Gamma_V \xrightarrow{\sim} \Lambda_V \quad \mrm{and} \quad \Gamma_V \xrightarrow{\sim} \Gamma_V \Lambda_V\]
which we shall call the \emph{MGM equivalence}. As such $\Gamma_V\C$ and $\Lambda_V\C$ are equivalent $\infty$-categories.
\item When viewed as endofunctors on $\C$ via the inclusions, the functors $(\Gamma_V, \Lambda_V)$ form an adjoint pair in that there is a natural equivalence
    \[\uHom(\Gamma_V X, Y) \simeq \uHom(X, \Lambda_V Y)\]
    for all $X,Y \in \C$. In particular we have $\uHom(\Gamma_V \unit, Y) \simeq \Lambda_V Y$.
    \item \label{item:hptybicart} For every $X \in \C$  there is a pullback square
    \[
    \xymatrix{X \ar[r] \ar[d] & L_{V^c}X \ar[d] \\ \Lambda_V X \ar[r] & L_{V^c} \Lambda_V X \rlap{.}}
    \]
    whose vertical and horizontal fibres are $\Delta_{V^c}X$ and $\Gamma_V X$ respectively.
\end{enumerate}
\end{prop}
\begin{proof}
This may be found in~\cite[Theorem 3.3.5]{HoveyPalmieriStrickland97} and~\cite[Theorem 2.21]{BarthelHeardValenzuela18}; also see~\cite{Greenlees01b}. 
\end{proof}

We will be interested in several particular specialization closed subsets of the Balmer spectrum. As such we give these distinguished notation as follows.
\begin{nota}\label{Lpnotation}\leavevmode
\begin{enumerate}[label=(\alph*)]
\item\label{nota1} For each prime $\fp \in \Spc(\C^\omega)$ we may consider the specialization closed subset $ \downcl(\p)=\{\q \mid \q \subseteq \p\} $. We note that $\downcl(\p) = \overline{\{\p\}}$. We write $\Gamma_\p$ and $\Lambda_\p$ for the torsion and completion functors associated to this specialization closed subset.
\item\label{nota2} For each prime $\fp \in \Spc(\C^\omega)$, consider the specialization closed subset $\upcl(\p)^c$ where $\upcl(\p)=\{\q \mid \p \subseteq \q\}$. We write $L_\p$ for the associated localization functor.
\item\label{nota3} When $V = \{\p \mid \mrm{dim}(\p) \leqslant n\}$ we write $\Gammale{n}$, $\Lge{n+1}$ and $\Lambdale{n}$ for the associated torsion, localization, and completion functors.
\end{enumerate}
\end{nota} 

We refer the reader to~\cite[\S 5]{torsion1} for several examples of these functors in concrete terms. Here, we justify the notation in parts \ref{nota1} and \ref{nota2} by noting that when $\C = \msf{D}(R)$, the derived category of a commutative Noetherian ring $R$, the functors $\Gamma_\p$, $L_\p$, and $\Lambda_\p$ are the familiar derived torsion, localization, and derived completion functors. The notation in \ref{nota3} is justified by the behaviour of supports, as we will see below.

We emphasize that our notational choice means
\[\Gamma_\p X \to X \to L_\p X\] is generally \emph{not} a cofibre sequence,
since  the support of $L_\p X$ is the up-closure of $\p$ (for example, the
prime $\p$ will usually lie in the support of both $\Gamma_\p X$ and $L_\p X$).

Using the functors recalled above we can define a well-behaved notion 
of support for objects $X \in \C$. Indeed, picking a specialization closed subset $V$ of the Balmer spectrum, one should see $\Gamma_V$ as detecting those objects supported at $V$, and $L_{V^c}$ as detecting objects supported away from $V$. This allows us to isolate objects supported exactly at a single prime ideal. In particular, we recall the following definition
\[
\supp(X)=\{\p \in \Spc(\C^\omega) \mid L_\p\Gamma_\p X \not \simeq 0\}.
\]
We refer the reader to~\cite[Proposition 4.5]{torsion1} for a list of properties that 
this support function satisfies. Here we recall a few key properties which we will use repeatedly:
\begin{enumerate}[label=(\alph*)]
\item the two notions of support we have defined agree on compact objects;
\item $\supp(\Gammale{n}\1)$ (resp., $\supp(\Lge{n+1}\1)$) consists precisely of 
the set of primes of dimension less than or equal to $n$ (resp., greater than or equal to $n+1$);
\item we have the following \emph{detection property}: an object $X \in \C$ is zero if and only if $\Gamma_\p L_\p X \simeq 0$ for all $\p \in \Spc(\C^\omega)$, if and only if $\mrm{supp}(X) = \varnothing$.
\item we have $\supp(X \otimes Y) \subseteq \supp(X) \cap \supp(Y)$ for all $ X,Y\in\C$.
\end{enumerate}

\begin{defn}
Let $\p \in \Spc(\C^\omega)$. We say that $X \in \C$ is:
\begin{enumerate}[label=(\alph*)]
\item \emph{$\p$-torsion} if the canonical map $\Gamma_\p X \to X$ is an equivalence;
\item \emph{$\p$-local} if the canonical map $X \to L_\p X$ is an equivalence;
\item \emph{$\p$-complete} if the canonical map $X \to \Lambda_\p X$ is an equivalence.
\end{enumerate}
In a similar way, one defines $V$-torsion and $V$-complete objects for arbitrary specialization closed subsets $V$ of $\Spc(\C^\omega)$.
\end{defn}

% 
%\begin{rem}\todo{Make a definition, also make this explicit}
%  Let $\Gamma_V$ be a torsion functor in $\C$, for some specialization closed subset $V$ of $\Spc(\C^\omega)$. 
%  We will say that $X \in \C$ is $V$-torsion (or $\Gamma_V$-torsion) if the canonical 
%  map $ \Gamma_V X \to X$ is an equivalence. In the special case when $V = \downcl(\p)$, we will say that $X$ is $\p$-torsion. Similarly, we say that 
%  $X$ is $V$-complete (resp., $V$-local) if the canonical map $X \to \Lambda_V X$ (resp., $X \to L_{V}X$) is 
%  an equivalence. 
%\end{rem}

\begin{lem}\label{lem:torsioniscell}
Let $V$ be a specialization closed subset of $\Spc(\C^\omega)$ and $f\colon X \to Y$ be a map in $\C$. Then $\Gamma_V(f)$ is an equivalence if and only if $K_\p \otimes f$ is an equivalence for all $\p\in V$.
\end{lem}

\begin{proof}
 Since $\Gamma_V\1 \in \mrm{Loc}^\otimes(K_\p \mid \p \in V)$, the latter condition implies the former. The converse holds as $K_\p \otimes f \simeq \Gamma_V(K_\p) \otimes f \simeq K_\p \otimes \Gamma_V(f)$.
\end{proof}

The following proposition explains how the torsion functor behaves with respect to infinite products. 

\begin{prop}\label{prop:Gammaproducts}
Let $V$ be a specialization closed subset of $\Spc(\C^\omega)$ and let $\Gamma_V$ denote the associated torsion functor. For any set $\{X_i\}$ of objects in $\C$, the natural map \[\Gamma_V \prod_i \Gamma_V X_i \xrightarrow[\quad]{\sim} \Gamma_V \prod_i X_i\] is an equivalence.
\end{prop}
\begin{proof}
Note that there is a canonical map $f\colon \prod_i \Gamma_V X_i  \to \prod_i X_i$. By \cref{lem:torsioniscell}, it suffices to show that the map $K_\p \otimes f\colon K_\p \otimes \prod_i \Gamma_V X_i \to K_\p \otimes \prod_i X_i$ is an equivalence for all $\p\in V$. This now follows from the observation that as $K_\p$ is compact and hence dualizable, $K_\p \otimes -$ commutes with products, together with the fact that $K_\p \otimes \Gamma_V X_i\simeq K_\p \otimes X_i$ for all $\p \in V$.
\end{proof}

\subsection{Splitting properties}\label{sec:splitting}
In this section we prove some splitting properties of the torsion and localization functors, which we will use throughout the paper. For a subset $P$ of $\Spc(\C^\omega)$, we write $\mrm{max}(P)$ and $\mrm{min}(P)$ for the set of maximal (resp., minimal) elements of $P$ under the specialization ordering. 

We first prove a splitting property of $\Gamma_V$ and then turn to the splitting of $L_V$ afterwards.
\begin{prop}\label{splitting} Let $X \in \C$ and $V$ be a specialization closed subset of $\Spc(\C^\omega)$.
If $\supp(X) \cap V \subseteq \max(V)$, then the canonical map
\[\bigoplus_{\p \in \mrm{max}(V)}\Gamma_\p X  \xrightarrow[\quad]{\sim} \Gamma_VX   \]
is an equivalence.
\end{prop}
\begin{proof}
By \cref{lem:torsioniscell} it suffices to show:
\begin{enumerate}
\item $\bigoplus_{\p \in \mrm{max}(V)}\Gamma_\p X$ is $V$-torsion, and,
\item the natural map $\beta_{\q} \colon \bigoplus_{\p \in \mrm{max}(V)} K_\q \otimes\Gamma_\p X \to K_\q\otimes X$ is an equivalence for all $\q \in V$.
\end{enumerate}
The first follows as $\mrm{supp}(\Gamma_\p X) = \downcl(\p) \cap \mrm{supp}(X) \subseteq V$ since $V$ is specialization closed. For the second condition, we argue via support. If $\q \not\in \max(V)$, then the source and target of $\beta_\q$ are both 0 as $\downcl(\q) \cap \supp(X) = \varnothing$. If $\q \in \max(V)$, then $K_\q \otimes \Gamma_\p X \simeq 0$ for all $\p \in \mrm{max}(V)$ with $\p \neq \q$ by the assumption on the support of $X$, and hence \[\oplus_{\p \in \max(V)} K_\q \otimes \Gamma_\p X \simeq K_\q \otimes \Gamma_\q X \simeq K_\q \otimes X\] as required.
\end{proof}

In order to prove an analogous splitting property for $L_{V^c}$, we require the following two lemmas.
\begin{lem}\label{lemma1}
Let $X \in \C$ and $V$ be a specialization closed subset of $\Spc(\C^\omega)$. Write $W = \downcl(\min(V^c))$ for the specialization closure of $\min(V^c)$. Then \[\Gamma_W\prod_{\p \in \min(V^c)} L_\p X \simeq \bigoplus_{\p \in \min(V^c)} \Gamma_\p L_\p X.\]
\end{lem}
\begin{proof}
Note that $\max(W) = \min(V^c)$ so $\supp(\prod_{\p \in \min(V^c)}L_\p X) \cap W \subseteq \mrm{max}(W)$. Therefore
\[\Gamma_W \prod_{\p \in \min(V^c)}L_\p X \simeq \bigoplus_{\q \in \min(V^c)} \Gamma_\q \prod_{\p \in \min(V^c)} \Gamma_\q L_\p X\] by \cref{splitting} and \cref{prop:Gammaproducts}. The statement follows since $\Gamma_\q L_\p \simeq 0$ if $\p \neq \q$ for all $\p,\q \in \min(V^c)$.
\end{proof}

\begin{lem}\label{lemma2}
Let $V$ be a specialization closed subset of $\Spc(\C^\omega)$ and write $W = \downcl(\min(V^c))$. Let $Y \in \C$ be such that $\supp(Y) \subseteq W$. Then the natural map
\[\Hom(Y, \bigoplus_{\p \in \min(V^c)} L_\p X) \to \Hom(Y, \prod_{\p \in \min(V^c)} L_\p X)\]
is an equivalence for all $X \in \C$.
\end{lem}
\begin{proof}
By the assumption on the support of $Y$, we see that $\Hom(Y, -) \simeq \Hom(Y, \Gamma_W(-))$. Now \[\Gamma_W \bigoplus_{\p \in \min(V^c)} L_\p X \simeq \bigoplus_{\p \in \min(V^c)} \Gamma_\p L_\p X \simeq \Gamma_W \prod_{\p \in \min(V^c)}L_\p X\] where the first equivalence follows from the observation that $\Gamma_W L_\p \simeq \Gamma_\p L_\p$ since $\max(W) = \min(V^c)$, while the second equivalence follows from \cref{lemma1}. As such, the claim follows.
\end{proof}

With the previous two lemmas, we can now prove the promised splitting of localization.
\begin{prop}\label{splitL}
Let $X \in \C$ and $V$ be a specialization closed subset of $\Spc(\C^\omega)$. If $\supp(X) \cap V^c \subseteq \min(V^c)$, then \[L_{V^c}X \xrightarrow[\quad]{\sim} \bigoplus_{\p \in \mrm{min}(V^c)}L_\p X.\]
\end{prop}
\begin{proof}
The canonical map $L_{V^c} X \to \prod_{\p \in \min(V^c)} L_\p X$ factors through the direct sum by \cref{lemma2}. We denote this factoring by $s\colon L_{V^c} X \to \oplus_{\p \in \min(V^c)} L_\p X$. By the detection property it suffices to show that $\Gamma_\q L_\q (s)$ is an equivalence for all $\q \in \Spc(\C^\omega)$. Note that \[\supp(L_{V^c} X) \subseteq \min(V^c) \quad \text{and} \quad \supp(\bigoplus_{\p \in \min(V^c)} L_\p X) \subseteq \min(V^c) \] by the assumption on the support of $X$. Therefore we may reduce to checking that $\Gamma_\q L_\q (s)$ is an equivalence for $\q \in \min(V^c)$. Now this follows by inspection.
\end{proof}

\subsection{Assembly data}\label{sec:assembly}
The specialization poset of the Balmer spectrum provides a natural filtration on objects. However, there are some occasions, such as  in the key example of rational torus-equivariant spectra explored in \cref{subsec:torusspec}, where one may want to index over a different poset. In this section, we will make this precise by introducing assembly data.
Recall that we write $\leqslant$ for the natural poset structure on the Balmer spectrum induced by the specialization order.

\begin{defn}\label{posetrefinement}
Let $\C$ be as in \cref{hyp:hyp}. \emph{Assembly data} for $\C$ consists of the data of a subposet $\X \subseteq \Spc(\C^\omega)$ and a poset map $\alpha\colon \Spc(\C^\omega) \to \X$ such that:
\begin{enumerate}[label=(\alph*)]
\item $\X$ is a retract of $\Spc(\C^\omega)$ via $\alpha$, i.e., $\alpha(x) = x$ for all $x \in \X$;
\item $\alpha$ preserves the dimension of elements, where the dimension of $x$ is defined to be the length of a maximal chain of elements less than $x$.
\end{enumerate}
\end{defn}

\begin{rem}
We note that it follows that the inclusion $\X \hookrightarrow \Spc(\C^\omega)$ also preserves dimension.
\end{rem}

We write $\downcl$ for the down closure in $\Spc( \C^\omega)$ with respect to the specialization ordering $\leqslant$, and write $\downcl_\X$ for the down closure in $\X$, and similarly for up cones. We note that for any specialization closed subset $V \subseteq \X$, we have that $\alpha^{-1}(V)$ is specialization closed in $\Spc(\C^\omega)$ since $\alpha$ is a poset map. As such, we may make the following definition.
\begin{defn}\label{assembledfunctors}
For any specialization closed subset $V$ of $\X$ we define functors \[\Gamma^\X_V := \Gamma_{\alpha^{-1}V} \qquad \Lambda^\X_V := \Lambda_{\alpha^{-1} V} \qquad L^\X_{V^c} := L_{\alpha^{-1} V^c}.\]
As in \cref{Lpnotation}, for every $x \in \X$ we  write \[\Gamma^\X_x := \Gamma_{\downcl_\X(x)}^\X \qquad \Lambda^\X_x := \Lambda^\X_{\downcl_\X(x)} \qquad L^\X_x := L_{\upcl_\X(x)}^\X.\]
\end{defn}

Let us record some examples to illuminate the above definitions.
\begin{ex}\label{assemblyexamples} \leavevmode
\begin{enumerate}[label=(\alph*)]
\item One may always take $\X = \Spc(\C^\omega)$ and $\alpha$ to be the identity. In this case, the functors defined in \cref{assembledfunctors} agree with those of \cref{Lpnotation}. One may think of this as the finest assembly data on $\C$. \label{finestassembly}
\item If $\Spc(\C^\omega)$ has dimension $d$, then one may always view $[d] = \{0 < 1 < \cdots < d\}$ as a subposet of $\Spc(\C^\omega)$ by specifying a distinguished Balmer prime of each dimension. The dimension function $\mrm{dim}\colon \Spc(\C^\omega) \to [d]$ provides assembly data, and for any $i \in [d]$, we have \[\Gamma_i^\X = \Gammale{i} \qquad \Lambda^\X_i = \Lambdale{i} \qquad L^\X_i = \Lge{i}.\] One may think of this as the coarsest assembly data on $\C$.
\item \label{rationalassembly} When $\C = L_\Q \mathrm{Sp}_{\mathbb{T}^r}$ is the category of rational torus equivariant-spectra for a torus of rank $r$, the Balmer spectrum is given by $\mrm{Sub}(\mathbb{T}^r)$ with cotoral inclusions (i.e., $H \leqslant K$ if and only if $K/H$ is a torus)~\cite{Greenlees19}. We may take $\X$ to be the poset of connected subgroups of $G$, and $\alpha = \mrm{conn}$ to be the function taking the connected component of the identity. We refer the reader to \cref{subsec:torusspec} for further information on this example. \cref{fig:figure1,fig:figure2,fig:figure3} provide a graphical description of this poset and the aforementioned different choices of assembly data in the case of $r=2$.  The shaded triangles in \cref{fig:figure1} represent the fact that there are infinitely many subgroups of dimension 1 associated to each connected one-dimensional subgroup (i.e., to each subtori).
\end{enumerate}
\end{ex}

\begin{figure}[h]
\begin{minipage}[t]{0.29\linewidth}
\resizebox{\textwidth}{!}{
\begin{tikzpicture}[yscale=2]
%Initialize nodes
\node[circle,draw=black, fill=black, inner sep=0pt, minimum size=5pt] (0) at (0,0) {};
\node[circle,draw=black, fill=black, inner sep=0pt, minimum size=5pt] (10) at (-3,-1.5) {};
\node[circle,draw=black, fill=black, inner sep=0pt, minimum size=5pt] (11) at (-1.5,-1.5) {};
\node[circle,draw=black, fill=black, inner sep=0pt, minimum size=5pt] (12) at (0,-1.5) {};
\node[circle,draw=black, fill=black, inner sep=0pt, minimum size=5pt] (13) at (1.5,-1.5) {};
\node (15) at (3,-1.5) {$\cdots$};
\node[circle,draw=black, fill=black, inner sep=0pt, minimum size=5pt] (20) at (-3,-3) {};
\node[circle,draw=black, fill=black, inner sep=0pt, minimum size=5pt] (21) at (-1.5,-3) {};
\node[circle,draw=black, fill=black, inner sep=0pt, minimum size=5pt] (22) at (0,-3) {};
\node[circle,draw=black, fill=black, inner sep=0pt, minimum size=5pt] (23) at (1.5,-3) {};
\node (25) at (3,-3) {$\cdots$};

%Draw second edges -- these are mostly just random
\draw[->, shorten >=0.1cm, line join=round,decorate, decoration={zigzag, segment length=6, amplitude=.9,post=lineto, post length=2pt}] (10) -- (20);
\draw[->, shorten >=0.1cm, line join=round,decorate, decoration={zigzag, segment length=6, amplitude=.9,post=lineto, post length=2pt}] (11) -- (20);
\draw[->, shorten >=0.1cm, line join=round,decorate, decoration={zigzag, segment length=6, amplitude=.9,post=lineto, post length=2pt}] (12) -- (20);
\draw[->, shorten >=0.1cm, line join=round,decorate, decoration={zigzag, segment length=6, amplitude=.9,post=lineto, post length=2pt}] (13) -- (20);
\draw[->, shorten >=0.1cm, line join=round,decorate, decoration={zigzag, segment length=6, amplitude=.9,post=lineto, post length=2pt}] (15) -- (20);

\draw[->, shorten >=0.1cm, line join=round,decorate, decoration={zigzag, segment length=6, amplitude=.9,post=lineto, post length=2pt}] (10) -- (21);
\draw[->, shorten >=0.1cm, line join=round,decorate, decoration={zigzag, segment length=6, amplitude=.9,post=lineto, post length=2pt}] (10) -- (23);
\draw[->, shorten >=0.1cm, line join=round,decorate, decoration={zigzag, segment length=6, amplitude=.9,post=lineto, post length=2pt}] (10) -- (25);

\draw[->, shorten >=0.1cm, line join=round,decorate, decoration={zigzag, segment length=6, amplitude=.9,post=lineto, post length=2pt}] (11) -- (21);
\draw[->, shorten >=0.1cm, line join=round,decorate, decoration={zigzag, segment length=6, amplitude=.9,post=lineto, post length=2pt}] (11) -- (22);
\draw[->, shorten >=0.1cm, line join=round,decorate, decoration={zigzag, segment length=6, amplitude=.9,post=lineto, post length=2pt}] (11) -- (25);

\draw[->, shorten >=0.1cm, line join=round,decorate, decoration={zigzag, segment length=6, amplitude=.9,post=lineto, post length=2pt}] (12) -- (21);
\draw[->, shorten >=0.1cm, line join=round,decorate, decoration={zigzag, segment length=6, amplitude=.9,post=lineto, post length=2pt}] (12) -- (25);

\draw[->, shorten >=0.1cm, line join=round,decorate, decoration={zigzag, segment length=6, amplitude=.9,post=lineto, post length=2pt}] (13) -- (25);
\draw[->, shorten >=0.1cm, line join=round,decorate, decoration={zigzag, segment length=6, amplitude=.9,post=lineto, post length=2pt}] (13) -- (23);

\draw[->, shorten >=0.1cm, line join=round,decorate, decoration={zigzag, segment length=6, amplitude=.9,post=lineto, post length=2pt}] (15) -- (25);
\draw[->, shorten >=0.1cm, line join=round,decorate, decoration={zigzag, segment length=6, amplitude=.9,post=lineto, post length=2pt}] (15) -- (23);
\draw[->, shorten >=0.1cm, line join=round,decorate, decoration={zigzag, segment length=6, amplitude=.9,post=lineto, post length=2pt}] (15) -- (22);
\draw[->, shorten >=0.1cm, line join=round,decorate, decoration={zigzag, segment length=6, amplitude=.9,post=lineto, post length=2pt}] (15) -- (21);

%First load of white buffer
\draw[white, fill=white] (-3,-1.5) circle (.2cm);
\draw[white, fill=white] (-1.5,-1.5) circle (.2cm);
\draw[white, fill=white] (0,-1.5) circle (.2cm);
\draw[white, fill=white] (1.5,-1.5) circle (.2cm);

\draw[fill = black!10, draw=black!40] (-3,-1.5) -- (-3.35,-1.25) -- (-2.65,-1.25) -- cycle;
\draw[fill = black!10, draw=black!40] (-1.5,-1.5) -- (-1.85,-1.25) -- (-1.15,-1.25) -- cycle;
\draw[fill = black!10, draw=black!40] (0,-1.5) -- (-0.35,-1.25) -- (0.35,-1.25) -- cycle;
\draw[fill = black!10, draw=black!40] (1.5,-1.5) -- (1.85,-1.25) -- (1.15,-1.25) -- cycle;

\node[circle,draw=black, fill=black, inner sep=0pt, minimum size=5pt] at (-3,-1.5) {};
\node[circle,draw=black, fill=black, inner sep=0pt, minimum size=5pt] at (-1.5,-1.5) {};
\node[circle,draw=black, fill=black, inner sep=0pt, minimum size=5pt] at (0,-1.5) {};
\node[circle,draw=black, fill=black, inner sep=0pt, minimum size=5pt] at (1.5,-1.5) {};

%%Draw first edges
%%(-3,-1.5)
\draw[->, shorten >=0.1cm, line join=round,decorate, decoration={zigzag, segment length=6, amplitude=.9,post=lineto, post length=2pt}] (0) -- (-3,-1.25);
\draw[->, shorten >=0.1cm, line join=round,decorate, decoration={zigzag, segment length=6, amplitude=.9,post=lineto, post length=2pt}] (0) -- (-1.5,-1.25);
\draw[->, shorten >=0.1cm, line join=round,decorate, decoration={zigzag, segment length=6, amplitude=.9,post=lineto, post length=2pt}] (0) -- (0,-1.25);
\draw[->, shorten >=0.1cm, line join=round,decorate, decoration={zigzag, segment length=6, amplitude=.9,post=lineto, post length=2pt}] (0) -- (1.5,-1.25);
\draw[->, shorten >=0.1cm, line join=round,decorate, decoration={zigzag, segment length=6, amplitude=.9,post=lineto, post length=2pt}] (0) -- (15);

%%old edges
%\draw[->, shorten >=0.1cm, line join=round,decorate, decoration={zigzag, segment length=6, amplitude=.9,post=lineto, post length=2pt}] (0) -- (10);
%\draw[->, shorten >=0.1cm, line join=round,decorate, decoration={zigzag, segment length=6, amplitude=.9,post=lineto, post length=2pt}] (0) -- (11);
%\draw[->, shorten >=0.1cm, line join=round,decorate, decoration={zigzag, segment length=6, amplitude=.9,post=lineto, post length=2pt}] (0) -- (12);
%\draw[->, shorten >=0.1cm, line join=round,decorate, decoration={zigzag, segment length=6, amplitude=.9,post=lineto, post length=2pt}] (0) -- (13);
%\draw[->, shorten >=0.1cm, line join=round,decorate, decoration={zigzag, segment length=6, amplitude=.9,post=lineto, post length=2pt}] (0) -- (15);

%Adding in white buffer around start nodes (hack needed because of the wavy lines)
\draw[white, fill=white] (0,0) circle (.2cm);
\node[circle,draw=black, fill=black, inner sep=0pt, minimum size=5pt] at (0,0) {};
\draw (-3.5,0.25) -- (3.5,0.25) -- (3.5,-3.25) -- (-3.5,-3.25) -- cycle;

\end{tikzpicture}

}
\caption{The poset arising from the Balmer spectrum of $L_\Q \mathsf{Sp}_{\mathbb{T}^2}$.\\}
\label{fig:figure1}
\end{minipage}%
\hfill
\begin{minipage}[t]{0.29\linewidth}

\resizebox{\textwidth}{!}{
\begin{tikzpicture}[yscale=2]
%Initialize nodes
\node[circle,draw=black, fill=black, inner sep=0pt, minimum size=5pt] (0) at (0,0) {};
\node[circle,draw=black, fill=black, inner sep=0pt, minimum size=5pt] (10) at (-3,-1.5) {};
\node[circle,draw=black, fill=black, inner sep=0pt, minimum size=5pt] (11) at (-1.5,-1.5) {};
\node[circle,draw=black, fill=black, inner sep=0pt, minimum size=5pt] (12) at (0,-1.5) {};
\node[circle,draw=black, fill=black, inner sep=0pt, minimum size=5pt] (13) at (1.5,-1.5) {};
\node (15) at (3,-1.5) {$\cdots$};
\node[circle,draw=black, fill=black, inner sep=0pt, minimum size=5pt] (22) at (0,-3) {};

%Draw second edges -- these are mostly just random
\draw[->, shorten >=0.1cm, line join=round,decorate, decoration={zigzag, segment length=6, amplitude=.9,post=lineto, post length=2pt}] (10) -- (22);
\draw[->, shorten >=0.1cm, line join=round,decorate, decoration={zigzag, segment length=6, amplitude=.9,post=lineto, post length=2pt}] (11) -- (22);
\draw[->, shorten >=0.1cm, line join=round,decorate, decoration={zigzag, segment length=6, amplitude=.9,post=lineto, post length=2pt}] (12) -- (22);
\draw[->, shorten >=0.1cm, line join=round,decorate, decoration={zigzag, segment length=6, amplitude=.9,post=lineto, post length=2pt}] (13) -- (22);
\draw[->, shorten >=0.1cm, line join=round,decorate, decoration={zigzag, segment length=6, amplitude=.9,post=lineto, post length=2pt}] (15) -- (22);

%First load of white buffer
\draw[white, fill=white] (-3,-1.5) circle (.2cm);
\draw[white, fill=white] (-1.5,-1.5) circle (.2cm);
\draw[white, fill=white] (0,-1.5) circle (.2cm);
\draw[white, fill=white] (1.5,-1.5) circle (.2cm);
\node[circle,draw=black, fill=black, inner sep=0pt, minimum size=5pt] at (-3,-1.5) {};
\node[circle,draw=black, fill=black, inner sep=0pt, minimum size=5pt] at (-1.5,-1.5) {};
\node[circle,draw=black, fill=black, inner sep=0pt, minimum size=5pt] at (0,-1.5) {};
\node[circle,draw=black, fill=black, inner sep=0pt, minimum size=5pt] at (1.5,-1.5) {};

%%Draw first edges
\draw[->, shorten >=0.1cm, line join=round,decorate, decoration={zigzag, segment length=6, amplitude=.9,post=lineto, post length=2pt}] (0) -- (10);
\draw[->, shorten >=0.1cm, line join=round,decorate, decoration={zigzag, segment length=6, amplitude=.9,post=lineto, post length=2pt}] (0) -- (11);
\draw[->, shorten >=0.1cm, line join=round,decorate, decoration={zigzag, segment length=6, amplitude=.9,post=lineto, post length=2pt}] (0) -- (12);
\draw[->, shorten >=0.1cm, line join=round,decorate, decoration={zigzag, segment length=6, amplitude=.9,post=lineto, post length=2pt}] (0) -- (13);
\draw[->, shorten >=0.1cm, line join=round,decorate, decoration={zigzag, segment length=6, amplitude=.9,post=lineto, post length=2pt}] (0) -- (15);

%Adding in white buffer around start nodes (hack needed because of the wavy lines)
\draw[white, fill=white] (0,0) circle (.2cm);
\node[circle,draw=black, fill=black, inner sep=0pt, minimum size=5pt] at (0,0) {};
\draw (-3.5,0.25) -- (3.5,0.25) -- (3.5,-3.25) -- (-3.5,-3.25) -- cycle;
\end{tikzpicture}
}
\caption{The assembly data given by applying the function $\mrm{conn}$ to \cref{fig:figure1}.}
\label{fig:figure2}
\end{minipage}%
\hfill
\begin{minipage}[t]{0.28\linewidth}
\vspace{-41.1mm}
\resizebox{\textwidth}{!}{
\begin{tikzpicture}[yscale=2]
%Initialize nodes
\node[circle,draw=black, fill=black, inner sep=0pt, minimum size=5pt] (0) at (0,0) {};
\node[circle,draw=black, fill=black, inner sep=0pt, minimum size=5pt] (12) at (0,-1.5) {};
\node[circle,draw=black, fill=black, inner sep=0pt, minimum size=5pt] (22) at (0,-3) {};

%Draw second edges -- these are mostly just random
\draw[->, shorten >=0.1cm, line join=round,decorate, decoration={zigzag, segment length=6, amplitude=.9,post=lineto, post length=2pt}] (12) -- (22);

%First load of white buffer
\draw[white, fill=white] (0,-1.5) circle (.2cm);
\node[circle,draw=black, fill=black, inner sep=0pt, minimum size=5pt] at (0,-1.5) {};

%%Draw first edges
\draw[->, shorten >=0.1cm, line join=round,decorate, decoration={zigzag, segment length=6, amplitude=.9,post=lineto, post length=2pt}] (0) -- (12);

%Adding in white buffer around start nodes (hack needed because of the wavy lines)
\draw[white, fill=white] (0,0) circle (.2cm);
\node[circle,draw=black, fill=black, inner sep=0pt, minimum size=5pt] at (0,0) {};
\draw (-3.5,0.25) -- (3.5,0.25) -- (3.5,-3.25) -- (-3.5,-3.25) -- cycle;
\end{tikzpicture}
}
\caption{The coarsest assembly data for $L_\Q \mathsf{Sp}_{\mathbb{T}^2}$.\\ \\}
\label{fig:figure3}
\end{minipage}%
\end{figure}

Note that for any specialization closed subset $V$ of $\X$, and $X \in \C$ we have \[\supp(\Gamma_V^\X X) = \alpha^{-1}V \cap \supp(X) \quad \text{and} \quad \supp(L_{V^c}^\X X) = \alpha^{-1}V^c \cap \supp(X).\] 
This demonstrates how the assembly data controls the new localization functors.

We end this subsection by recording a couple of properties the localization functors based on assembly data satisfy.
\begin{lem}\label{Pmonoidal}
Let $x \in \X$. The endofunctors $L^\X_x$ and $\Lambda^\X_x$ are lax symmetric monoidal, and as such they preserve commutative algebra objects.
\end{lem}
\begin{proof}
This is a special case of a more general result, namely that for any specialization closed subset $V$ of $\mrm{Spc}(\C^\omega)$, the endofunctors $L_{V^c}$ and $\Lambda_V$ are lax symmetric monoidal. In order to prove this, by~\cite[Lemma 3.4]{GepnerGrothNikolaus15} (also see~\cite[Proposition 2.2.1.9]{HA}) it suffices to show that if $X \to Y$ is an $L_{V^c}$-local (resp., $\Lambda_V$-local) equivalence, then $X \otimes Z \to Y \otimes Z$ is an $L_{V^c}$-local (resp., $\Lambda_V$-local) equivalence for all $Z \in \C$. This holds for $L_{V^c}$ since it is a smashing localization. For $\Lambda_V$, we note that a map $f\colon X \to Y$ in $\C$ is such that $\Lambda_V(f)$ is an equivalence if and only if $K(\p) \otimes f$ is an equivalence for all $\p \in \C$; this follows from \cref{lem:torsioniscell} together with the MGM equivalence, or alternatively, see for example~\cite[Theorem 4.2]{GreenleesMay95b}. The claim then follows from this.
\end{proof}

We set the following convention which we will use throughout the paper.

\begin{nota}\label{dimsubscript}
Unless otherwise stated, we use subscripts to denote the dimension of elements, so that $x_n$ denotes an arbitrary element of $\X$ of dimension $n$, and $\p_i$ denotes a prime in $\Spc(\C^\omega)$ of dimension $i$.
\end{nota}

\begin{lem}\label{maxL}
Let $(\X,\alpha)$ be assembly data for $\C$ and $X \in \C$.
\begin{enumerate}[label=(\alph*)]
\item\label{LXsplit} Let $x \in \X$. If $\supp(X) \cap \alpha^{-1}(\upcl_\X(x)) \subseteq \alpha^{-1}(x)$, then \[L^\X_{x}X \simeq \bigoplus_{\alpha(\p) = x} L_{\p}X.\]
\item\label{maxLsplit} When $d = \mrm{dim}(\C)$, we have $L_{\geqslant d}X \simeq \bigoplus_{x_d} L_{x_d}^\X X.$
\end{enumerate}
\end{lem}
\begin{proof}
For part \ref{LXsplit}, firstly note that $\alpha^{-1}(\upcl_\X(x)^c)^c = \alpha^{-1}(\upcl_\X(x))$ and the minimal elements in here are precisely those $\p$ such that $\alpha(\p) = x$ as $\alpha$ is order-preserving, and dimension-preserving. By definition $L^\X_xX = L_{\alpha^{-1}(\upcl_\X(x))}X$, and so by \cref{splitL}, we know that this is equivalent to \[\bigoplus_{\p \in \alpha^{-1}(x)} L_\p X\] as required. For part \ref{maxLsplit}, we have \[\bigoplus_{x_d} L_{x_d}^\X X \simeq \bigoplus_{x_d} \bigoplus_{\alpha(\p_d) = x_d} L_{\p_d}X \simeq \bigoplus_{\p_d} L_{\p_d} X \simeq \Lge{d}X\] using part \ref{LXsplit} and \cref{splitL}.
\end{proof}

\subsection{Objects with mono-dimensional support}
In this subsection we define the key players in this paper, objects with mono-dimensional support.
\begin{defn}
An object $X\in \C$ has \emph{mono-dimensional support $i$} if and only if $\supp(X)$ consists only of primes of dimension $i$, or equivalently, if $X$ is both $\Gammale{i}$-torsion and $\Lge{i}$-local.
\end{defn}

\begin{ex}
For $0 \leqslant i \leqslant d$, we write $\Elr{i} = \Gammale{i}\Lge{i}\1$. By construction this has mono-dimensional support $i$, and given any other object $X$ with mono-dimensional support $i$, we have $X\simeq X \otimes\Elr{i}$. For this reason the object $\Elr{i}$ will play a crucial role throughout this paper. We note that there are cofibre sequences $\Gammale{i-1}\1 \to \Gammale{i}\1 \to \Elr{i}$ for each $i$. These objects with mono-dimensional support admit a splitting as shown in the following.
\end{ex}

\begin{lem}\label{lem:Epointy}
Let $0 \leqslant i \leqslant d$. We have $\Elr{i} \simeq \bigoplus_{\p_i}\Gamma_{\p_i} L_{\p_i}\1$.
\end{lem}
\begin{proof}
This follows from~\cref{splitting} and \cref{splitL}.
\end{proof}

\section{Cubes and punctured cubes}\label{sec:cubes}
The adelic and torsion model are built out of categories whose objects are cubes in which the vertices do not live in the same category. In this section, we introduce the necessary formalism to deal with such situations.
\subsection{Constructing cubes}
Let $\C$ be a presentably symmetric monoidal stable $\infty$-category. Throughout the paper we use diagrams indexed on the power set $\mc{P}([i])$ of $[i] = \{0,1,\cdots,i\}$, which can be viewed as $(i+1)$-cubes. We set the convention that $[-1] = \varnothing$. 
\begin{nota}\label{dj}
Fix $j \in [i]$. We write $d_j\colon \mc{P}([i{-}1]) \to \mc{P}([i])$ for the inclusion of the $i$-face indexed by $j$ into the whole $(i+1)$-cube, and $d_{\backslash j}\colon \mc{P}([i{-}1]) \to \mc{P}([i])$ for the inclusion of the $i$-face whose entries do not contain $j$ into the whole cube.
\end{nota}

Fix a diagram of rings $R\colon \mc{P}([i]) \to \msf{CAlg}(\C)$. 
\begin{nota}\label{cubefaces}
For every $0\leqslant j \leqslant i$ we denote by $R_j$ the $i$-cube of rings spanned by all elements containing $j$, and by $\Rwithout{j}$ the $i$-cube of rings spanned by all elements not containing $j$. Observe that every $i$-face of $R$ is of one of this types. 
\end{nota}

\begin{nota}
For $A \subseteq B \in \mathcal{P}([i])$, there is a ring map $R(A) \to R(B)$ and hence a restriction map $\mrm{res}^B_A\colon\mod{R(B)} \to \mod{R(A)}$, which has a left adjoint $\mrm{ext}_A^B\colon \mod{R(A)} \to \mod{R(B)}$ given by extension of scalars. 
\end{nota}  

We now construct a category $\cubeL(R)$, which informally speaking, has as objects diagrams of modules $M$ over the diagram of rings $R$, consisting of an $R(A)$-module $M(A)$ for each $A \in \mc{P}([i])$, together with structure maps $M(A) \to \mrm{res}^B_AM(B)$ for any inclusion $A \subseteq B$ in $\mc{P}([i])$. The following definition makes this precise $\infty$-categorically, and the particular form of definition we give will allow us to make inductive proofs later on.
\begin{defn}\label{defn:cubeL}
	Let $R\colon \mc{P}([i]) \to \msf{CAlg}(\C)$ be a diagram of rings. We inductively define a category $\cubeL(R)$ together with a functor $\pi_{i}\colon \cubeL(R) \to \mrm{Fun}(\mc{P}([i]),\mod{R(\varnothing)})$. We note that the index $i$ comes from the diagram of rings $R$. When $i=-1$ we set $\cubeL(R) = \mod{R(\varnothing)}$ and $\pi_{-1} = \mrm{id}$. For $i \geqslant 0$ we then define $\cubeL(R)$ via the pullback 
	\[ 
	\begin{tikzcd}[column sep=0.5cm]
		\cubeL(R) \ar[r, "\pi_{i}"] \ar[d, "{(\rho_{i},(\rho_j, \rho_{\backslash j})_j)}"'] & \mrm{Fun}(\mc{P}([i]), \mod{R(\varnothing)}) \ar[d, "{(d_{i}^*,(d_j^*, d_{\backslash j}^*)_j)}"] \\
		\cubeL(R_{i}) \times \prod_{j=0}^{i-1}\left(\cubeL(R_j) \times \cubeL(\Rwithout{j})\right) \ar[r] & \substack{\mrm{Fun}(\mc{P}([i{-}1]), \mod{R(\varnothing)}) \\ \times \\ \prod_{j=0}^{i-1}\mrm{Fun}(\mc{P}([i{-}1]), \mod{R(\varnothing)})^{\times 2}}
	\end{tikzcd}
	\]
	where $d_j$ and $d_{\backslash j}$ are as in \cref{dj}.
	The bottom horizontal map of the above square is given by the product of the following two maps
	\[
	\cubeL(R_{i}) \xrightarrow{(\mrm{res}^i_{\varnothing})_* \circ \pi_{i-1}} \mrm{Fun}(\mc{P}([i{-}1]), \mod{R(\varnothing)})
	\]
	and 
	\[ 
	\prod_{j=0}^{i-1}\left(\cubeL(R_j) \times \cubeL(\Rwithout{j})\right) \xrightarrow{ \prod_j ((\mrm{res}^j_{\varnothing})_* \circ \pi_{i-1}) \times \pi_{i-1}} 
	\prod_{j=0}^{i-1}\mrm{Fun}(\mc{P}([i{-}1]), \mod{R(\varnothing)})^{\times 2}.
	\]
	Since $\varnothing$ is initial, there is always a restriction functor $\mrm{res}^A_\varnothing$, and the map $\pi_{i}$ applies these restriction functors to push the whole diagram back into $R(\varnothing)$-modules.
\end{defn}

\begin{rem}
Since we will make several similar constructions in the following sections, let us briefly elaborate on the idea behind the previous definition. The top right vertex of the pullback determines the shape of the objects of $\cubeL(R)$, namely it forces them to be $(i+1)$-cubes. The bottom left vertex determines the extra structure which these objects have, by restricting to faces of the cube; by induction, this gives objects the claimed module structures. The bottom right vertex contains the information about how to glue the shape of the objects together with the extra structure they have. The following examples will further clarify this idea.
\end{rem}

\begin{ex}
	If $i=0$, the data of an object in $\cubeL(R)$ is the data of
        an $R(\varnothing)$-module $M(\varnothing)$, an $R(0)$-module
        $M(0)$, and a map of $R(\varnothing)$-modules $M(\varnothing)
        \to \mrm{res}^0_\varnothing M(0)$. In this case $\pi_0\colon \cubeL(R) \to \mrm{Fun}(\mc{P}([0]), \mod{R(\varnothing)})$ sends the above data to $M(\varnothing) \to \mrm{res}^0_\varnothing M(0)$. 
\end{ex}

\begin{ex}\label{ex-dim1-cubeL}
	If $i=1$, the data of an object in $\cubeL(R)$ is the data of modules $M(\varnothing), M(0), M(1)$ and $M(10)$ over the rings indicated by the subset together with module maps $f_{\varnothing}^1 \colon M(\varnothing) \to \mrm{res}^1_{\varnothing}M(1)$, $f_{0}^{10} \colon M(0) \to \mrm{res}^{10}_{0}M(10)$, $f_{1}^{10} \colon M(1) \to \mrm{res}^{10}_{0}M(10)$ and $f_{\varnothing}^0 \colon M(\varnothing) \to \mrm{res}^0_{\varnothing}M(0)$ making the diagram
	\[
	\begin{tikzcd}[column sep=2cm]
		M(\varnothing) \arrow[r, "f_{\varnothing}^1"] \arrow[d,"f_{\varnothing}^{0}"'] & \mrm{res}^1_{\varnothing} M(1) \arrow[d, "\mrm{res}^1_{\varnothing}f_1^{10}"]\\
		\mrm{res}^0_{\varnothing} M(0) \arrow[r, "\mrm{res}_{\varnothing}^0 f_0^{10}"'] & \mrm{res}^0_{\varnothing}\mrm{res}^{10}_0 M(10) \simeq  \mrm{res}_{\varnothing}^1 \mrm{res}_1^{10}M(10)
	\end{tikzcd}
	\]
	commute.
	We note that the map $\rho_1\colon \cubeL(R)\to\cubeL(R_1)$ records only the right vertical arrow in the above square. 
\end{ex}

\begin{rem}\label{rem:dataofcubeL}
 An object $M$ of $\cubeL(R)$ can be described as the data of:
 \begin{enumerate}[label=(\roman*)]
 	\item for each $A \in \mathcal{P}([i])$ an $R(A)$-module $M(A)$;
	\item for each inclusion $A \subseteq B$ in $\mathcal{P}([i])$ a map $M(A) \to \mathrm{res}_{A}^{B} M(B)$ of $R(A)$-modules.
 \end{enumerate}
These objects and maps assemble into an $(i+1)$-cube which is required to commute.
\end{rem}

\begin{prop}\label{restricttofaces}
	Let $R\colon \mc{P}([i]) \to \msf{CAlg}(\C)$ be a diagram of rings, and let $R_\sigma$ be a face of $R$ (of arbitrary dimension). Then there exists a functor $\rho_\sigma\colon\cubeL(R) \to \cubeL(R_\sigma)$ which restricts to the face corresponding to $R_\sigma$, that is, given $M \in \cubeL(R)$ as in \cref{rem:dataofcubeL}, $\rho_\sigma(M)$ forgets the part of $M$ indexed on $A \not\in \sigma$.
\end{prop}
\begin{proof}
	It suffices to restrict to the case when $R_\sigma$ is an $i$-face, so that $R_\sigma = R_j$ or $R_{\backslash j}$. In order to do this, we prove by induction that given any $R\colon \mc{P}([i]) \to \msf{CAlg}(\C)$, that there exists restriction functors $\rho_j\colon \cubeL(R) \to \cubeL(R_j)$ and $\rho_{\backslash j}\colon \cubeL(R) \to \cubeL(R_{\backslash j})$ for all $0 \leqslant j \leqslant i$. 
	Note that we have $\rho_j$ for all $0 \leqslant j \leqslant i$ and $\rho_{\backslash j}$ for all $0 \leqslant j \leqslant i-1$ by \cref{defn:cubeL}, so it suffices to deal with the case $\cubeL(R) \to \cubeL(R_{\backslash i})$.
	
	For the base case $i=0$, we define $\rho_{\backslash 0}$ to be the composite \[\cubeL(R) \xrightarrow{\pi_0} \mrm{Fun}(\mc{P}([0]), \mod{R(\varnothing)}) \xrightarrow{d_{\backslash 0}^*} \mod{R(\varnothing)} = \cubeL(R_{\backslash 0}).\]

	So suppose the claim holds for $i=k-1$, and fix an $R\colon \mc{P}([k]) \to \msf{CAlg}(\C)$. In order to construct the restriction map $\rho_{\backslash k}\colon \cubeL(R) \to \cubeL(\Rwithout{k})$, by definition of $\cubeL(R_{\backslash k})$ as a pullback, it suffices to construct compatible functors:
	\begin{enumerate}
		\item $\cubeL(R) \to \mrm{Fun}(\mc{P}([k{-}1]), \mod{R_{\backslash k}(\varnothing)})$,
		\item $\cubeL(R) \to \cubeL((R_{\backslash k})_{j})$ for $0 \leqslant j \leqslant k-1$,
		\item $\cubeL(R) \to \cubeL((R_{\backslash k})_{\backslash j})$ for all $0 \leqslant j \leqslant k-2$.
	\end{enumerate} 
	For (1), we note that $R_{\backslash k}(\varnothing) = R(\varnothing)$. We have a functor \[\pi_{k}\colon \cubeL(R) \to \mrm{Fun}(\mc{P}([k]), \mod{R(\varnothing)})\] and we postcompose this with the restriction along the inclusion $\mc{P}([k{-}1]) \to \mc{P}([k])$ which takes the face avoiding $k$. For (2), we take the composite \[\cubeL(R) \xrightarrow{\rho_j} \cubeL(R_j) \xrightarrow{\rho_{\backslash k}} \cubeL((\Rwithout{k})_j)\] where the latter map exists by the inductive hypothesis. For (3), we take the composite \[\cubeL(R) \xrightarrow{\rho_{\backslash j}} \cubeL(\Rwithout{j}) \xrightarrow{\rho_{\backslash k}} \cubeL((\Rwithout{k})_{\backslash j})\] where again the last map exists by the inductive hypothesis.
\end{proof}

We now turn to defining punctured cubes. We write $\mc{P}([i])_{\neq\varnothing}$ for the full subcategory of $\mc{P}([i])$ consisting of non-empty subsets of $[i]$. We note that this can be pictured as a punctured $(i+1)$-cube.
\begin{defn}\label{Cequals} 
Let $R\colon \mc{P}([i]) \to \msf{CAlg}(\C)$ be a diagram of rings. We define a category $\equals{R}$. When $i=0$ we set $\equals{R} = \mod{R(0)}$. For $i \geqslant 1$ we define $\equals{R}$ as the pullback
\[\begin{tikzcd}[column sep=2cm]
\equals{R} \arrow[r, "\pi_i^\mrm{Pc}"] \arrow[d, "(\tau_j)_j"'] & \mrm{Fun}(\mc{P}([i])_{\neq\varnothing}, \mod{R(\varnothing)}) \arrow[d, "(d_j^*)_j"]\\
\prod_{j=0}^i \cubeL(R_j) \arrow[r,"\prod(\mrm{res}_{\varnothing}^j)_* \circ \pi_{i-1}"'] & \prod_{j=0}^i  \mrm{Fun}(\mc{P}([i{-}1]), \mod{R(\varnothing)}).
\end{tikzcd}\]
\end{defn}

\begin{ex}
For $i=1$, an object in $\equals{R}$ can be represented by a punctured $2$-cube which is obtained from the square in \Cref{ex-dim1-cubeL} by removing the module $M(\varnothing)$.
\end{ex}

\begin{rem}\label{rem:dataofpc}
In an analogous fashion to \cref{rem:dataofcubeL}, an object $M$ of $\equals{R}$ can be described as the data of:
 \begin{enumerate}[label=(\roman*)]
 	\item for each $A \in \mathcal{P}([i])_{\neq \varnothing}$ an $R(A)$-module $M(A)$;
	\item\label{structuremaps} for each inclusion $A \subseteq B$ in $\mathcal{P}([i])_{\neq \varnothing}$ a map $M(A) \to \mathrm{res}_{A}^{B} M(B)$ of $R(A)$-modules.
 \end{enumerate}
The data can be assembled into a punctured $(i+1)$-cube whose faces are required to commute. We note that by adjunction, the structure maps in \ref{structuremaps} give maps $\mrm{ext}_A^B M(A) \to M(B)$ for each $A \subseteq B$ in $\mc{P}([i])_{\neq\varnothing}$, which we refer to as \emph{adjoint structure maps}.
\end{rem}

\subsection{Adjunctions} We now discuss various adjunctions between the underlying category and the category of cubes.
\begin{lem}\label{leftadjointoncubes}
Let $R\colon \mc{P}([i]) \to \msf{CAlg}(\C)$ be a diagram of rings. There is a functor $R\otimes_{R(\varnothing)} -\colon \mod{R(\varnothing)} \to \cubeL(R)$ given by $(R \otimes X)(A) = R(A) \otimes_{R(\varnothing)} X$. Moreover, this has a right adjoint defined by the composite \[\cubeL(R) \xrightarrow{\pi_i} \mrm{Fun}(\mc{P}([i]), \mod{R(\varnothing)}) \xrightarrow{\mrm{lim}} \mod{R(\varnothing)}.\] In other words, the right adjoint applies restriction of scalars to the cube so that each vertex is an $R(\varnothing)$-module and then takes the limit of this cube. 
\end{lem}
\begin{proof}
To prove the existence of the functor $R \otimes_{R(\varnothing)} -$ we argue by induction on $i$. For the base case $i=-1$, $\cubeL(R) = \mod{R(\varnothing)}$ and hence the claim the trivial. So suppose the claim is true for $i=k-1$. Now fix $R\colon \mc{P}([k]) \to \msf{CAlg}(\C)$. Since $\cubeL(R)$ is defined as a pullback, it suffices to give compatible maps:
\begin{enumerate}
\item $\mod{R(\varnothing)} \to \mrm{Fun}(\mc{P}([k]), \mod{R(\varnothing)})$
\item $\mod{R(\varnothing)} \to \cubeL(R_k) \times \prod_{j=0}^{k-1}(\cubeL(R_j) \times \cubeL(\Rwithout{j}))$. 
\end{enumerate}
For the latter, on each component we take the functors which exist by inductive hypothesis.
For the former we take the functor adjoint to the functor 
\[\mc{P}([k]) \times \mod{R(\varnothing)} \xrightarrow{\mathrm{res}_{\varnothing}(R) \times \mrm{id}} \mod{R(\varnothing)} \times \mod{R(\varnothing)} \xrightarrow{-\otimes_{R(\varnothing)}-} \mod{R(\varnothing)}\] 
where $\mathrm{res}_{\varnothing}(R)$ is the functor that sends any subset $A$ to $\mrm{res}^A_{\varnothing}R(A)$. In order to construct $\mathrm{res}_{\varnothing}(R)$ we first consider the functor $\mod{R}\colon \mc{P}([k])^{\mathrm{op}} \to \Cat_\infty$ sending a subset $A$ to $\mod{R(A)}$ and any inclusion $A \subseteq B$ to the restriction of scalars functor $\mathrm{res}^B_A \colon \mod{R(B)}\to \mod{R(A)}$. The functor $\mathrm{res}_{\varnothing}(R)$ is obtained from the functor $\mod{R}$ by applying \cite[Theorem 7.15]{GLP2024}. This concludes the proof of the existence of the functor $R \otimes_{R(\varnothing)} -$.

For the existence of the right adjoint, we again argue by induction on $i$. When $i=-1$, $R \otimes_{R(\varnothing)} -$ is the identity and so has a right adjoint. So suppose the claim holds for $i=k-1$, and fix a diagram of rings $R\colon \mc{P}([k]) \to \msf{CAlg}(\C)$. In order to construct a right adjoint to $R \otimes_{R(\varnothing)} -$, by \cite[Theorem 5.5]{rightadjoints} it suffices to construct right adjoints on each of the vertices of the defining pullback of $\cubeL(R)$. In the bottom left vertex we have right adjoints by the induction hypothesis, and in the top right vertex we have an adjoint given by taking the limit. Therefore $R\otimes_{R(\varnothing)} -\colon \mod{R(\varnothing)} \to \cubeL(R)$ has a right adjoint, and it is of the claimed form by \cite[Theorem 5.5]{rightadjoints}.
\end{proof}

%\begin{lem}\label{rightadjointoncubes}
%Let $R\colon \mc{P}([i]) \to \msf{CAlg}(\C)$ be a diagram of rings. The functor $R \otimes_{R(\varnothing)} -\colon \mod{R(\varnothing)} \to \cubeL(R)$ of \cref{leftadjointoncubes} has a right adjoint defined by the composite \[\cubeL(R) \xrightarrow{\pi_i} \mrm{Fun}(\mc{P}([i]), \mod{R(\varnothing)}) \xrightarrow{\mrm{lim}} \mod{R(\varnothing)}.\] In other words, the right adjoint applies restriction of scalars to the cube so that each vertex is an $R(\varnothing)$-module and then takes the limit of this cube. 
%\end{lem}
%\begin{proof}
%We again argue by induction on $i$. When $i=-1$, the limit functor is the identity. So suppose the claim holds for $i=k-1$, and fix a diagram of rings $R\colon \mc{P}([k]) \to \msf{CAlg}(\C)$. In order to construct a right adjoint to $R \otimes_{R(\varnothing)} -$, by \cite[Theorem 5.5]{rightadjoints} it suffices to construct right adjoints on each of the vertices of the defining pullback of $\cubeL(R)$. In the bottom left vertex we have right adjoints by the induction hypothesis, and in the top right vertex we have an adjoint given by taking the limit. Therefore $R\otimes_{R(\varnothing)} -\colon \C \to \cubeL(R)$ \jtodo{Inconsistent but probably correct if we do indeed need $R(\varnothing) = \1$} has a right adjoint, and it is of the claimed form by \cite[Theorem 5.5]{rightadjoints}.
%\end{proof}

So far, since the diagrams in $\cubeL(R)$ have an initial object, the limit just picks out this initial object. So now we move to the punctured setting, and show that the previous two results hold true in this setting also.
\begin{prop}\label{puncturedcubeadjunction}
Let $R\colon \mc{P}([i]) \to \msf{CAlg}(\C)$ be a diagram of rings. There is a functor \[R \otimes_{R(\varnothing)} -\colon \mod{R(\varnothing)} \to \equals{R}\] which has a right adjoint given by the composite \[\equals{R} \xrightarrow{\pi_i^\mrm{Pc}} \mrm{Fun}(\mc{P}([i])_{\neq\varnothing}, \mod{R(\varnothing)}) \xrightarrow{\mrm{lim}} \mod{R(\varnothing)}\] where $\pi_i^\mrm{Pc}$ is as in \cref{Cequals}.  
\end{prop}
\begin{proof}
In order to construct the functor $R \otimes_{R(\varnothing)} -\colon \mod{R(\varnothing)} \to \equals{R}$, it suffices to give compatible functors $\mod{R(\varnothing)} \to \mrm{Fun}(\mc{P}([i])_{\neq\varnothing}, \mod{R(\varnothing)})$ and $\mod{R(\varnothing)} \to \prod_{j=0}^i \cubeL(R_j)$ by the defining pullback. For the former we take the functor \[\mod{R(\varnothing)} \to \mrm{Fun}(\mc{P}([i]), \mod{R(\varnothing)}) \to \mrm{Fun}(\mc{P}([i])_{\neq\varnothing}, \mod{R(\varnothing)})\] where the first functor is as in the proof of \cref{leftadjointoncubes}, and the second is restriction along the evident inclusion. We note that this functor has a right adjoint given by taking the limit. For the functor $\mod{R(\varnothing)} \to \prod_{j=0}^i \cubeL(R_j)$ we take the functor constructed in \cref{leftadjointoncubes} on each component, which has a right adjoint. As such we conclude by applying \cite[Theorem 5.5]{rightadjoints}.
\end{proof}

\section{The adelic model}\label{sec:adelic}
In this section, we will prove the existence of the adelic model for a tensor-triangulated category $\C$, based on assembly data $(\X, \alpha)$. In particular, this generalises the adelic model of~\cite{adelicm}, thus bridging a gap between the model of \cite{adelicm}  and the version of~\cite{GreenleesShipley18}, as we will discuss in \cref{subsec:torusspec}.
Some of the arguments in this section follow similar ideas to those of~\cite{adelicm}. However, that paper works model categorically, so here we note that the argument applies equally well in the $\infty$-categorical setting. The adelic model constructed here provides the first step towards building the torsion model as we will see in the following sections. 

\subsection{Constructing the model}
Henceforth we assume that $\C$ is as in \cref{hyp:hyp}, and we fix assembly data $(\X, \alpha)$ for $\C$. We write $d$ for the dimension of the Balmer spectrum $\Spc(\C^\omega)$. The reader may find it helpful to keep in mind the `finest' assembly data of \cref{assemblyexamples}\ref{finestassembly}.

\begin{defn}\label{def:adeliccube}
Let $A = \{i_0{<}i_1{<}\ldots{<}i_{n-1}{<} i_n\}$ be a non-empty subset of $\{0,1, \ldots, d\}$. We define
\[\1_\ad^\X(A) = \prod_{\substack{x_n \in \X \\ \dim(x_n) = i_n}} L_{x_n}^\X\prod_{\substack{x_{n-1} \in \X \\ \dim(x_{n-1}) = i_{n-1}}} L_{x_{n-1}}^\X \cdots \prod_{\substack{x_0 \in \X \\ \dim(x_0) = i_0}} L_{x_0}^\X\Lambda_{x_0}^\X\1.\]
The resulting object is a commutative algebra object in $\C$ by \cref{Pmonoidal}. 
\end{defn}

%\begin{lem}
%Given any $A \subseteq B$ in $I_\ad(d)$ we have a ring map $\1_\ad^\X(A) \to \1_\ad^\X(B)$.
%\end{lem}
%\begin{proof}
%It suffices to prove the case $B = A \cup i$. If $i > \max(A)$ then $\1_\ad^\X(B) = \prod_{x_i} L_{x_i}^\X \1_\ad(A)$, so the claim follows by applying the natural transformation $\mrm{id} \Rightarrow \prod_{x_i} L_{x_i}^\X$ to $\1_\ad(A)$. This is a map of rings as the localizations are lax symmetric monoidal by \cref{Pmonoidal}. Otherwise, we may write $A$ as a disjoint union $A_{<i} \sqcup A_{>i}$ where $A_{<i}$ is the subset of $A$ consisting of the elements which are smaller than $i$, and $A_{>i}$ is defined similarly. We have a ring map $\theta\colon\1_\ad^\X(A_{< i}) \to \1_\ad^\X(A_{<i} \cup i)$ by the previous case. We may write $A_{>i} = \{j_0 < j_1 < \cdots < j_n\}$ and applying the functor \[\prod_{\substack{x_n \in \X \\ \dim(x_n) = j_n}} L_{x_n}^\X \prod_{\substack{x_{n-1} \in \X \\ \dim(x_{n-1}) = j_{n-1}}} L_{x_{n-1}}^\X \cdots \prod_{\substack{x_0 \in \X \\ \dim(x_0) = j_0}} L_{x_0}^\X\] to $\theta$ gives the required map, as the functor is lax symmetric monoidal by \cref{Pmonoidal}.
%\end{proof}
%\todo{The next lemma implies the previous one. Probably we are secretly using it to prove it, but we need to make this clearer}

\begin{rem}\label{differentcubes}
We note that there are two differences between the adelic cube defined
in \cref{def:adeliccube} and the one given in \cite{adelicm}. Firstly,
here we consider an arbitrary choice of assembly data on $\C$, while in \cite{adelicm} only the finest assembly data is considered. 
More significantly, we correct a mistake
in \cite{adelicm}: in \cref{def:adeliccube} the products are taken
over all primes, this rectifies an error in \cite{adelicm}, where the
products are indexed by inclusion chains of primes (for example,
one may check that restricting to inclusion chains is incorrect in
dimension 1 if the Balmer spectrum is reducible). The proof given in
\cite{adelicm} applies without change when the indexing is corrected,
but it will be repeated here. 
\end{rem}

\begin{lem}\label{1adfunctor}
The assignment $A \mapsto \1_\ad^\X(A)$ (and $\varnothing \mapsto \1$) defines a functor \[\1_\ad^\X\colon \mc{P}([d]) \to \msf{CAlg}(\C)\] which we call the \emph{adelic cube}, and therefore by restriction also a functor $\1_\ad^\X\colon \mc{P}([d])_{\neq\varnothing} \to \msf{CAlg}(\C)$ which we call the \emph{punctured adelic cube}.
\end{lem}
\begin{proof}
We construct this by induction on $d$. For each $x_i$ there is a natural transformation $\mrm{id} \Rightarrow L_{x_i}^\X$, and these assemble into a natural transformation $\eta_i\colon\mrm{id} \Rightarrow \prod_{x_i} L_{x_i}^\X$. 
When $d=1$ we obtain the functor via applying the natural transformation $\eta_1$ to the map $\1 \to \prod_{x_0}\Lambda_{x_0}^\X\1$. Now for $d\geqslant 1$, by induction we may suppose we have constructed the $d$-cube $C$, and we obtain the $(d+1)$-cube as the morphism $C \to \prod_d L_{x_d}^\X C$ induced by the natural transformation $\eta_d$. As the localizations are lax symmetric monoidal by \cref{Pmonoidal}, we see that the constructed functor lands in $\msf{CAlg}(\C)$ using \cite[Lemma 3.6]{GepnerGrothNikolaus15}. 
\end{proof}

\begin{ex}
When $d=2$ and the Balmer spectrum is irreducible, writing $\g$ for the unique generic point, $\p$ for the primes of dimension 1, and $\m$ for the primes of dimension 0, the punctured adelic cube with respect to the finest assembly data is
\[
\xymatrix@=1em{
& \prod_\p L_\p \Lambda_\p \1 \ar[dd] \ar[rr] & & L_\g \prod_\p L_\p \Lambda_\p \1 \ar[dd] & \\ & & L_\g \1 \ar[ur] \ar[dd] \\
& \prod_\p L_\p \prod_\m \Lambda_\m \1 \ar[rr]|-\hole & & L_\g \prod_\p L_\p \prod_\m \Lambda_\m \1 & \\
\prod_\m \Lambda_\m \1 \ar[rr] \ar[ur] & & L_\g \prod_\m \Lambda_\m \1  \ar[ur] 
}
\]
\end{ex}

With the definition of the punctured adelic cube in hand, we may now define the adelic model.  

\begin{defn}
The \emph{adelic model} $\C_\ad^\X$ for $\C$ is defined to be the full subcategory of $\equals{\1_\ad}$ on those objects $M$ for which the adjoint structure maps 
\[
\mathrm{ext}_A^B M(A) \to M(B)
\]
are equivalences for all inclusions $A \subseteq B$ (see \cref{rem:dataofpc}).
\end{defn}

\subsection{Adelic approximation}
In this section we show that the pullback of the punctured adelic cube $\1_\ad^\X$ is the tensor unit $\1$. This is one key input into the proof that $\C_\ad^\X$ is equivalent to $\C$. %The case when $\X$ is the finest assembly data previously appears as \cite[Theorem 8.1]{adelicm}, but even in this special case we give a different proof. 

We start with the following lemma which we will use frequently throughout the paper. We recall from \cref{dimsubscript} the convention that $x_i$ denotes an element of $\X$ of dimension $i$.
\begin{lem}\label{epointyproduct}
Let $(\X,\alpha)$ be assembly data for $\C$, and $0 \leqslant i \leqslant d$. 
\begin{enumerate}[label=(\alph*)]
\item\label{epointy1} Let $X \in \C$. Then we have \[\Gammale{i}\prod_{x_i} L^\X_{x_i}X \simeq \Elr{i} \otimes X.\]
\item\label{epointy2} We have $\Elr{i} \simeq \Elr{i} \otimes \1_\ad^\X(i)$.
\end{enumerate}
\end{lem}
\begin{proof}
For part \ref{epointy1}, we have \[\Gammale{i}\prod_{x_i} L^\X_{x_i}X \simeq \Gammale{i}\prod_{x_i} \Gammale{i}L^\X_{x_i}X\] by \cref{prop:Gammaproducts}. 
Now $\Gammale{i}L_{x_i}^\X X \simeq \oplus_{\p_i} \Gamma_{\p_i}L_{x_i}^\X X \simeq \oplus_{\p_i} L_{x_i}^\X\Gamma_{\p_i}X$ by \cref{splitting}. 
Note that $\downcl(\p_i) \cap \alpha^{-1}(\upcl_\X(x_i)) \subseteq \alpha^{-1}(x_i)$ since $\alpha$ is dimension preserving. 
Therefore by \cref{maxL}\ref{LXsplit}, \[L_{x_i}^\X \Gamma_{\p_i}X \simeq \bigoplus_{\alpha(\q_i)=x_i} L_{\q_i}\Gamma_{\p_i}X.\] 
But $L_{\q_i}\Gamma_{\p_i} \simeq 0$ unless $\p_i = \q_i$. 
Therefore combining all these observations, we have \[\Gammale{i}\prod_{x_i} L^\X_{x_i}X \simeq \Gammale{i}\prod_{x_i}\bigoplus_{\alpha(\q_i) = x_i} \Gamma_{\q_i}L_{\q_i}X.\] Now $\prod_{x_i}\oplus_{\alpha(\q_i)=x_i} \Gamma_{\q_i}L_{\q_i}X$ is $L_{\geqslant i}$-local,
so by applying \cref{splitting} and \cref{prop:Gammaproducts}, we see that \[\Gammale{i}\prod_{x_i} L^\X_{x_i}X \simeq \bigoplus_{\p_i} \Gamma_{\p_i}\prod_{x_i} \bigoplus_{\alpha(\q_i) = x_i} \Gamma_{\q_i}L_{\q_i} X \simeq \bigoplus_{\p_i} \Gamma_{\p_i}\prod_{x_i} \bigoplus_{\alpha(\q_i) = x_i} \Gamma_{\p_i}\Gamma_{\q_i}L_{\q_i} X.\] 
But since $\Gamma_{\p_i}L_{\q_i}X \simeq 0$ whenever $\p_i \neq \q_i$, we see that this is moreover equivalent to \[\bigoplus_{\p_i}\Gamma_{\p_i}L_{\p_i}X \simeq \Elr{i} \otimes X\] using \cref{lem:Epointy}, which completes the proof. 
For part \ref{epointy2}, we note that \[\Elr{i} \otimes \1_\ad^\X(i) = \Lge{i}\Gammale{i} \prod_{x_i} L_{x_i}^\X \Lambda_{x_i}^\X \1.\] Following the same steps as in the proof of part \ref{epointy1}, one sees that this is equivalent to \[\Lge{i}\bigoplus_{\p_i} \Gamma_{\p_i}\prod_{x_i} \bigoplus_{\alpha(\q_i) = x_i} \Gamma_{\p_i}\Gamma_{\q_i}L_{\q_i}\Lambda_{\q_i}\1\] so the result follows from the MGM equivalence together with \cref{lem:Epointy}.
\end{proof}

The next result is a higher dimensional version of the pullback square of \cref{prop:localduality}\ref{item:hptybicart}.

\begin{thm}\label{limitsofcubes}
Let $(\X,\alpha)$ be assembly data for $\C$. Consider the punctured adelic cube $\1_\ad^\X$ from \cref{1adfunctor}. The natural map \[\1 \to \mrm{lim}(\1_\ad^\X)\] is an equivalence.
\end{thm}
\begin{proof}
Given an $(i+1)$-cube, that is, a functor $C\colon \mathcal{P}([i]) \to \C$ , we define an $i$-cube $F_i(C)$ via \[F_i(C)(A) = \mrm{fib}(C(A) \to C(A \cup i)).\] The cubical reduction principle~\cite[\S 8.B]{adelicm} states that $C$ is a pullback $(i+1)$-cube if and only if $F_i(C)$ is a pullback $i$-cube.

Throughout this proof we consider $\1_\ad^\X$ as a $(d+1)$-cube with value at $\varnothing$ given by the unit object $\1$. We define an $i$-cube $(\1_\ad^\X)_f^i$ in an iterative fashion via \[(\1_\ad^\X)_f^i = F_i((\1_\ad^\X)_f^{i+1})\] for all $1 \leqslant i \leqslant d$, with the convention that $(\1_\ad^\X)^{d+1}_f = \1_\ad^\X$. By the cubical reduction principle, to show that $\1_\ad^\X$ is a pullback $(d+1)$-cube, it suffices to show that $(\1_\ad^\X)_f^1$ is a pullback $1$-cube. 

We claim that $(\1_\ad^\X)_f^i(A) \simeq \Gammale{i-1}\1_\ad^\X(A)$ for all $A \in \mathcal{P}([i{-}1])$ and $1 \leqslant i \leqslant d+1$. We prove this by downward induction on $i$. The base case $i = d+1$ is immediate since $\Gammale{d}$ is the identity. So suppose that the claim holds for some $i+1$. Then \begin{align*}(\1_\ad^\X)_f^i &= \mrm{fib}\left((\1_\ad^\X)_f^{i+1}(A) \to (\1_\ad^\X)_f^{i+1}(A \cup i+1)\right) & \\
&\simeq \mrm{fib}\left(\Gammale{i}\1_\ad^\X(A) \to \Gammale{i}\1_\ad^\X(A \cup i)\right) &\text{by induction hypothesis} \\
&\simeq \mrm{fib}\left(\Gammale{i}\1_\ad^\X(A) \to \Gammale{i}\prod_{x_i} L_{x_i}^\X\1_\ad^\X(A)\right) & \text{by definition of $\1_\ad^\X(A \cup i)$}\\
&\simeq \mrm{fib}\left(\Gammale{i}\1_\ad^\X(A) \to \Elr{i} \otimes \1_\ad^\X(A)\right) &\text{by \cref{epointyproduct}\ref{epointy1}} \\
&\simeq \Gammale{i-1}\1_\ad^\X(A)
\end{align*}
which proves the claim by induction.

We take $i=1$ in the above claim, to see that $(\1_\ad^\X)_f^1 = (\Gammale{0}\1 \to \Gammale{0}\1_\ad^\X(0))$. We need to show this is a pullback $1$-cube, that is, the map is an equivalence. But this follows from \cref{epointyproduct}\ref{epointy2}, which completes the proof.
\end{proof}

\subsection{Establishing the adelic model}
We turn to proving that $\C$ is equivalent to $\C_\ad^\X$. 
\begin{prop}\label{adelicadjunction}
There is an adjunction \[\1_\ad^\X \otimes -\colon \C \rightleftarrows \cad^\X : \lim\] where $(\1_\ad^\X \otimes X)(A) = \1_\ad^\X(A) \otimes X$ for all $A \in \mc{P}([d])_{\neq\varnothing}$ and $X \in \C$.
\end{prop}
\begin{proof}
By \cref{puncturedcubeadjunction} we have an adjunction \[\1_\ad^\X \otimes -\colon \C \rightleftarrows \equals{\1_\ad} : \lim\] so it suffices to show that the image of the left adjoint lands in the full subcategory $\C_\ad^\X$ which is clear from the definition. 
\end{proof}

\begin{prop}\label{1ad-fullyfaithful}
For all $X\in\C$, the unit map $X \to \lim(\1_\ad^\X \otimes X)$ is an equivalence. In particular, the functor $\1_\ad^\X \otimes - \colon \C\to \cad^\X$ is fully faithful.
\end{prop}

\begin{proof}
The unit of the  $(\1_\ad^\X \otimes -, \lim)$-adjunction is an equivalence on $\1$ by \cref{limitsofcubes}. Since $\1_\ad^\X \otimes -$ commutes with finite limits (as we work stably), it follows that the unit is an equivalence on all $X \in \C$. Therefore $\1_\ad^\X \otimes -$ is fully faithful.
\end{proof}

The next step is to show that $\1_\ad^\X \otimes -\colon \C \to \cad^\X$ is essentially surjective, but first we will need a preliminary result. 
We say that an object $M \in \cad^\X$ is \emph{supported in dimension $\leqslant i$} if $M(s) \simeq 0$ for all $s > i$.
\begin{lem}\label{torsioninlimit}
Let $M \in \cad^\X$ be supported in dimension $\leqslant i$. Then $M(i)$ is $\Gammale{i}$-torsion.
\end{lem}
\begin{proof}
It suffices to show that $\Elr{j} \otimes M(i) \simeq 0$ for all $j > i$. As $M \in \cad^\X$ and is supported in dimension $\leqslant i$, we know that \[\1_\ad^\X(j,i) \otimes_{\1_\ad^\X(i)} M(i) \simeq M(j,i) \simeq \1_\ad^\X(j,i) \otimes_{\1_\ad^\X(j)} M(j) \simeq 0.\] Applying $\Gammale{j}$ to this, we see that \[0 \simeq \Gammale{j}(\1_\ad^\X(j,i) \otimes_{\1_\ad^\X(i)} M(i)) \simeq \Gammale{j}\prod_{x_j}L_{x_j}^\X \1_\ad^\X(i) \otimes_{\1_\ad^\X(i)} M(i) \simeq \Elr{j} \otimes M(i)\] using \cref{epointyproduct}\ref{epointy1} as required.
\end{proof}

\begin{prop}\label{surjective}
The functor $\1_\ad^\A \otimes -\colon \C \to \cad^\X$ is essentially surjective.
\end{prop}
\begin{proof}
Consider the full subcategory $\equals{\1_\ad^\X}^\dagger$ of $\equals{\1_\ad^\X}$ spanned by the objects $M$, where the adjoint structure maps \[\1_\ad(A) \otimes_{\1_\ad(i)} M(i) \to M(A)\] with $\max(A) = i$ are equivalences. By~\cite[Proposition 5.2.2.12]{HTT}, to check that the evaluation map $\mrm{ev}_i\colon \equals{\1_\ad^\X}^\dagger \to \mod{\1_\ad(i)}$ has a right adjoint $f_i$, it suffices to check on homotopy categories. The derived functor of $\mrm{ev}_i$ has a right adjoint by \cite[Lemma 9.6]{adelicm}, so we therefore obtain an adjunction $\mrm{ev}_i \dashv f_i$, and then one may verify that $f_i$ has the form
\begin{equation}\label{fiform}
f_i(Z)(A) = \begin{cases}
\1_\ad^\X(A) \otimes_{\1_\ad^\X(i)} Z & \text{if $\mrm{max}(A) = i$} \\
\1_\ad^\X(A \cup i) \otimes_{\1_\ad^\X(i)} Z & \text{if $\mrm{max}(A) < i$} \\
0 & \text{otherwise}
\end{cases}
\end{equation}
as in~\cite[Lemma 9.6]{adelicm}, since objects in an $\infty$-category are the same as objects in its homotopy category. We note that $f_i(Z)$ is supported in dimension $\leqslant i$.

Fix $M \in \cad^\X$ which is supported in dimension $\leqslant i$. There is a map \[\phi_{M(i)}\colon \1_\ad^\X \otimes M(i) \to f_i(M(i))\] constructed as the adjoint to the action map \[\mrm{ev}_i(\1_\ad^\X \otimes M(i)) = \1_\ad^\X(i) \otimes M(i) \to M(i)\] where the module structure on the source comes from the left factor. We claim that $\phi_{M(i)}$ is an equivalence. This in particular will show that $ f_i(M(i))$ is in the essential image of $\1_\ad^\X\otimes -$ whenever $M\in\cad^\X$ and is supported in dimension $\leqslant i$.

Using the definition of $f_i$ as in \cref{fiform}, one verifies that if $\phi_{N(i)}$ is an equivalence (where $N \in \cad^\X$ is supported in dimension $\leqslant i$) then $\phi_{(N \otimes M)(i)}$ is an equivalence. In other words, the set of $M \in \cad^\X$ supported in dimension $\leqslant i$ for which $\phi_{M(i)}$ is an equivalence is a $\otimes$-ideal. 

We have that $M(i)$ is $\Gammale{i}$-torsion by \cref{torsioninlimit}, and is moreover $\Lge{i}$-local as it is a module over $\1_\ad^\X(i)$. As such \[M(i) \simeq (\1_\ad^\X(i) \otimes \Elr{i}) \otimes_{\1_\ad^\X(i)} M(i).\] Therefore, it suffices to check that $\phi_{N(i)}$ is an equivalence for $N = \1_\ad^\X \otimes \Elr{i}$, as $N(i) \simeq \1_\ad^\X(i) \otimes \Elr{i}$.

By \cref{epointyproduct}\ref{epointy2}, the source of $\phi_{\1_\ad^\X(i) \otimes \Elr{i}}(A)$ is $\1_\ad^\X(A) \otimes \Elr{i}$, and the target is 
\[\begin{cases}
\1_\ad^\X(A) \otimes \Elr{i} & \text{if $\mrm{max}(A) = i$} \\
\1_\ad^\X(A \cup i) \otimes \Elr{i} & \text{if $\mrm{max}(A) < i$} \\
0 & \text{otherwise.}
\end{cases}\]
So the case $\max(A) = i$ is clear, and the case $\max(A) > i$ holds since $\1_\ad^\X(A)$ is $\Lge{\max(A)}$-local. For the remaining case when $\max(A) < i$, we have \[\1_\ad^\X(A \cup i) \otimes \Elr{i} \simeq \Gammale{i}\prod_{x_i} L_{x_i}^\X \1_\ad^\X(A) \simeq \Elr{i} \otimes \1_\ad^\X(A)\] by \cref{epointyproduct}\ref{epointy1}. Therefore $\phi_{\1_\ad^\X(i) \otimes \Elr{i}}$ is an equivalence, and hence $\phi_{M(i)}$ is for all $M \in \cad^\X$ supported in dimension $\leqslant i$.

%Previous proof below
%By the detection property, it suffices to check that $\Elr{j} \otimes \phi_M$ is an equivalence for all $0 \leqslant j \leqslant \mrm{dim}(\C)$. If $j < i$, then both the source and target of $\Elr{j} \otimes \phi_M$ are equivalent to 0 since $\1_\ad^\X\otimes M(i)$ and $f_i(M(i))$ are both $\Lge{i}$-local in each entry. If $j > i$, then again both the source and target of $\Elr{j} \otimes \phi_M$ are equivalent to 0 as $M(i)$ is $\Gammale{i}$-torsion by \cref{torsioninlimit}. So it remains to check that $\Elr{i} \otimes \phi_M$ is an equivalence. 
%
%Now $\Elr{i} \otimes \1_\ad^\X(A) \otimes M(i) \simeq 0$ whenever $i < \max(A)$ since $\1_\ad^\X(A)$ is $\max(A)$-local, and similarly for $f_i(M(i))$ by \cref{fiform}. If $\max(A) = i$ then \[\Elr{i} \otimes f_i(M(i))(A) \simeq \Elr{i} \otimes \1_\ad^\X(A) \otimes M(i)\] by \cref{fiform} together with \cref{epointyproduct}\ref{epointy2}. When $\max(A) < i$, then we have \begin{align*}\Elr{i} \otimes f_i(M(i))(A) &\simeq \Elr{i} \otimes \1_\ad^\X(A \cup i) \otimes M(i) \\ &\simeq \Gammale{i}\prod_{x_i} L_{x_i}^\X \1_\ad(A) \otimes M(i) \\ &\simeq \Elr{i} \otimes \1_\ad^\X(A) \otimes M(i)\end{align*} by \cref{fiform} and \cref{epointyproduct}\ref{epointy2}, together with \cref{epointyproduct}\ref{epointy1} for the last equivalence. Therefore $\phi_M$ is an equivalence whenever $M \in \cad^\X$ is supported in dimension $\leqslant i$. 

We now fix an arbitrary $M \in \cad^\X$ and using the above, argue that it is in the essential image of $\1_\ad^\X \otimes -$. Take $s$ to be the maximal number such that $M$ is supported in dimension $\leqslant s$. The unit of the $(\mrm{ev}_s, f_s)$-adjunction provides a map $\eta\colon M \to f_s(M(s))$ and we have seen that $f_s(M(s))$ belongs to the essential image of $\1_\ad^\X \otimes -$. By construction the fibre of $\eta$ is supported in dimension $\leqslant s-1$ so by induction we may assume that it belongs to the essential image of $\1_\ad^\X \otimes -$ too.  Using that $\1_\ad^\X \otimes -$ is exact and fully faithful by \cref{1ad-fullyfaithful} (so that we can lift maps), we deduce that $M$ is in the essential image of $\1_\ad^\X \otimes -$ as claimed.
\end{proof}

We are now in a position to state and prove the main result of this section, which says that the adjunction of \cref{adelicadjunction} is in fact an equivalence. %This recovers~\cite[Theorem 9.3]{adelicm} in the case when the assembly data is the finest one.
\begin{thm}\label{thm-adelicm}
Let $(\X,\alpha)$ be assembly data for $\C$.  The adjunction \[\1_\ad^\X \otimes -\colon \C \rightleftarrows \cad^\X: \mrm{lim}\] is an equivalence of categories.
\end{thm}
\begin{proof}
The functor $\1_\ad^\X \otimes -$ is fully faithful by~\cref{1ad-fullyfaithful} and essentially surjective by \cref{surjective}. 
\end{proof}

\section{Cofibre functors}\label{sec:cofibres}
In the previous section we proved that $\C$ is equivalent to the adelic model $\C_\ad^\X$, which is a certain full subcategory of the punctured cube category. The torsion model will be obtained from this by taking iterated cofibres, see the introduction for a schematic of the 2-dimensional case. As such, the goal of this section is to define categories which are the image of cofibres on cubes, and to construct these cofibre functors. To motivate the story, suppose that the diagram of rings $R$ is a 2-cube. In this section we will construct categories whose objects are diagrams of the following shapes, together with functors as indicated: 
\[
\begin{tikzcd}[column sep = 0.4cm, row sep = 0.3cm]
\bullet \ar[rr] \ar[dd] && \bullet  \ar[dd] & \ar[rr,  "\mrm{cof}", yshift = -1em] & \hspace{5mm} &\,  &  \bullet \ar[dd] \ar[rr] & & \bullet \ar[dd] \ar[rr, color=OIblue] & & \bullet \ar[dd] & \ar[rr,  "p^\msf{m}", yshift = -1em] & \hspace{5mm} &\,  & \bullet \ar[dd] \ar[rr, color=OIblue] & & \bullet \ar[dd] \\
 &&&& \textcolor{white}{\simeq} &&&&&&&& {} \\
\bullet \ar[rr] &\ar[phantom,"{\cubeL(R)}"below=0.5em]& \bullet & \ar[rr,  "p"', leftarrow, yshift = 1em] & \hspace{5mm} &\,  & \bullet  \ar[rr]  & & \bullet  \ar[rr, color=OIblue] \ar[phantom,"{\cubeLR(R)}"below=0.5em] && \bullet  & \ar[rr,  "\mrm{fib}"', yshift = 1em, leftarrow] & \hspace{5mm} &\,  & \bullet \ar[rr, color=OIblue] &\ar[phantom,"{\cubeR(R)}"below=0.5em] & \bullet
\end{tikzcd}
\]

Starting with the left hand diagram, we take the cofibre of the horizontal maps. The middle category $\cubeLR(R)$ then remembers the data of the whole cofibre sequence, whereas the rightmost category forgets the leftmost part of the cofibre sequence. We explain the meaning of the coloured edges in the next paragraph.

The superscript $\msf{m}$ stands for `mixed', meaning that the handedness of some maps is altered. 
We demonstrate the idea behind this in the simplest case, when $R$ is a 1-cube, so that an
object of $\cubeL(R)$ provides a map $M(\varnothing) \to
\mrm{res}^0_\varnothing M(0)$ (see \cref{rem:dataofcubeL}). Taking the
cofibre of this map, we obtain $\mrm{res}^0_\varnothing M(0) \to
C$. We call the former type of map \emph{oplax}, and latter type \emph{lax}. Thus in the above diagram, the blue maps are lax, and the black maps are oplax. 

\subsection{Mixed cubes}
For the rest of the section, we let $\C$ be a presentably symmetric monoidal stable $\infty$-category. In a similar fashion to how we defined cubes in \cref{defn:cubeL}, we
may define a cube in which the handedness of maps becomes mixed in the sense that there are both lax and oplax maps, see \cref{rem:dataofcubeR}.
The following definition formalises this idea. Recall \cref{dj,cubefaces}.
\begin{defn}\label{cubeR}
	Let $R\colon \mc{P}([i]) \to \msf{CAlg}(\C)$ be a diagram of rings. We inductively define a category of mixed cubes $\cubeR(R)$ which comes equipped with a functor $\pi_{i}^\msf{m}\colon \cubeR(R) \to \mrm{Fun}(\mc{P}([i]), \mod{R(\varnothing)})$. When $i=-1$ we take $\cubeR(R) = \mod{R(\varnothing)}$ and $\pi_{-1}^\msf{m} = \mrm{id}$. If $i \geqslant 0$ we define $\cubeR(R)$ as the pullback
	\[ 
	\begin{tikzcd}[column sep=0.5cm]
		\cubeR(R) \ar[r, "\pi_{i}^\msf{m}"] \ar[d, "{(\theta_{i}, (\psi_j, \psi_{\backslash j})_j)}"'] & \mrm{Fun}(\mc{P}([i]), \mod{R(\varnothing)}) \ar[d, "{(d_{\backslash i}^*,(d_j^*, d_{\backslash j}^*)_j)}"] \\
		\cubeL(R_{i}) \times \prod_{j=0}^{i-1} \left(\cubeR(R_j) \times \cubeR(\Rwithout{j})\right) \ar[r] & \substack{\mrm{Fun}(\mc{P}([i{-}1]), \mod{R(\varnothing)}) \\ \times \\ \prod_{j=0}^{i-1}\mrm{Fun}(\mc{P}([i{-}1]), \mod{R(\varnothing)})^{\times 2}}
	\end{tikzcd}
	\]
	where the bottom horizontal map is the product of the two maps 
	\[
	\cubeL(R_i) \xrightarrow{(\mrm{res}^i_{\varnothing})_*\circ \pi_{i-1}} \mrm{Fun}(\mc{P}([i{-}1]), \mod{R(\varnothing)}) 
	\]
	and
	\[
	\prod_{j=0}^{i-1} \left(\cubeR(R_j) \times \cubeR(\Rwithout{j})\right)  \xrightarrow{\prod_j ( (\mrm{res}^j_{\varnothing})_* \circ \pi_{i-1}^{\msf{m}}) \times \pi_{i-1}^{\msf{m}}} \prod_{j=0}^{i-1}\mrm{Fun}(\mc{P}([i{-}1]), \mod{R(\varnothing)})^{\times 2}.
	\]
\end{defn}

\begin{ex}
	If $i=0$, the data of an object in $\cubeR(R)$ is the data of an $R(\varnothing)$-module $M(\varnothing)$, an $R(0)$-module $M(0)$, and a map of $R(\varnothing)$-modules $\mrm{res}^0_{\varnothing }M(0) \to M(\varnothing)$.
\end{ex}

\begin{ex}\label{ex-cubeR-dim1}
	If $i=1$, the data of an object in $\cubeR(R)$ is the data of modules $M(\varnothing), M(0), M(1)$ and $M(10)$ over the rings indicated by the subset together with module maps $g_1^{\varnothing} \colon \mrm{res}_{\varnothing}^1 M(1) \to M(\varnothing)$, $g_{10}^0 \colon \mrm{res}_0^{10} M(10) \to M(0)$, $f_{1}^{10} \colon M(1) \to \mrm{res}^{10}_{1}M(10)$ and $f_{\varnothing}^0 \colon M(\varnothing) \to \mrm{res}^0_{\varnothing}M(0)$ making the following diagram commute
	\[
	\begin{tikzcd}
		\mrm{res}^1_{\varnothing} M(1) \arrow[r,"g_1^{\varnothing}"] \arrow[d,"\mrm{res}^1_{\varnothing}f_1^{10}"'] & M(\varnothing) \arrow[d,"f_{\varnothing}^0"]\\
		\mrm{res}^1_{\varnothing} \mrm{res}^{10}_1 M(10) \simeq \mrm{res}^0_{\varnothing} \mrm{res}^{10}_0 M(10) \arrow[r, "\mrm{res}^0_{\varnothing}g_{10}^0"] &\mrm{res}_{\varnothing}^0 M(0).
	\end{tikzcd}
	\]
	Note that the map $\theta_1\colon \cubeR(R) \to \cubeL(R_1)$ records only the left vertical arrow in the above square, the map $\psi_0\colon \cubeR(R) \to \cubeR(R_0)$ records the bottom horizontal, and the map $\psi_{\backslash 0}\colon \cubeR(R) \to \cubeR(R_{\backslash 0})$ records the top horizontal. As such, the notation for these restriction maps reflects the module structure, rather than the indexing by the poset. We recommend the reader to compare this example with the corresponding $i=1$ case of $\cubeL(R)$ as described in \cref{ex-dim1-cubeL}.
\end{ex}

\begin{rem}\label{rem:dataofcubeR}
 An object $M$ of $\cubeR(R)$ can be described as the data of:
 \begin{enumerate}[label=(\roman*)]
 	\item for each $A \in \mathcal{P}([i])$ an $R(A)$-module $M(A)$;
	\item for each $A \in \mathcal{P}([i])$ and $k \not\in A$ with $k \neq i$, an $R(A)$-module map $M(A) \to \mrm{res}^{A \cup k}_A M(A\cup k)$;
	\item for each $A \in \mathcal{P}([i{-}1])$, an $R(A)$-module map $\mrm{res}^{A \cup i}_A M(A \cup i) \to M(A)$.
 \end{enumerate}
The data can be assembled into an $(i+1)$-cube whose faces are required to commute.
\end{rem}

\begin{prop}\label{restrictR}
	Let $R\colon \mc{P}([i]) \to \msf{CAlg}(\C)$ be a diagram of rings. There is a restriction map $\theta_{\backslash i}\colon \cubeR(R) \to \cubeL(R_{\backslash i})$. 
\end{prop}
\begin{proof}
	This is analogous to \cref{restricttofaces}.
\end{proof}

The following defines a category $\cubeLR(R)$ which is the result of gluing the cubes in $\cubeL(R)$ and the corresponding mixed cubes in $\cubeR(R)$ together, along their common face.
\begin{defn}\label{defn:cubeLR}
	Let $R\colon \mc{P}([i]) \to \msf{CAlg}(\C)$ be a diagram of rings. We define a category $\cubeLR(R)$ as the pullback 
	\[\begin{tikzcd}
		\cubeLR(R) \arrow[r, "p"] \arrow[d, "p^\msf{m}"'] & \cubeL(R) \arrow[d, "\rho_{i}"]\\
		\cubeR(R) \arrow[r,"\theta_{i}"'] & \cubeL(R_{i})
	\end{tikzcd}\]
	where $\rho_i$ and $\theta_i$ are as defined in \cref{defn:cubeL,cubeR} respectively.
\end{defn}

\begin{ex}
	When $i=0$, an object in $\cubeLR(R)$ can be pictured as a diagram \[M(\varnothing) \to \mrm{res}^0_\varnothing M(0) \to N(\varnothing).\]
\end{ex}

\begin{ex}
	When $i=1$, an object in $\cubeLR(R)$ can be pictured as a commutative diagram 
	\[
	\begin{tikzcd}
		M(\varnothing) \arrow[r, "f_{\varnothing}^1"] \arrow[d,"f_{\varnothing}^{0}"'] & \mrm{res}^1_{\varnothing} M(1) \arrow[d, "\mrm{res}^1_{\varnothing}f_1^{10}"'] \arrow[r,"g_1^{\varnothing}"] &N(\varnothing) \arrow[d,"h^0_{\varnothing}"]\\
		\mrm{res}^0_{\varnothing} M(0) \arrow[r, "\mrm{res}_{\varnothing}^0 f_0^{10}"] & \mrm{res}^0_{\varnothing}\mrm{res}^{10}_0 M(10) \simeq  \mrm{res}_{\varnothing}^1 \mrm{res}_1^{10}M(10) \arrow[r, "\mrm{res}^0_{\varnothing}g_{10}^0"] & \mrm{res}_{\varnothing}^0 N(0).
	\end{tikzcd}
	\]
\end{ex}

\subsection{Cofibre functors on cubes} 
In this subsection, we will construct a cofibre functor $\cubeL(R) \to \cubeR(R)$ which factors through $\cubeLR(R)$. In order to do so, we recall the following lemma which constructs a cofibre functor between cubical diagrams in a stable $\infty$-category. In short, this deals with the case when the vertices of the cubical diagrams all live in the same underlying category, which is not the case in $\cubeL(R)$ where the vertices are modules over different rings.
\begin{lem}\label{boringcofibre}
	There is a cofibre functor \[\mrm{cof}_{i}\colon \mrm{Fun}(\mc{P}([i]), \C) \to \mrm{Fun}(\mc{P}([i]), \C)\] which takes the cofibre along the maps $A \to A \cup i$. That is, \[
	\mrm{cof}_i(X)(A) = \begin{cases} X(A \cup i) & \text{if $i \not\in A$} \\ \mrm{cofib}(X(A\backslash i) \to X(A)) & \text{if $i \in A$}.\end{cases}
	\]
	Similarly, there is a fibre functor \[\mrm{fib}_i\colon \mrm{Fun}(\mc{P}([i]), \C) \to \mrm{Fun}(\mc{P}([i]), \C)\] which takes the fibre along the maps $A \to A \cup i$. These functors are inverse to each other.
\end{lem}
\begin{proof}
	We make the identification $\mc{P}([i]) = \mc{P}([i{-}1]) \times \Delta^1$, where the $\Delta^1$ records the adding $i$ direction. Then by adjunction $\mrm{Fun}(\mc{P}([i]), \C) \simeq \mrm{Fun}(\mc{P}([i{-}1]), \mrm{Fun}(\Delta^1, \C))$. The functor we want is then given by post-composition with the functor $\mrm{cofib}\colon \mrm{Fun}(\Delta^1, \C) \to \mrm{Fun}(\Delta^1, \C)$ as constructed in~\cite[Remark 1.1.1.7]{HA}. The construction of $\mrm{fib}_i$ is similar, and they are inverse to one another by stability as shown in \cite[Remark 1.1.1.7]{HA}.
\end{proof}

The following proposition enhances the previous statement by allowing the vertex categories of the cubical diagrams to change.
\begin{prop}\label{cofLtoR}
	Let $R\colon \mc{P}([i]) \to \msf{CAlg}(\C)$ be a diagram of rings. There is a cofibre functor $\mrm{cof}^\msf{m}\colon \cubeL(R) \to \cubeR(R)$, and a fibre functor $\mrm{fib}^\msf{m}\colon \cubeR(R) \to \cubeL(R)$ which form an equivalence of categories \[\xymatrix{
		\cubeL(R) \ar@<0.7ex>[r]^-{\mrm{cof}^\msf{m}} \ar@{}[r]|-{\simeq} \ar@<-0.7ex>@{<-}[r]_-{\mrm{fib}^\msf{m}} & \cubeR(R).
	}\]
\end{prop}
\begin{proof}
	We argue by induction. For the case $i=-1$, we have $\cubeL(R) = \mod{R(\varnothing)}$ and $\cubeR(R) = \mod{R(\varnothing)}$ so in this case we take $\mrm{cof}^\msf{m} = \mrm{id}$ and $\mrm{fib}^\msf{m} = \mrm{id}$. So now suppose both functors are constructed and are equivalences for $i$. We extend this to $i+1$ as follows. As both $\cubeL(R)$ and $\cubeR(R)$ are defined as  pullbacks (see \cref{defn:cubeL,cubeR}) it suffices to give functors which are quasi-inverses on each component. On the top right vertex, for $\mrm{cof}^\msf{m}$ we take the map $\mrm{cof}_{i+1}\colon \mrm{Fun}(\mc{P}([i+1]), \mod{R(\varnothing)}) \to \mrm{Fun}(\mc{P}([i+1]), \mod{R(\varnothing)})$ given by the cofibre constructed in \cref{boringcofibre}, and for $\mrm{fib}^\msf{m}$ we similarly take the fibre functor. These are quasi-inverses by \cref{boringcofibre}. On the bottom right vertex, that is, for the map 
	\[\begin{tikzcd}
		\mrm{Fun}(\mc{P}([i]), \mod{R(\varnothing)}) \times \prod_{j=0}^{i-1}\mrm{Fun}(\mc{P}([i]), \mod{R(\varnothing)})^{\times 2} \ar[d] \\ \mrm{Fun}(\mc{P}([i]), \mod{R(\varnothing)}) \times \prod_{j=0}^{i-1}\mrm{Fun}(\mc{P}([i]), \mod{R(\varnothing)})^{\times 2},
	\end{tikzcd}\] for $\mrm{cof}^\msf{m}$, we take the identity on the first component, and on the other component we take $(\mrm{cof}_{i}, \mrm{cof}_{i})$ from \cref{boringcofibre}. Analogously, for $\mrm{fib}^\msf{m}$ we take the identity on the first component and $(\mrm{fib}_i, \mrm{fib}_i)$ on the other component. These are again quasi-inverse to each other by \cref{boringcofibre}. On the bottom left vertex, that is for the map
	\[\cubeL(R_{i}) \times \prod_{j=0}^{i-1}\left(\cubeL(R_j) \times \cubeL(\Rwithout{j})\right) \to \cubeL(R_{i}) \times \prod_{j=0}^{i-1} \left(\cubeR(R_j) \times \cubeR(\Rwithout{j})\right)\] 
	we take the identity on the first component, and the map given by the inductive step on the remaining factor. By induction, these are quasi-inverses, and hence we conclude.
\end{proof}

%\begin{rem}
%We note that by construction the following diagram is commutative \todo{Do we need this?}
%\begin{equation}\label{com-square-cof}
%	\begin{tikzcd}
%		\cubeL(R) \arrow[r,"\mrm{cof}^{\msf{m}}"] \arrow[d,"\pi_{i+1}"'] & \cubeR(R) \arrow[d,"\pi_{i+1}^{\msf{m}}"] \\
%		\mrm{Fun}(\mathcal{P}([i]), \mod{R(\varnothing)}) \arrow[r,"\mrm{cof}_i"] & \mrm{Fun}(\mathcal{P}[i], \mod{R(\varnothing)}).
%	\end{tikzcd}
%\end{equation}
%\end{rem}

Combining the results of this section, we can enhance the previous result so that the cofibre functor records the map whose cofibre we are taking. More precisely, we show that the functor $\mrm{cof}^\msf{m}\colon \cubeL(R) \to \cubeR(R)$ constructed in the previous result factors through $\cubeLR(R)$.
\begin{cor}\label{cofLtoLR}
	Let $R\colon \mc{P}([i]) \to \msf{CAlg}(\C)$ be a diagram of rings. There is a cofibre functor $\mrm{cof}\colon\cubeL(R) \to \cubeLR(R)$ such that the diagram
	\[\begin{tikzcd}\cubeL(R)\ar[r, "\mrm{cof}"] \ar[rd, "\mrm{cof}^\msf{m}"'] & \cubeLR(R) \ar[d, "p^\msf{m}"] \\
	& \cubeR(R)
	\end{tikzcd}\] commutes. Similarly, there is a fibre functor $\mrm{fib}\colon\cubeR(R) \to \cubeLR(R)$ such that $\mrm{fib}^\msf{m} \simeq p\circ\mrm{fib}$.
	\end{cor}
\begin{proof}
	By definition of $\cubeLR(R)$ it suffices to give compatible maps $\cubeL(R) \to \cubeL(R)$ and $\cubeL(R) \to \cubeR(R)$. For the former we take the identity functor, and for the latter we take the cofibre functor $\mrm{cof}^\msf{m}$ constructed in \cref{cofLtoR}. By construction, the diagram in the statement commutes. The construction of the fibre functor is similar.
\end{proof}

\begin{rem}\label{rem:Vcofid}
It is immediate from the construction of $\mrm{cof}\colon \cubeL(R) \to \cubeLR(R)$ that the diagram 
\[\begin{tikzcd}
\cubeL(R) \ar[r, "\mrm{cof}"] \ar[dr, "\mrm{id}"'] & \cubeLR(R) \ar[d, "p"] \\
& \cubeL(R)
\end{tikzcd}\]
commutes.
\end{rem}

\begin{ex}
	Let us give an example of the two above cofibre functors in the case $i=0$. Given an object ($f\colon M(\varnothing) \to \mrm{res}^0_\varnothing M(0)$) in $\cubeL(R)$, we have 
	\[
	\begin{tikzcd}
		{(M(\varnothing) \xrightarrow{f} \mrm{res}^0_\varnothing M(0))} \ar[r, rightsquigarrow, "{(\ref{cofLtoR})}"] \ar[d, rightsquigarrow, "{(\ref{cofLtoLR})}"'] & {(\mrm{res}^0_\varnothing M(0) \to \mrm{cofib}(f))} \\  {(M(\varnothing) \xrightarrow{f} \mrm{res}^0_\varnothing M(0) \to \mrm{cofib}(f))} &
	\end{tikzcd}
	\]
\end{ex}

\begin{ex}
	In the case $i=1$, the cofibre functor of (\ref{cofLtoLR}) sends the diagram in \cref{ex-dim1-cubeL} to 
	\[
	\begin{tikzcd}
		M(\varnothing) \arrow[r, "f_{\varnothing}^1"] \arrow[d,"f_{\varnothing}^{0}"'] & \mrm{res}^1_{\varnothing} M(1) \arrow[d, "\mrm{res}^1_{\varnothing}f_1^{10}"'] \arrow[r,"g_1^{\varnothing}"] & \mrm{cof}(f_\varnothing^1) \arrow[d]\\
		\mrm{res}^0_{\varnothing} M(0) \arrow[r, "\mrm{res}_{\varnothing}^0 f_0^{10}"] & \mrm{res}^0_{\varnothing}\mrm{res}^{10}_0 M(10) \simeq  \mrm{res}_{\varnothing}^1 \mrm{res}_1^{10}M(10) \arrow[r, "\mrm{res}^0_{\varnothing}g_{10}^0"] & \mrm{res}^0_\varnothing \mrm{cof}(f^{10}_0) \simeq \mrm{cof}(\mrm{res}^0_\varnothing(f^{10}_0))
	\end{tikzcd}
	\]
	where the right hand vertical is the induced map on cofibres. On the other hand, the cofibre functor of (\ref{cofLtoR}) records the right hand square of the diagram.
\end{ex}

\section{Categories of cofibre sequences}\label{catsofcofs}
In the previous section, we constructed mixed cubes and various (co)fibre functors relating the different categories. In this section we explain how to modify the category $\cubeLR(R)$ so that the cofibre functor $\mrm{cof}\colon \cubeL(R) \to \cubeLR(R)$ restricts to an equivalence. Loosely speaking, we want to restrict to the full subcategory of $\cubeLR(R)$ consisting of the cofibre sequences. However, it is important to note that being a cofibre sequence is not a property: it is the \emph{data} of a null homotopy, together with the property that the induced map into the cone is an equivalence. As such, to record the data of cofibre sequences, we first must enlarge $\cubeLR(R)$ to build in the data of null homotopies, which gives a category denoted by $\cubeLRplus(R)$. We then take an appropriate full subcategory $\cubeLRplus(R)_\mrm{c}$ of this enlargement. 

\subsection{The one-dimensional case}
We begin by implementing the aforementioned enlargement in the one-dimensional case.
\begin{defn}\label{defn:dummy}
	Let $R\colon \mc{P}([0]) \to \msf{CAlg}(\C)$ be a diagram of rings. We define the category $\D^R$ to be the pullback
	\[
	\begin{tikzcd}[column sep=1.5cm]
		\D^R \ar[r,"\pi^\D"] \ar[d, "{(p_t, p_r)}"'] & \mrm{Fun}(\mc{P}([1]), \mod{R(\varnothing)}) \ar[d, "{(\mrm{ev}_t ,\mrm{ev}_r)}"] \\
		\cubeL(R) \times \cubeR(R) \ar[r, "\pi_0 \times \pi_0^\msf{m}"'] & \mrm{Fun}(\mc{P}([0]), \mod{R(\varnothing)})^{\times 2}.
	\end{tikzcd}
	\]
	where the evaluations $\mrm{ev}_t$ and $\mrm{ev}_r$ restrict to the top and the right faces of the square respectively, and $\pi_0$ and $\pi_0^\msf{m}$ are as defined in \cref{defn:cubeL,cubeR} respectively.
\end{defn}

\begin{ex}
The objects of $\D^R$ may be pictured as commutative diagrams
\[
\begin{tikzcd}
	M(\varnothing) \ar[r] \ar[d] & \mrm{res}^0_\varnothing N(0) \ar[d] \\
	P(\varnothing) \ar[r] & Q(\varnothing). 
\end{tikzcd}
\] 
The vertex $P(\varnothing)$ is to be thought of as a `dummy' vertex, see \cref{adultdummy}. The indexing denotes the ring which the objects are modules over.
\end{ex}

We next recall the following result from~\cite[Remark 1.1.1.7]{HA}, which describes an equivalence between the arrow category and the full subcategory of cofibre sequences. 
\begin{lem}\label{lem-cofibre-HA}
	Let $\C$ be a stable $\infty$-category and let $\mrm{Fun}(\mathcal{P}([1]), \C)_\mrm{c}$ denote the full subcategory of objects $X \in \mrm{Fun}(\mathcal{P}([1]), \C) $ such that
	\begin{itemize}
		\item $X(1) \simeq 0$;
		\item the square represented by $X$ is a pushout in $\C$.
	\end{itemize} 
	There is a cofibre functor 
	\[
	\mathrm{cof}\colon \mrm{Fun}(\mathcal{P}([0]), \C) \to \mrm{Fun}(\mathcal{P}([1]), \C)
	\]
	and an equivalence of $\infty$-categories 
	\[
	\begin{tikzcd}
		\mrm{Fun}(\mathcal{P}([0]), \C) \arrow[r, yshift=0.2cm, "\mathrm{cof}"] \arrow[r, phantom, "\simeq" description]      & \mrm{Fun}(\mathcal{P}([1]), \C)_\mrm{c} \arrow[l, yshift=-0.2cm, "\mrm{ev}_t"]
	\end{tikzcd}
	\]
	where $\mrm{ev}_t$ records the arrow indexed on $\varnothing \to 0$ of the square. \qed
\end{lem}

We next upgrade the cofibre functor $\mrm{cof}\colon \cubeL(R) \to \cubeLR(R)$ from \cref{cofLtoLR} in the one-dimensional case, so that it records the data of the null homotopy.
\begin{con}\label{adultdummy}
	We now define a functor $\mrm{cof}_+\colon\cubeL(R) \to \D^R$, which  takes an object $f\colon M \to \mrm{res}^0_\varnothing N$ to the commutative square
	\[
	\begin{tikzcd}
		M \ar[r, "f"] \ar[d] & \mrm{res}^0_\varnothing N \ar[d] \\
		0 \ar[r] & \mrm{cofib}(f). 
	\end{tikzcd}
	\] 
	In other words this enhances the functor $\mrm{cof}^{\msf{m}}\colon \cubeL(R) \to \cubeR(R)$ from \cref{cofLtoR} in the case $i=0$ by recording the null homotopy. By definition of $\D^R$ it suffices to give maps $\cubeL(R) \to \mrm{Fun}(\mc{P}([1]), \mod{R(\varnothing)})$ and $\cubeL(R) \to \cubeL(R) \times \cubeR(R)$. For the former, we take the composite \[\cubeL(R) \xrightarrow{\pi_0} \mrm{Fun}(\mc{P}([0]), \mod{R(\varnothing)}) \xrightarrow{\mathrm{cof}} \mrm{Fun}(\mc{P}([1]), \mod{R(\varnothing)})\] where the latter map is the cofibre functor from~\cref{lem-cofibre-HA}. For the map $\cubeL(R) \to \cubeL(R) \times \cubeR(R)$, we take the identity on the first component, and the functor $\mrm{cof}^\msf{m}$ constructed in \cref{cofLtoR} on the second component.
\end{con}

\subsection{The general case}
Using the constructions from the one-dimensional case, we may enhance the definition of $\cubeLR(R)$ in higher dimensions, by including `dummy vertices' so that we can record the data of null homotopies when taking cofibres. 

\begin{nota}
Fix a diagram of rings $R\colon \mc{P}([i]) \to \msf{CAlg}(\C)$. We
define another diagram of rings for every subset $A \subseteq
\{0,\ldots, i-1\}$, by setting $R^{A,i}\colon \mc{P}([0]) \to
\msf{CAlg}(\C)$ to be $R^{A,i}(\varnothing) = R(A)$ and $R^{A,i}(0) =
R(A \cup i)$. We also consider the category $\D^{R^{A,i}}$ which we
denote as $\D^{A,i}$ for improved  typography. 
\end{nota}

\begin{ex}
	An object in $\D^{A,i}$ can be represented by a commutative diagram
	\[
	\begin{tikzcd}
		M(A) \ar[r] \ar[d] & \mrm{res}^{A \cup i}_A N(A \cup i) \ar[d] \\
		P(A) \ar[r] & Q(A). 
	\end{tikzcd}
	\] 
	The indexing denotes the rings which the objects are modules over.
\end{ex}
\begin{defn}\label{adddummies}
	Let $R\colon \mc{P}([i]) \to \msf{CAlg}(\C)$ be a diagram of rings. We define a category $\cubeLRplus(R)$ as the pullback
	\[
	\begin{tikzcd}[column sep=2cm]
		\cubeLRplus(R) \ar[r, "\mrm{fgt}_+"] \ar[d] & \cubeLR(R) \ar[d, "(\eta_A)_A"] \\
		\prod\limits_{A \in \mc{P}([i{-}1])} \D^{A,i} \ar[r, "\prod \chi_A"'] & \prod\limits_{A \in \mc{P}([i{-}1])} \cubeLR(R^{A,i}) 
	\end{tikzcd}
	\]
	where the maps $\chi_A$ and $\eta_A$ are defined as follows. For $\chi_A\colon \D^{A,i} \to \cubeLR(R^{A,i})$, by \cref{defn:cubeLR} it suffices to give compatible functors $\D^{A,i} \to \cubeL(R^{A,i})$ and $\D^{A,i} \to \cubeR(R^{A,i})$ for which we take  $p_t$ and $p_r$ respectively, as defined in \cref{defn:dummy}. We define $\eta_A$ as follows. Since $\cubeLR(R^{A,i})$ is defined as a pullback (see \cref{defn:cubeLR}), it suffices to give compatible maps $\cubeLR(R) \to \cubeL(R^{A,i})$ and $\cubeLR(R) \to \cubeR(R^{A,i})$ for which we take the composites
	\[\cubeLR(R) \xrightarrow{p} \cubeL(R) \to \cubeL(R^{A,i}) \quad \text{and} \quad \cubeLR(R) \xrightarrow{p^\msf{m}} \cubeR(R) \to \cubeR(R^{A,i})\] where we use that $R^{A,i}$ is a face of $R$ for the unlabelled arrows, and $p$ and $p^\msf{m}$ are the maps as in \cref{defn:cubeLR}. For the restriction to faces, in the first case we use \cref{restricttofaces}, and in the second case we use the restrictions $\psi_j$ and $\psi_{\backslash j}$ from \cref{cubeR}. 
\end{defn}

\begin{rem}\label{rem:filtrationdegree}
Let $R\colon \mc{P}([i]) \to \msf{CAlg}(\C)$ be a diagram of rings. We introduce notation for some of the vertices of objects in $\cubeLRplus(R)$. An object $M \in \cubeLRplus(R)$ in particular contains the data of
\begin{enumerate}
\item an $R(A)$-module $M(A^i)$ for each $A \in \mc{P}([i])$;
\item an $R(A)$-module $M(A^{i-1})$ for each $A \in \mc{P}([i{-}1])$;
\item an $R(A)$-module map $M(A^i) \to \mrm{res}_A^BM(B^i)$ for each $A\subseteq B$;
\item an $R(A)$-module map $\mrm{res}_A^B M(B^i) \to M(A^{i{-}1})$ for each $A \subseteq B$;
\item an $R(A)$-module $M(A^{(i{-}1)})$ for each $A \in \mc{P}([i{-}1])$;
\item an $R(A)$-module map $M(A^i) \to M(A^{(i{-}1)})$ for each $A \in \mc{P}([i{-}1])$;
\item an $R(A)$-module map $M(A^{(i-1)}) \to M(A^{i{-}1})$ for each $A \in \mc{P}([i{-}1])$;
\end{enumerate}
together with coherence data. We call the superscript the \emph{filtration degree}, so that $M(A^i)$ has filtration degree $i$ for example. The objects $M(A^{(i-1)})$ will be thought of as dummy vertices, built in only to encode null homotopies as we will see in \cref{glueplusf}.
\end{rem}

\begin{ex}\label{ex:glueplus}
	When $i=1$, objects in the category $\cubeLRplus(R)$ can be pictured as a commutative diagram 
	\[
	\begin{tikzcd}[column sep=2cm]
		& M(\varnothing^{(0)}) \ar[dr] & \\
		M(\varnothing^1) \arrow[ur] \arrow[r, "f_{\varnothing}^1"] \arrow[d,"f_{\varnothing}^{0}"'] & \mrm{res}^1_{\varnothing} M(1^1) \arrow[d, "\mrm{res}^1_{\varnothing}f_1^{10}"'] \arrow[r,"g_1^{\varnothing}"] &M(\varnothing^{0}) \arrow[d,"h^0_{\varnothing}"]\\
		\mrm{res}^0_{\varnothing} M(0^1) \arrow[dr] \arrow[r, "\mrm{res}_{\varnothing}^0 f_0^{10}"] & \mrm{res}^0_{\varnothing}\mrm{res}^{10}_0 M(10^1) \simeq  \mrm{res}_{\varnothing}^1 \mrm{res}_1^{10}M(10^1) \arrow[r, "\mrm{res}^0_{\varnothing}g_{10}^0"] & \mrm{res}_{\varnothing}^0 M(0^{0}) \\
		& \mrm{res}^0_\varnothing M(0^{(0)}) \ar[ur] & 
	\end{tikzcd}
	\]
\end{ex}

We now turn to proving that there is a cofibre functor $\mrm{cof}_+\colon \cubeL(R) \to \cubeLRplus(R)$ which gives an equivalence of categories between $\cubeL(R)$ and the category of cofibre sequences in $\cubeLRplus(R)$. Firstly, we construct the desired functor.

\begin{con}\label{constructcofplus}
In order to construct a cofibre functor $\mrm{cof}_+\colon \cubeL(R) \to \cubeLRplus(R)$, since $\cubeLRplus(R)$ is defined as a pullback (\cref{adddummies}) it suffices to give compatible functors $\cubeL(R) \to \cubeLR(R)$ and $\cubeL(R) \to \D^{A,i}$ for each subset $A$. For the former we take the cofibre functor constructed in \cref{cofLtoLR}, and for the latter we take the composite
\[\cubeL(R) \to \cubeL(R^{A,i}) \xrightarrow{\mrm{cof}_+} \D^{A,i}\]
where the first map exists by \cref{restricttofaces}, and the latter map is from \cref{adultdummy}.
In an analogous way, one defines a fibre functor $\mrm{fib}_+\colon \cubeR(R) \to \cubeLRplus(R)$.
\end{con}

\begin{rem}\label{eq:coffgt}
It follows from the construction of $\mrm{cof}_+$ that we have a natural equivalence
\[
\mrm{fgt}_+\mrm{cof}_+ \simeq \mrm{cof}
\]
of functors $\cubeL(R) \to \cubeLR(R)$.
\end{rem}

Since we want to restrict to a setting where the functor $\mrm{cof}_+$ of \cref{constructcofplus} is an equivalence, we need to identify its image. Morally, the image consists of the cofibre sequences, and the following definition makes this precise.
\begin{defn}\label{glueplusf}
	Let $R\colon \mc{P}([i]) \to \msf{CAlg}(\C)$ be a diagram of rings. We define $\cubeLRplus(R)_\mrm{c}$ to be the full subcategory of $\cubeLRplus(R)$ on the objects $M$ for which the objects $M(A^{\s{i-1}}) \simeq 0$ and the commutative square 
	\[
	\begin{tikzcd}[column sep=5em, row sep=3em]
		M(A^{i}) \arrow[r] \ar[d] & \mrm{res}_{A}^{A \cup i}M((A \cup i)^{i}) \arrow[d]  \\
		0\simeq M(A^{\s{i-1}}) \arrow[r] & M(A^{i-1})
	\end{tikzcd}
	\]
	is a pushout, for all $A \in \mc{P}([i{-}1])$. 
\end{defn}

\begin{ex}
Looking back at \cref{ex:glueplus}, for the object $M$ to lie in the full subcategory $\cubeLRplus(R)_\mrm{c}$ we require that the objects $M(\varnothing^{(0)})$ and $M(0^{(0)})$ are zero objects, and the triangular-shaped subdiagrams are pushouts.
\end{ex}

In order to prove the assertion that $\cubeLRplus(R)_\mrm{c}$ is the image of $\mrm{cof}_+$, we require the following result which allows us to reformulate the definition of $\cubeLRplus(R)_\mrm{c}$.
\begin{nota}\label{unotation}
We write $u\colon \cubeLRplus(R) \to \cubeL(R)$ for the composite functor $p\circ\mrm{fgt}_+$, and $u^\msf{m}\colon \cubeLRplus(R) \to \cubeR(R)$ for the composite $p^\msf{m} \circ \mrm{fgt}_+$.
\end{nota}

\begin{lem}\label{lem:counitmapexists}
Let $R\colon \mc{P}([i]) \to \msf{CAlg}(\C)$ be a diagram of rings. For any $M \in \cubeLRplus(R)$, there is a natural map $\epsilon_M\colon \mrm{cof}_+(u(M)) \to M$.
\end{lem}
\begin{proof}
Since $\cubeLRplus(R)$ is defined as a pullback (\cref{adddummies}), and $\mrm{Fun}(\cubeLRplus(R),-)$ commutes with limits, the data of a natural transformation $\mrm{cof}_+u \Rightarrow \mrm{id}$ is equivalent to the data of compatible natural transformations
\[
\begin{gathered}
\begin{tikzcd}[row sep=-0.1cm]
&  \, \arrow[dd, Rightarrow] \\
\cubeLR(R) \ar[rr, "\mrm{cof}\circ p", yshift=3mm] \ar[rr, "\mrm{id}"', yshift=-3mm] & & \cubeLR(R)
\\  &  \, 
\end{tikzcd}
\end{gathered}
\text{ and }
\begin{gathered}\begin{tikzcd}[row sep =-0.1cm]
 & \, \ar[dd, Rightarrow, yshift=-0.5mm] \\ 
\D^{A,i} \ar[rr, "\mrm{cof}_+\circ p_t", yshift=3mm] \ar[rr, "\mrm{id}"', yshift=-3mm] & & \D^{A,i}.
\\ & \, 
\end{tikzcd}
\end{gathered}
\] 
for each subset $A$. By the universal property of pullbacks, the data of the latter natural transformation $\mrm{cof}_+p_t \Rightarrow \mrm{id}$ is equivalent to the data of compatible natural transformations 
\[
\begin{gathered}
\begin{tikzcd}[row sep=-0.1cm, column sep = 0.5em]
& \, \ar[dd, Rightarrow] \\ 
\cubeL(R) \ar[rr, "\mrm{id}", yshift=3mm] \ar[rr, "\mrm{id}"', yshift=-3mm] && \cubeL(R) 
\\ & \,
\end{tikzcd}
\end{gathered}
\text{ and }
\begin{gathered}
\begin{tikzcd}[row sep=-0.1cm, column sep = 0.6em]
&\, \ar[dd, Rightarrow]\\
\mrm{Fun}(\mc{P}([1]), \mod{R(\varnothing)}  \ar[rr, "\mrm{cof}\circ\mrm{ev}_t", yshift=3mm] \ar[rr, "\mrm{id}"', yshift=-3mm] && \mrm{Fun}(\mc{P}([0]), \mod{R(\varnothing)}).
\\ & \,
\end{tikzcd}
\end{gathered}
\] 
%
%\[\begin{tikzcd}
%\cubeL(R) \ar[r, "\mrm{id}", yshift=1mm] \ar[r, "\mrm{id}"', yshift=-1mm] & \cubeL(R)
%\end{tikzcd}~~\mrm{and}~~\begin{tikzcd}
%\mrm{Fun}(\mc{P}([1]), \mod{R(\varnothing)}) \ar[r, "\mrm{cof}\circ\mrm{ev}_t", yshift=1mm] \ar[r, "\mrm{id}"', yshift=-1mm] & \mrm{Fun}(\mc{P}([0]), \mod{R(\varnothing)}).
%\end{tikzcd}\]
For the former we take the identity transformation, and for the latter we take the equivalence from \cref{lem-cofibre-HA}. 

The data of a natural transformation $\mrm{cof}\circ p \Rightarrow \mrm{id}$ is equivalent to the data of compatible natural transformations 
\[
\begin{gathered}
\begin{tikzcd}[row sep=-0.1cm]
& \, \ar[dd, Rightarrow] \\ 
\cubeL(R) \ar[rr, "\mrm{id}", yshift=3mm] \ar[rr, "\mrm{id}"', yshift=-3mm] && \cubeL(R) 
\\ & \,
\end{tikzcd}
\end{gathered}
\text{ and }
\begin{gathered}\begin{tikzcd}[row sep = -0.1cm]
 & \, \ar[dd, Rightarrow] \\ 
\cubeR(R)\ar[rr, "\mrm{cof}^\msf{m}\circ \mrm{fib}^\msf{m}", yshift=3mm] \ar[rr, "\mrm{id}"', yshift=-3mm] & &\cubeR(R).
\\ & \, 
\end{tikzcd}
\end{gathered}
\] 
%
%\[\begin{tikzcd}
%\cubeL(R) \ar[r, "\mrm{id}", yshift=1mm] \ar[r, "\mrm{id}"', yshift=-1mm] & \cubeL(R)
%\end{tikzcd}~~\mrm{and}~~\begin{tikzcd}
%\cubeR(R) \ar[r, "\mrm{cof}^\msf{m}\circ \mrm{fib}^\msf{m}", yshift=1mm] \ar[r, "\mrm{id}"', yshift=-1mm] & \cubeR(R).
%\end{tikzcd}\]
and we take the identity for the former, and the equivalence from \cref{cofLtoR} for the latter.
\end{proof}

\begin{rem}
Note that the category $\cubeLRplus(R)_\mrm{c}$ of \cref{glueplusf} may be equivalently described as the full subcategory of $\cubeLR(R)$ on the objects $M$ for which the map $\varepsilon_M\colon \mrm{cof}_+(u(M)) \to M$ of \cref{lem:counitmapexists} is an equivalence.
\end{rem}

We may now assemble the results of this section to prove the desired equivalence of categories between $\cubeL(R)$ and the collection of cofibre sequences in $\cubeLRplus(R)$.
\begin{thm}\label{cubestocofibres}
	Let $R\colon \mc{P}([i]) \to \msf{CAlg}(\C)$ be a diagram of rings. The cofibre functor $\mrm{cof}_+\colon \cubeL(R) \to \cubeLRplus(R)$ and the fibre functor $\mrm{fib}_+\colon \cubeR(R) \to \cubeLRplus(R)$ restrict to give equivalences of categories
	\[
	\begin{tikzcd}
		\cubeL(R) \ar[r, yshift=0.2cm, "\mrm{cof}_+"] \ar[r, phantom, "\simeq" description] & \cubeLRplus(R)_\mrm{c} \ar[r, yshift=0.2cm, "u^\msf{m}"] \ar[l, yshift=-0.2cm, "u"] \ar[r, phantom, "\simeq" description] & \cubeR(R) \ar[l, yshift=-0.2cm, "\mrm{fib}_+"] 
	\end{tikzcd}
	\]
\end{thm}
\begin{proof}
We prove the left hand equivalence; one then obtains the right hand equivalence by an analogous argument. The map $\epsilon_M\colon \mrm{cof}_+(u(M)) \to M$ from \cref{lem:counitmapexists} is an equivalence for all $M \in \cubeLRplus(R)_\mrm{c}$ by construction. We claim that there is a natural equivalence $N \to u(\mrm{cof}_+(N))$ for all $N \in \cubeL(R)$. To see this, we have natural equivalences
\begin{align*}
u(\mrm{cof}_+(N)) &\simeq p(\mrm{fgt}_+(\mrm{cof}_+(N))) &\text{by definition of $u$ (\cref{unotation})} \\
&\simeq p(\mrm{cof}(N)) &\text{by \cref{eq:coffgt}} \\
&\simeq N &\text{by \cref{rem:Vcofid}}
\end{align*} 
which completes the proof.
\end{proof}

\section{The punctured story}\label{sec:punctured}
Since the adelic model is built from a punctured cube, we next revisit the theory of \cref{sec:cofibres,catsofcofs} in the punctured setting. Many of the definitions and statements in this section are similar to those given in the previous sections, but we provide brief proofs for completeness.

\begin{nota}
The poset $\mc{P}([i])_{\neq\{i\}}$ is obtained from an $(i+1)$-cube by removing the vertex $\{i\}$. It can be viewed as a reflected form of a punctured $(i+1)$-cube, see \cref{ex:equalsmix1} for the shape in the case $i=1$. 
\end{nota}

\begin{defn}\label{mix}
Let $R\colon \mc{P}([i]) \to \msf{CAlg}(\C)$ be a diagram of rings. We define $\equalsmix{R}$ as the pullback
\[\begin{tikzcd}[column sep=4cm]
\equalsmix{R}\arrow[r] \arrow[d, "{(\zeta_i, (\xi_j)_j)}"'] & \mrm{Fun}(\mc{P}([i])_{\neq\{i\}}, \mod{R(\varnothing)}) \arrow[d, "{(d_{\backslash i}^*, (d_j^*)_j)}"]\\
\cubeL(R_i) \times \prod_{j=0}^{i-1} \cubeR(R_j) \arrow[r,"(\mrm{res}^i_{\varnothing})_*\circ \pi_i \times \prod (\mrm{res}^j_\varnothing)_*\circ \pi_i^\msf{m}"'] & \substack{\mrm{Fun}(\mc{P}[i{-}1], \mod{R(\varnothing)}) \\ \times \\ \prod_{j=0}^{i-1}  \mrm{Fun}(\mc{P}([i{-}1]), \mod{R(\varnothing)})}
\end{tikzcd}\] 
where $d_j$ and $d_{\backslash i}$ are as in \cref{dj}, $\pi_i$ is as in \cref{defn:cubeL}, and $\pi_i^\msf{m}$ is as in \cref{cubeR}.
\end{defn}

\begin{ex}\label{ex:equalsmix1}
When $i=1$, an object in $\equalsmix{R}$ can be represented by a diagram
\[
\begin{tikzcd}
\mrm{res}^1_{\varnothing} M(1) \arrow[d,"\mrm{res}^1_{\varnothing}f_1^{10}"'] & \\
\mrm{res}^1_{\varnothing} \mrm{res}^{10}_0 M(10) \simeq \mrm{res}^0_{\varnothing} \mrm{res}^{10}_0 M(10) \arrow[r, "\mrm{res}^0_{\varnothing}g_{10}^0"] &\mrm{res}_{\varnothing}^0 M(0).
\end{tikzcd}
\]
The functor $\zeta_1$ restricts to the left hand vertical in the above diagram, and $\xi_0$ restricts to the bottom horizontal. As such, just as in \cref{ex-cubeR-dim1}, the notation for the functors $\zeta_i$ and $\xi_j$ reflects the indexing on modules, rather than the indexing by the poset. 
\end{ex}

\begin{rem}\label{rem:dataequalsmix}
One may think of objects in $\equalsmix{R}$ in an analogous fashion to \cref{rem:dataofcubeR}, by restricting to non-empty subsets.
\end{rem}

Recall that in \cref{defn:cubeLR} we glued together the cubes in $\cubeL(R)$ and the corresponding mixed cubes in $\cubeR(R)$ along their common face. We implement this in the punctured setting in the following definition.
\begin{defn}\label{twolayers}
Let $R\colon \mc{P}([i]) \to \msf{CAlg}(\C)$ be a diagram of rings. We define $\glue{R}$ as the pullback
\[\begin{tikzcd}
\glue{R} \ar[r, "r"] \ar[d, "r^\msf{m}"'] & \equals{R} \ar[d, "\tau_i"] \\
\equalsmix{R} \ar[r, "\zeta_i"'] & \cubeL(R_i)
\end{tikzcd}\]
where $\tau_i$ and $\zeta_i$ are as in \cref{Cequals,mix} respectively.
\end{defn}

With the previous definitions in hand, we now turn to defining cofibre functors on punctured cubes with a view to giving analogues of \cref{cofLtoR} and \cref{cubestocofibres}. 
\begin{prop}\label{puncLtoR}
Let $R\colon \mc{P}([i]) \to \msf{CAlg}(\C)$ be a diagram of rings. There is a cofibre functor $\mrm{cof}^\msf{m}\colon \equals{R} \to \equalsmix{R}$, and a fibre functor $\mrm{fib}^\msf{m}\colon \equalsmix{R} \to \equals{R}$ which form an equivalence of categories 
\[\xymatrix{
		\equals{R} \ar@<0.7ex>[r]^-{\mrm{cof}^\msf{m}} \ar@{}[r]|-{\simeq} \ar@<-0.7ex>@{<-}[r]_-{\mrm{fib}^\msf{m}} & \equalsmix{R}.
	}\]
\end{prop}
\begin{proof}
We define a cofibre functor $\mrm{cof}^\msf{m}\colon \equals{R} \to \equalsmix{R}$ as follows. Since both are pullbacks we define the map component-wise. For the map on the top right vertex, consider the inclusions $\iota\colon\mc{P}([i])_{\neq\varnothing} \to \mc{P}([i])$ and $\iota_i\colon \mc{P}([i])_{\neq\{i\}} \to \mc{P}([i])$. We then take the required map to be the composite
\begin{align*}
\mrm{Fun}(\mc{P}([i])_{\neq\varnothing}, \mod{R(\varnothing)}) &\xrightarrow{\iota_*} \mrm{Fun}(\mc{P}([i]), \mod{R(\varnothing)}) \\ 
&\xrightarrow{\mrm{cof}_i} \mrm{Fun}(\mc{P}([i]), \mod{R(\varnothing)}) \\
&\xrightarrow{\iota_i^*} \mrm{Fun}(\mc{P}([i])_{\neq\{i\}}, \mod{R(\varnothing)})
\end{align*}
where $\iota_*$ is the right Kan extension, and $\mrm{cof}_i$ is as in \cref{boringcofibre}. On the bottom right vertex, on the $j$th component we take the functor $\mrm{cof}_{i-1}$ if $j \neq i$, and the identity otherwise. On the bottom left vertex, we take the identity on the $j=i$ component, and the cofibre functor $\mrm{cof}^\msf{m}\colon \cubeL(R_j) \to \cubeR(R_j)$ from \cref{cofLtoR} on all other components. Since on each component the functor is an equivalence, the induced functor on the pullbacks is also an equivalence. 
\end{proof}

The next construction is an analogue of \cref{cofLtoLR}.
\begin{con}\label{puncnodummies}
We now define a cofibre functor $\mrm{cof}\colon\equals{R} \to \glue{R}$ which records the domains of the maps whose cofibres we take. By definition of $\glue{R}$ such a functor is uniquely determined by compatible functors $\equals{R} \to \equals{R}$ and $\equals{R} \to \equalsmix{R}$. For the former we take the identity functor. For the latter, we take the functor $\mrm{cof}^\msf{m}$ as constructed in \cref{puncLtoR}. We note that by construction, $r\circ \mrm{cof} \simeq \mrm{id}$.
\end{con}

We now turn to enhancing the definition of $\glue{R}$ given in \cref{twolayers} to include `dummy' vertices, so that we may record the data of null homotopies when taking cofibres. 
\begin{defn}\label{adddummies}
Let $R\colon \mc{P}([i]) \to \msf{CAlg}(\C)$ be a diagram of rings. We define the category $\glueplus{R}$ as the pullback
\[
\begin{tikzcd}[column sep=2cm]
\glueplus{R} \ar[r, "\mrm{fgt}_+"] \ar[d] & \glue{R} \ar[d, "(\eta_A)_A"] \\
\prod\limits_{A \in \mc{P}([i{-}1])_{\neq\varnothing}} \D^{A,i} \ar[r, "\prod \chi_A"'] & \prod\limits_{A \in \mc{P}([i{-}1])_{\neq\varnothing}} \cubeLR(R^{A,i}) 
\end{tikzcd}
\]
where the maps $\chi_A$ and $\eta_A$ are defined as follows. For $\chi_A\colon \D^{A,i} \to \cubeLR(R^{A,i})$ it suffices to give compatible functors $\D^{A,i} \to \cubeL(R^{A,i})$ and $\D^{A,i} \to \cubeR(R^{A,i})$ for which we take  $p_t$ and $p_r$ respectively, as defined in \cref{defn:dummy}. 

We define $\eta_A$ as follows. By definition of $\cubeLR(R^{A,i})$ it suffices to give compatible maps $\glue{R} \to \cubeL(R^{A,i})$ and $\glue{R} \to \cubeR(R^{A,i})$ for which we take the composites
	\[\glue{R} \xrightarrow{r} \equals{R} \xrightarrow{\tau_a} \cubeL(R_a) \to \cubeL(R^{A,i})\] and \[\glue{R} \xrightarrow{r^\msf{m}} \equalsmix{R} \xrightarrow{\xi_a} \cubeR(R_a) \to \cubeR(R^{A,i})\] where $r$ and $r^\msf{m}$ are as in \cref{twolayers}, $\tau_a$ is as in \cref{Cequals}, and $\xi_a$ is as in \cref{mix}. The final map in both cases uses that $R^{A,i}$ is a face of $R_a$ for any $a \in A$, see \cref{restricttofaces} and \cref{cubeR} respectively.
 \end{defn}

\begin{rem}
We employ the notation introduced in \cref{rem:filtrationdegree} to describe objects of $\glueplus{R}$; the only difference is that the subsets $A$ of $\{0,\ldots,i\}$ are now required to be non-empty.
\end{rem}

With the definition of $\glueplus{R}$, we may now define an `enhanced' cofibre functor $\mrm{cof}_+\colon \equals{R} \to \glueplus{R}$ which records the data of the null homotopies.
\begin{con}\label{con:enhancedcofibre}
In order to define a cofibre functor $\mrm{cof}_+\colon \equals{R} \to \glueplus{R}$, by definition of $\glueplus{R}$ as a pullback (see \cref{adddummies}), it suffices to give compatible maps $\equals{R} \to \glue{R}$ and $\equals{R} \to \prod_A \D^{A,i}$. For the former we take the cofibre functor constructed in \cref{puncnodummies}. For the latter, on the $A$ component we take the composite 
\[\equals{R} \xrightarrow{\tau_a} \cubeL(R_a) \xrightarrow{\mrm{res}} \cubeL(R^{A,i}) \xrightarrow{\mrm{cof}_+} \D^{A,i}\]
where $a \in A$ and the maps are defined as in \cref{Cequals}, \cref{restricttofaces}, and \cref{adultdummy} respectively. We note that $\mrm{fgt}_+\circ\mrm{cof}_+ \simeq \mrm{cof}$ by construction.
\end{con}

\begin{nota}\label{nota:v}
We define the functor $v\colon \glueplus{R} \to \equals{R}$ to be the composite $r \circ \mrm{fgt}_+$, and $v^\msf{m}\colon \glueplus{R} \to \equalsmix{R}$ to be $r^\msf{m}\circ\mrm{fgt}_+$. Informally speaking, $v$ and $v^\msf{m}$ forget the dummy vertices and then restrict to the underlying unmixed and mixed punctured cubes respectively.
\end{nota}
\begin{lem}\label{punccounit}
Let $R\colon \mc{P}([i]) \to \msf{CAlg}(\C)$ be a diagram of rings. For any $M \in \glueplus{R}$, there is a natural map $\epsilon_M\colon \mrm{cof}_+(v(M)) \to M$.
\end{lem}
\begin{proof}
The data of a natural transformation $\mrm{cof}_+v \Rightarrow \mrm{id}$ is equivalent to the data of compatible natural transformations
\[
\begin{gathered}
\begin{tikzcd}[row sep=-0.1cm]
& \, \ar[dd, Rightarrow] \\ 
\glue{R} \ar[rr, "\mrm{cof}\circ r", yshift=3mm] \ar[rr, "\mrm{id}"', yshift=-3mm] && \glue{R}
\\ & \,
\end{tikzcd}
\end{gathered}
\text{ and }
\begin{gathered}
\begin{tikzcd}[row sep=-0.1cm]
&\, \ar[dd, Rightarrow, yshift=-0.5mm]\\
\D^{A,i} \ar[rr, "\mrm{cof}_+\circ p_t", yshift=3mm] \ar[rr, "\mrm{id}"', yshift=-3mm] && \D^{A,i}.
\\ & \,
\end{tikzcd}
\end{gathered}
\] 
or each non-empty subset $A$. The latter was constructed in \cref{lem:counitmapexists}. 
For the former, the data of a natural transformation $\mrm{cof}\circ r \Rightarrow \mrm{id}$ is equivalent to the data of compatible natural transformations 
\[
\begin{gathered}
\begin{tikzcd}[row sep=-0.1cm]
& \, \ar[dd, Rightarrow] \\ 
\equals{R}  \ar[rr, "\mrm{id}", yshift=3mm] \ar[rr, "\mrm{id}"', yshift=-3mm] && \equals{R}
\\ & \,
\end{tikzcd}
\end{gathered}
\text{ and }
\begin{gathered}
\begin{tikzcd}[row sep=-0.1cm]
&\, \ar[dd, Rightarrow]\\
\equalsmix{R}  \ar[rr, "\mrm{cof}^\msf{m}\circ \mrm{fib}^\msf{m}", yshift=3mm] \ar[rr, "\mrm{id}"', yshift=-3mm] && \equalsmix{R}.
\\ & \,
\end{tikzcd}
\end{gathered}
\] 
and we take the identity for the former, and the equivalence from \cref{puncLtoR} for the latter.
\end{proof}

\begin{defn}\label{glueplusf}
Let $R\colon \mc{P}([i]) \to \msf{CAlg}(\C)$ be a diagram of rings. We define $\glueplus{R}_\mrm{c}$ to be the full subcategory of $\glueplus{R}$ on the objects $M$ for which the objects $M(A^{\s{i-1}}) \simeq 0$ and the commutative square 
\[
\begin{tikzcd}[column sep=5em, row sep=3em]
	M(A^{i}) \arrow[r] \ar[d] & \mrm{res}_{A}^{A \cup i}M((A \cup i)^{i}) \arrow[d]  \\
	0\simeq M(A^{\s{i-1}}) \arrow[r] & M(A^{i-1})
\end{tikzcd}
\]
	is a pushout, for all $A \in \mc{P}([i{-}1])_{\neq\varnothing}$. Equivalently, $\glueplus{R}_\mrm{c}$ is the full subcategory of $\glueplus{R}$ on the objects $M$ for which the map $\epsilon_M\colon \mrm{cof}_+(v(M)) \to M$ of \cref{punccounit} is an equivalence.
\end{defn}

With all the previous results in hand, we are ready to prove the analogue of \cref{cubestocofibres} in the punctured setting. 
\begin{thm}\label{punccofthm}
Let $R\colon \mc{P}([i]) \to \msf{CAlg}(\C)$ be a diagram of rings. There are equivalences of categories 
\[
\begin{tikzcd}
\equals{R} \ar[r, yshift=0.2cm, "\mrm{cof}_+"] \ar[r, phantom, "\simeq" description] & \glueplus{R}_\mrm{c} \ar[r, yshift=0.2cm, "v^\msf{m}"] \ar[l, yshift=-0.2cm, "v"] \ar[r, phantom, "\simeq" description] & \equalsmix{R}. \ar[l, yshift=-0.2cm, "\mrm{fib}_+"] 
\end{tikzcd}
\]
\end{thm}
\begin{proof}
We prove the left-hand equivalence first. The map $\epsilon_M\colon \mrm{cof}_+(v(M)) \to M$ from \cref{lem:counitmapexists} is an equivalence for all $M \in \glueplus{R}_\mrm{c}$ by construction. We claim that there is a natural equivalence $N \to v(\mrm{cof}_+(N))$ for all $N \in \equals{R}$. To see this, we have natural equivalences
\begin{align*}
v(\mrm{cof}_+(N)) &\simeq r(\mrm{fgt}_+(\mrm{cof}_+(N))) &\text{by definition of $v$} \\
&\simeq r(\mrm{cof}(N)) &\text{by \cref{con:enhancedcofibre}} \\
&\simeq N &\text{by \cref{puncnodummies}}
\end{align*} 
which completes the proof of the left-hand equivalence. The right-hand equivalence can be proved analogously.
\end{proof}

\section{Equivalences via iterated cofibres}\label{sec:global}
Recall that the adelic model is a certain full subcategory of the category of punctured cubes. The torsion model will be obtained from the adelic model by taking iterated cofibres, so in this section we formalise the process of taking iterated cofibres on punctured cubes for arbitrary diagram of rings $R\colon\mc{P}([d]) \to \msf{CAlg}(\C)$. We will inductively define categories $\biggerthan{i}{R}$ where $\biggerthan{i}{R}$ is the result of taking cofibres in $\biggerthan{i+1}{R}$ and remembering the data of null homotopies and domains. In the terminology of \cref{rem:filtrationdegree}, the superscripts record the filtration degrees which are permitted. By extending the framework developed in the previous sections,  we prove that there are equivalences
\[\equals{R} \xrightarrow{\sim} \biggerthanf{d-1}{R} \xrightarrow{\sim} \biggerthanf{d-2}{R} \xrightarrow{\sim} \cdots \xrightarrow{\sim} \biggerthanf{1}{R} \xrightarrow{\sim} \biggerthanf{0}{R}\] where each step is given by a cofibre functor, and the subscripts $\mrm{c}$ denote a certain full subcategory of cofibre sequences. In the next section, we will prove the existence of the torsion model by identifying the essential image of the adelic model, seen as a subcategory of $\equals{\1_\ad}$, under these equivalences.

\subsection{Defining layers}
Fix a natural number $d$ and a diagram of rings $R\colon
\mc{P}([d]) \to \msf{CAlg}(\C)$. For every $0 \leqslant i \leqslant
d-1$, we write  $R^{\leqslant i}\colon \mc{P}([i]) \to \msf{CAlg}(\C)$
for the restriction of the diagram $R$ to $\mc{P}([i])$. We begin by defining categories $\biggerthan{i}{R}$ by induction for each $0 \leqslant i \leqslant d$, which are equipped with restriction functors $u_{=i}\colon \biggerthan{i}{R} \to \equals{R^{\leqslant i}}$ for all $i \neq d$. In order to do this, we require the following lemma.

\begin{lem}\label{restricttobottomlayer}
There is a restriction functor $s^{i{-}1}\colon \equalsmix{R^{\leqslant i}} \to \equals{R^{\leqslant i-1}}$ which forgets all the objects $M(A)$ where $i \in A$ (using the notation of \cref{rem:dataequalsmix,rem:dataofpc}).
\end{lem}
\begin{proof}
Without loss of generality suppose $R = R^{\leqslant i}$. Since $\equals{R^{\leqslant i-1}}$ is defined as a pullback (see \cref{Cequals}) it suffices to give compatible maps \[\equalsmix{R} \to \mrm{Fun}(\mc{P}([i{-}1])_{\neq \varnothing}, \mod{R(\varnothing)}) \quad \text{and} \quad \equalsmix{R} \to \prod_{j=0}^{i-1}\cubeL((R^{\leqslant i-1})_j).\]  For the former, we have a map $\equalsmix{R} \to \mrm{Fun}(\mc{P}([i])_{\neq\{i\}}, \mod{R(\varnothing)})$ by \cref{mix} and we postcompose with the restriction along the inclusion $\mc{P}([i{-}1])_{\neq\varnothing} \to \mc{P}([i])_{\neq\{i\}}$ given by $A \mapsto A \cup i$. For the latter, we take the composite map 
\[\equalsmix{R} \xrightarrow{(\xi_j)_j} \prod_{j=0}^{i-1}\cubeR(R_j) \to \prod_{j=0}^{i-1} \cubeL((R^{\leqslant i-1})_j)\]
where $\xi_j$ is as in \cref{mix} and the latter map exists as $(R^{\leqslant i-1})_j$ is a face of $R_j$. 
\end{proof}

\begin{defn}\label{Cgeq}
Let $R\colon \mc{P}([d]) \to \msf{CAlg}(\C)$ be a diagram of rings. We set $\biggerthan{d}{R} = \cubeL(R_d)$ and define $\biggerthan{d-1}{R} = \equalsmix{R}$. There is a restriction functor \[u_{=d-1}:=s^{d-1}\colon \biggerthan{d-1}{R} \to \equals{R^{\leqslant d-1}}\] by \cref{restricttobottomlayer}.
For each $0 \leqslant i \leqslant d-2$, we define $\biggerthan{i}{R}$ inductively as the pullback \[
\begin{tikzcd}[column sep=1.5cm]
\biggerthan{i}{R} \ar[r, "u_{\geqslant i+1}"] \ar[d, "g_i^{i+1}"'] & \biggerthan{i+1}{R} \ar[d, "u_{=i+1}"] \\
\glueplus{R^{\leqslant i+1}} \ar[r, "v"'] & \equals{R^{\leqslant i+1}}
\end{tikzcd}
\]
and the functor $u_{=i}\colon\biggerthan{i}{R} \to \equals{R^{\leqslant i}}$ as the composite \[\biggerthan{i}{R} \xrightarrow{g_i^{i+1}} \glueplus{R^{\leqslant i+1}} \xrightarrow{v^\msf{m}} \equalsmix{R^{\leqslant i+1}}\xrightarrow{s^i} \equals{R^{\leqslant i}}\] where the latter two maps are defined in \cref{nota:v} and \cref{restricttobottomlayer} respectively.
\end{defn}

Since $\biggerthan{i}{R}$ consists of copies of $\glueplus{R^{\leqslant j}}$ for $j \geqslant i+1$ being glued together, we extend the description of objects of $\glueplus{R^{\leqslant j}}$ given in \cref{rem:filtrationdegree} to the setting of $\biggerthan{i}{R}$. This will allow us to describe the objects of $\biggerthan{i}{R}$ in a more tangible fashion. We refer the reader to \cref{objectsbiggerthan} for more details. 

The following definition describes indexing categories which encode the shape of objects in $\biggerthan{i}{R}$. The key categories introduced are
$I_{-}(d)$ and $I(d)$. The former category contains fewer objects than the
latter, hence the minus sign in the notation.
The latter is the main category of interest so that it receives the simplest notation. In the terminology of the previous section, $I(d)$ extends $I_{-}(d)$ by adding in dummy vertices.
\begin{defn}\label{poset}\leavevmode
\begin{enumerate}[label=(\roman*)]
\item We define $I_{-}(d)$ to be the category with objects $A^k$ where $0 \leqslant k \leqslant d$ and $A$ is a non-empty subset of $\{0,1,\ldots,k\}$, such that if $k=d$, we have $d \in A$. We call the exponent of $A^k$ its \emph{filtration degree}. The morphisms of $I_{-}(d)$ are generated by the maps:
\begin{enumerate}
	\item \emph{oplax restrictions}: $\ext{i}{k}{A}\colon A^k \to (A \cup i)^k$ whenever $i \not\in A$,
	\item \emph{lax restrictions}: $\res{k}{A}\colon (A \backslash k)^{k-1} \to A^k$ where $k \in A$, 
\end{enumerate}
such that every possible diagram involving the above arrows
commutes. We will follow the convention that we denote the oplax
restriction maps in black and the lax restriction maps in blue,
whenever we give a pictorial representation.
\item We define $I(d)$ to be the category which is obtained from $I_{-}(d)$ by adding an extra object $A^{(i)}$ and extra maps $A^{(i)} \to A^{i+1}$ and $A^{(i)} \to A^{i}$ making the diagram
 \begin{equation}\label{square-I_+(A,s)}
 \begin{tikzcd}[column sep=1.7cm]
 A^{i+1} \arrow[r,"\ext{i+1}{i+1}{A}"] & (A\cup i+1)^{i+1} \\
 A^{(i)} \arrow[u]\arrow[r] & A^{i} \arrow[u, "\res{i+1}{A \cup i+1}"'].
 \end{tikzcd}
 \end{equation}
 commute, for each $0 \leqslant i \leqslant d-2$ and $A \in \mc{P}([i])_{\not =\varnothing}$. We refer to the objects $A^{(i)}$ as dummy vertices.
\item We write $I^{\geqslant i}(d)$ for the full subcategory of $I(d)$ on the objects $A^k$ and $A^{(k)}$ with $k \geqslant i$. 
\end{enumerate}
\end{defn}

In order to illustrate Definition~\ref{poset}, we draw $I_{-}(d)$ for $1 \leqslant d \leqslant 3$. 

\begin{ex}[$I_{-}(1)$]\label{ex:it1}
\[
\xymatrix{
1^1 \ar[d] \\
10^1 \ar@{<-}@[OIblue][r] & 0^0
}
\]
\end{ex}

\begin{ex}[$I_{-}(2)$]\label{ex:it2}
\[
\xymatrix@=1em{
& 21^2 \ar[dd] \ar@{<-}@[OIblue][rr] && 1^1 \ar[dd] \\
 2^2 \ar[ur] \ar[dd] & & & & \\
 & 210^2 \ar@{<-}@[OIblue][rr] && 10^1 \ar@{<-}@[OIblue][rr] && 0^0 \\
 20^2  \ar[ur] \ar@{<-}@[OIblue][rr] && 0^1 \ar[ur]
}
\]
\end{ex}

\begin{ex}[$I_{-}(3)$]\label{ex:it3}\leavevmode
\[
\xymatrix@=0.5em{
                          &  & 31^3 \ar@{<-}@[OIblue]@/_/[rrrrd]  \ar[ddd] \ar[rrr] &                &  & 321^3 \ar@{<-}@/^/@[OIblue][rrrrd] \ar[ddd]  &  &  & & & & & \\
                           &  &                           &                &  &                &   1^2 \ar[ddd] \ar[rrr] & & & 21^2 \ar@{<-}@[OIblue][dr] \ar[ddd] & & & \\
3^3 \ar[rruu] \ar[ddd] \ar[rrr] &  &                           &     32^3  \ar@{<-}@/^/@[OIblue][rrrrd] \ar[ddd]  \ar[rruu]    &  &          &      &  & & & 1^1 \ar[ddd] & & \\
                           &  & 310^3 \ar[rrr]  \ar@{<-}@/_/@[OIblue][rrrrd]        &                &  & 3210^3 \ar@{<-}@/^/@[OIblue][rrrrd]&  & 2^2  \ar[ddd] \ar[rruu] & & & & & \\
                           &  &                           &                &  &                & 10^2  \ar[rrr] & & & 210^2 \ar@{<-}@[OIblue][dr] & & & \\
30^3 \ar@{<-}@/_/@[OIblue][rrrrd] \ar[rruu] \ar[rrr] &  &                           & 320^3 \ar[rruu] \ar@{<-}@/^/@[OIblue][rrrrd]&  &                &  & & & & 10^1 \ar@{<-}@[OIblue][drr] & & & \\
& & & &  0^2 \ar[rruu] \ar[rrr] & & & 20^2 \ar@{<-}@[OIblue][dr] \ar[rruu] & & & & &  0^0 \\
& & & & & & & & 0^1 \ar[rruu] & & & & \\
}
\]
\end{ex}

\begin{rem}
To obtain a picture of $I(d)$ from $I_{-}(d)$ one adds dummy vertices. For example, in the case $d=1$ there are no dummy vertices to be added,
and in the case $d=2$, there is only one dummy vertex $0^{(0)}$ to be
added. Since we will only care about full subcategories on which the
objects indexed on these dummy vertices are zero, we henceforth omit these dummy vertices from pictures. 
\end{rem}

With this notation in hand, we may describe objects in $\biggerthan{i}{R}$. 
\begin{rem}\label{objectsbiggerthan}
An object $M$ in $\biggerthan{i}{R}$ consists of the data of:
\begin{enumerate}[label=(\roman*)]
\item for each $A^k \in I^{\geqslant i}(d)$, an $R(A)$-module $M(A^k)$;
\item for each oplax restriction $\ext{i}{k}{A}\colon A^k \to (A \cup i)^k$, an $R(A)$-module map \[M(\ext{i}{k}{A})\colon M(A^k) \to \mrm{res}^{A \cup i}_A M((A \cup i)^k);\]
\item for each lax restriction $\res{k}{A}\colon (A \backslash k)^{k-1} \to A^k$, an $R(A\backslash k)$-module map \[M(\res{k}{A})\colon \mrm{res}^A_{A\backslash k} M(A^k) \to M((A\backslash k)^{k-1});\]
\item for each $A^{(k)} \in I^{\geqslant i}(d)$, an $R(A)$-module $M(A^{(k)})$ together with $R(A)$-module maps $M(A^{i+1}) \to M(A^{(i)})$ and $M(A^{(i)}) \to M(A^i)$.
\end{enumerate}
All this data can be assembled into a diagram of shape $I^{\geqslant i}(d)$ which is required to commute. The data in (iv) is there to record null homotopies, as we will see in \cref{biggerthanfdefn}. We note that associated to each structure map $M(\ext{i}{k}{A})\colon M(A^k) \to \mrm{res}^{A \cup i}_A M((A\cup i)^k)$, there is by adjunction a map \[M(\ext{i}{k}{A})^\flat\colon \mrm{ext}_A^{A \cup i}M(A^k) \to M(A \cup i)^k\] which we call an \emph{adjoint structure map}. Note that points (ii) and (iii) justify the terminology of oplax and lax restrictions used in \cref{poset}.
\end{rem}

\begin{rem}\label{rem:justifynotation}
It is worth justifying the notation of \cref{Cgeq} in terms of filtration degree. The category $\biggerthan{i}{R}$ consists of all objects of filtration degree ${\geqslant}i$, and the functor $u_{=i}\colon \biggerthan{i}{R} \to \equals{R^{\leqslant i}}$ picks out the objects of filtration degree precisely $i$. The functor $u_{\geqslant i+1}\colon \biggerthan{i}{R} \to \biggerthan{i+1}{R}$ restricts to the objects of filtration degree ${\geqslant} i+1$, and the functor $g_i^{i+1}\colon \biggerthan{i}{R} \to \glueplus{R^{\leqslant i+1}}$ restricts to the objects of filtration degree $i$ and $i+1$.
\end{rem}

\subsection{Cofibres between layers}
We may now construct a cofibre functor $\biggerthan{i}{R} \to \biggerthan{i-1}{R}$ for each $1 \leqslant i \leqslant d-1$. 
\begin{con}\label{cofibregeq}
In order to construct a cofibre functor $\mrm{cof}_{\geqslant i}\colon\biggerthan{i}{R} \to \biggerthan{i-1}{R}$, since $\biggerthan{i-1}{R}$ is defined as a pullback (see \cref{Cgeq}) it suffices to construct compatible maps $\biggerthan{i}{R} \to \biggerthan{i}{R}$ and $\biggerthan{i}{R} \to \glueplus{R^{\leqslant i}}$. For the former we take the identity functor, and for the latter we take the composite \[\biggerthan{i}{R} \xrightarrow{u_{=i}} \equals{R^{\leqslant i}} \xrightarrow{\mrm{cof}_+} \glueplus{R^{\leqslant i}}\] where the maps are given in \cref{Cgeq} and \cref{con:enhancedcofibre} respectively.
\end{con}

\begin{ex}\label{ex:2functor}
When $d=2$, the functor $\mrm{cof}_{\geqslant 1} \colon \biggerthan{1}{R} \to \biggerthan{0}{R}$ is:
\[
\begin{gathered}
\xymatrix@R=1em@C=-1.0em{
& M(21^2) \ar[dd] \ar@{->}@[OIblue][rr] && M(1^1) \ar[dd] \\
 M(2^2) \ar[ur] \ar[dd] & & & & \\
 & M(210^2) \ar@{->}@[OIblue][rr] && M(10^1)\\
 M(20^2)  \ar[ur] \ar@{->}@[OIblue][rr] && M(0^1) \ar[ur]
}
\end{gathered}
\xlongrightarrow{\mrm{cof}_{\geqslant 1}}
\begin{gathered}
\xymatrix@R=1em@C=-0.25em{
& M(21^2) \ar[dd] \ar@{->}@[OIblue][rr] && M(1^1) \ar[dd] \\
 M(2^2) \ar[ur] \ar[dd] & & & & \\
 & M(210^2) \ar@{->}@[OIblue][rr] && M(10^1) \ar@{->}@[OIblue][rrr] && &\mrm{cof}(M(0^1) \to M(10^1)) \\
 M(20^2)  \ar[ur] \ar@{->}@[OIblue][rr] && M(0^1) \ar[ur]
}
\end{gathered}
\]
\end{ex}

Recall the description of objects of $\biggerthan{i}{R}$ as given in \cref{objectsbiggerthan}. The following definition enhances \cref{glueplusf} to our case of interest, by describing a certain full subcategory of cofibre sequences.
\begin{defn}\label{biggerthanfdefn}
	Let $R\colon \mc{P}([d]) \to \msf{CAlg}(\C)$ be a diagram of rings. We set $\biggerthanf{d-1}{R} = \biggerthan{d-1}{R}$. For $0 \leqslant i \leqslant d-2$, we define $\biggerthanf{i}{R}$ to be the full subcategory of $\biggerthan{i}{R}$ on those objects $M$ for which $M(A^{\s{k}}) \simeq 0$ and the commutative square 
	\[
	\begin{tikzcd}[column sep=5em, row sep=3em]
		M(A^{k+1}) \arrow[r] \ar[d] & \mrm{res}_{A}^{A \cup k+1}M((A \cup k+1)^{k+1}) \arrow[d]  \\
		0\simeq M(A^{\s{k}}) \arrow[r] & M(A^{k})
	\end{tikzcd}
	\]
	induced by (\ref{square-I_+(A,s)}), is a pushout, for all $A \in \mc{P}([k])_{\neq\varnothing}$ and $k \geqslant i$. 
	\end{defn}
	
\begin{rem}\label{fibreaspullback}
By the definition of $\biggerthanf{i}{R}$ together with \cref{Cgeq}, we see that for each $1 \leqslant i \leqslant d-1$, the diagram
\[
\begin{tikzcd}[column sep=1.5cm]
\biggerthanf{i-1}{R} \ar[r, "u_{\geqslant i}"] \ar[d, "g_{i-1}^{i}"'] & \biggerthanf{i}{R} \ar[d, "u_{=i}"] \\
\glueplus{R^{\leqslant i}}_\mrm{c} \ar[r, "v"'] & \equals{R^{\leqslant i}}
\end{tikzcd}
\]
is a pullback.
\end{rem}

\begin{rem}
	The vertices $M(A^{\s{k}})$ are, by construction, always zero
        in $\biggerthanf{i}{R}$, and therefore for diagramatical
        simplicity we will omit them from diagrams of objects in $\biggerthanf{i}{R}$. Informally, in the context of the above definition we will say that the composite \[M(A^{k+1}) \to \mrm{res}_A^{A \cup k+1}M((A \cup k+1)^{k+1}) \to M(A^k)\] is a cofibre sequence.
\end{rem}

\begin{ex}\leavevmode
	In the following figure we depict
        $\biggerthanf{0}{R}$ in the case $d=3$. The red and green composite maps indicate the required cofibre sequences. As
        mentioned above, we omit the zero vertices. A dotted arrow indicates an oplax structure map, whereas a solid arrow indicates a lax structure map. 
	
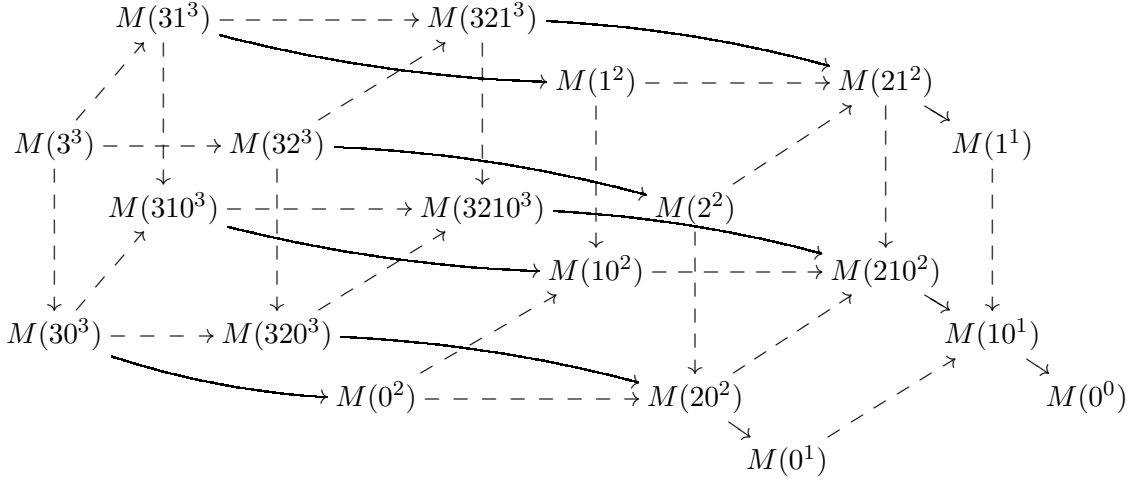
\[
	\xymatrix@R=0.5em@C=-0.5em{
		&  & M(31^3) \ar@/_/[rrrrd]  \ar@{-->}[ddd] \ar@{-->}[rrr] &                &  & M(321^3) \ar@/^/[rrrrd] \ar@{-->}[ddd]  &  &  & & & & & \\
		&  &                           &                &  &                &   M(1^2) \ar@{-->}[ddd] \ar@{-->}@[OIred][rrr] & & & M(21^2) \ar@[OIred][dr] \ar@{-->}[ddd] & & & \\
		M(3^3) \ar@{-->}[rruu] \ar@{-->}[ddd] \ar@{-->}[rrr] &  &                           &     M(32^3)  \ar@/^/[rrrrd] \ar@{-->}[ddd]  \ar@{-->}[rruu]    &  &          &      &  & & & M(1^1) \ar@{-->}[ddd] & & \\
		&  & M(310^3) \ar@{-->}[rrr]  \ar@/_/[rrrrd]        &                &  & M(3210^3) \ar@/^/[rrrrd]&  & M(2^2)  \ar@{-->}[ddd] \ar@{-->}[rruu] & & & & & \\
		&  &                           &                &  &                & M(10^2)  \ar@{-->}@[OIred][rrr] & & & M(210^2) \ar@[OIred][dr] & & & \\
		M(30^3) \ar@/_/[rrrrd] \ar@{-->}[rruu] \ar@{-->}[rrr] &  &                           & M(320^3) \ar@{-->}[rruu] \ar@/^/[rrrrd]&  &                &  & & & & M(10^1) \ar@[OIgreen][drr] & & & \\
		& & & &  M(0^2) \ar@{-->}[rruu] \ar@{-->}@[OIred][rrr] & & & M(20^2) \ar@[OIred][dr] \ar@{-->}[rruu] & & & & &  M(0^0) \\
		& & & & & & & & M(0^1) \ar@{-->}@[OIgreen][rruu] & & & & \\
	}
	\]
\end{ex}

We can now prove the equivalences of categories between $\equals{R}$ and the layers $\biggerthanf{i}{R}$, as promised as the start of the section.
\begin{prop}\label{adelictotors2}
	Let $R\colon \mc{P}([d]) \to \msf{CAlg}(\C)$ be a diagram of rings. There is an equivalence of categories 
	\[
	\begin{tikzcd}
		\equals{R} \ar[r, yshift=0.2cm, "\mrm{cof}^\msf{m}"] \ar[r, phantom, "\simeq" description] & \biggerthan{d-1}{R}. \ar[l, yshift=-0.2cm, "\mrm{fib}^\msf{m}"] 
	\end{tikzcd}
	\]
\end{prop}
\begin{proof}
	Recall that $\biggerthan{d-1}{R} = \equalsmix{R}$. Therefore this statement is a special case of \cref{puncLtoR}.
\end{proof}

\begin{prop}\label{betweenlayers}
	Let $R\colon \mc{P}([d]) \to \msf{CAlg}(\C)$ be a diagram of rings, and $1 \leqslant i \leqslant d-1$. The cofibre functor $\mrm{cof}_{\geqslant i}\colon\biggerthan{i}{R} \to \biggerthan{i-1}{R}$ from \cref{cofibregeq} restricts to an equivalence \[\mrm{cof}_{\geqslant i}\colon \biggerthanf{i}{R} \xrightarrow{\sim} \biggerthanf{i-1}{R}.\]
\end{prop}
\begin{proof}
	By \cref{punccofthm}, the bottom map $v$ in the pullback diagram of \cref{fibreaspullback} is an equivalence with quasi-inverse $\mrm{cof}_+$. Hence the top map in the pullback, $u_{\geqslant i}\colon \biggerthanf{i-1}{R} \to \biggerthanf{i}{R}$, is also an equivalence. By the pullback square in \cref{fibreaspullback}, we see that the quasi-inverse to $u_{\geqslant i}$ is the functor uniquely determined by the functors $\mrm{id}\colon \biggerthanf{i}{R} \to \biggerthanf{i}{R}$, and the functor $\mrm{cof}_+\circ u_{={i}}\colon \biggerthanf{i}{R} \to \glueplus{R^{\leqslant i}}_\mrm{c}$, since $\mrm{cof}_+$ is the quasi-inverse of $v$. This functor is nothing other than $\mrm{cof}_{\geqslant i}$ as described in \cref{cofibregeq}.
\end{proof}

Combining the above results we obtain the following.
\begin{cor}\label{fromdto0}
	Let $R\colon \mc{P}([d]) \to \msf{CAlg}(\C)$ be a diagram of rings. There is an equivalence of categories $\msf{L}\colon \equals{R} \xrightarrow{\sim} \biggerthanf{0}{R}$ where $\msf{L} := \mrm{cof}_{\geqslant 1} \circ \mrm{cof}_{\geqslant 2} \circ \cdots \circ \mrm{cof}_{\geqslant d-1} \circ \mrm{cof}^\msf{m}$. \qed 
\end{cor}

\section{The torsion model}\label{sec:thetorsionmodel}
In this section we prove the main theorem of this paper, which shows the existence of a torsion model for a tensor-triangulated category $\C$ satisfying \cref{hyp:hyp}. In \cref{sec:global}, we proved that given a diagram of rings $R$, there is an equivalence of categories $\equals{R} \simeq \biggerthanf{0}{R}$, and in \cref{sec:adelic}, we constructed a fully faithful functor $\C \hookrightarrow \equals{\1_\ad^\X}$. As such, combining these two results yields the following.
\begin{prop}\label{cor:functors}
	Let $(\X,\alpha)$ be assembly data for $\C$. There is an equivalence of categories
	\[\xymatrix{
\equals{\1_\ad^\X} \ar@<0.7ex>[r]^-{\msf{L}} \ar@{}[r]|-{\simeq} \ar@<-0.7ex>@{<-}[r]_-{\msf{R}} & \biggerthanf{0}{\1_\ad^\X}
}\]
and a fully faithful functor  $(-)_\tors^\X\colon \C \hookrightarrow \biggerthanf{0}{\1_\ad^\X}$ defined by $(-)_\tors^\X = \msf{L}(\1_\ad^\X \otimes -)$.
\end{prop}
\begin{proof}
	There is a fully faithful functor $\1_\ad^\X \otimes -\colon \C \to \equals{\1_\ad^\X}$ by \cref{1ad-fullyfaithful} which we may compose with the equivalence from \cref{fromdto0}.
\end{proof}

\begin{rem}
Recall from \cref{fromdto0} that the functor $\msf{L}$ is given by taking iterated cofibres. 
\end{rem}

\begin{rem}\label{rem:1tors}
	We note that $X_\tors^\X \simeq \1_\tors^\X \otimes X$ for all $X \in \C$ by \cref{thm-adelicm} and the fact that cofibres commute with tensors. We will describe the functor $(-)_\tors^\X$ in more detail in the following subsection.
\end{rem}

\subsection{Defining the model}
We have produced a fully faithful functor $(-)_\tors^\X\colon \C \hookrightarrow \biggerthanf{0}{\1_\ad^\X}$, so to provide a model for $\C$ we will need to identify the image of $(-)_\tors^\X$. We now define the torsion model, and the rest of this section will be dedicated to proving that this is indeed the image. Recall the notation for objects of $\biggerthanf{0}{\1_\ad}$ from \cref{objectsbiggerthan}.
\begin{defn}\label{defn:torsionmodel}
Let $(\X, \alpha)$ be assembly data for $\C$. The \emph{torsion model} (based on $(\X,\alpha)$) is the full subcategory $\ct$ of $\biggerthanf{0}{\1_\ad^\X}$ on the objects $M$ satisfying the following two properties:
\begin{enumerate}[label=(\roman*)]
\item for each oplax restriction map $\ext{i}{k}{A}\colon A^k \to (A \cup i)^k$ in $I(d)$, the adjoint structure map \[M(\ext{i}{k}{A})^\flat\colon \mrm{ext}_A^{A \cup i} M(A^k) \to M((A \cup i)^k)\] is an equivalence;
\item\label{torsioncondition} for all $0 \leqslant i \leqslant d$, $M(i^i)$ is $\Gammale{i}$-torsion.
\end{enumerate}
\end{defn}

\begin{rem}
We note that \ref{torsioncondition} in the previous definition is equivalent to $M(i^i)$ having mono-dimensional support $i$. This is because $M(i^i)$ is a $\1_\ad^\X(i)$-module, and hence is $\Lge{i}$-local.
\end{rem}

\begin{ex}\label{ex:2dhomotopy}
Let us illustrate the previous definition with an example in the two-dimensional case before continuing. We consider a general $M \in \biggerthan{0}{\1_\ad}$ which may be depicted as
\[
\xymatrix@=1em{
& M(21^2) \ar@{-->}[dd] \ar[rr] && M(1^1) \ar@{-->}[dd] \\
 M(2^2) \ar@{-->}[ur] \ar@{-->}[dd] & & & & \\
 & M(210^2) \ar[rr] && M(10^1) \ar[rr] && M(0^0) \\
 M(20^2)  \ar@{-->}[ur] \ar[rr] && M(0^1) \ar@{-->}[ur]
}
\]

%\[
%\begin{tikzpicture}[baseline= (a).base]
%\node[scale=0.75] (a) at (0,0){
%\begin{tikzcd}[column sep=0ex, row sep=1ex]
%                          &  &  \textcolor{black}{M(21^2)} \arrow[ddd, dashed] \arrow[rrr, color=gB] &                &  & \textcolor{black}{M(1^1)} \arrow[ddd, color=black, dashed]  &  &   \\
%                           &  &                           &                &  &                &  & \textcolor{black}{M(0^0)} \\
% \textcolor{black}{M(2^2)} \arrow[rruu, dashed] \arrow[ddd, dashed] &  &                           &                &  &                &  &   \\
%                           &  &  \textcolor{black}{M(210^2)} \arrow[rrr, color=gB]             &                &  & \textcolor{black}{M(10^1)} \arrow[rruu, color=gB] &  &   \\
%                           &  &                           &                &  &                &  &   \\
% \textcolor{black}{M(20^2)} \arrow[rruu, dashed] \arrow[rrr, color=gB] &  &                           & \textcolor{black}{M(0^1)} \arrow[rruu, color=black, dashed] &  &                &  &  
%\end{tikzcd}};
%\end{tikzpicture}\]
Recall that the solid maps represent maps after restricting scalars in the domain, and the dashed maps represent maps after restricting scalars in the codomain. Then $M \in \ct$ if and only if:
\begin{itemize}
\item[(i)] the sequence $M(0^1) \to M(10^1) \to M(0^0)$ is a cofibre sequence (keeping in mind that the dummy vertex $M(0^{(0)})$ is not displayed in the above diagram);
\item[(ii)] the adjoint structure maps corresponding to the dashed arrows in the above diagram are equivalences (see \cref{objectsbiggerthan} for details on adjoint structure maps);
\item[(iii)] $M(2^2)$ (resp., $M(1^1), M(0^0)$) has mono-dimensional support $2$ (resp., mono-dimensional support 1, 0).
%\begin{itemize}
%\setlength\itemsep{0.5em}
%	\item $\1_\ad(21) \otimes_{\1_\ad(2)} M(2^2) \xrightarrow{\sim} M(21^2)$;
%	\item $\1_\ad (20) \otimes_{\1_\ad (2)} M(2^2) \xrightarrow{\sim} M(20^2)$;
%	\item $\1_\ad (210) \otimes_{\1_\ad (2)} M(2^2) \xrightarrow{\sim} M(210^2)$;
%	\item $\1_\ad (210) \otimes_{\1_\ad (21)} M(21^2) \xrightarrow{\sim} M(210^2)$; 
%	\item $\1_\ad (210) \otimes_{\1_\ad (20)} M(20^2) \xrightarrow{\sim} M(210^2)$; 
%	\item $\1_\ad(10) \otimes_{\1_\ad(1)} M(1^1) \xrightarrow{\sim} M(10^1)$;
%	\item $\1_\ad(10) \otimes_{\1_\ad(0)} M(0^1) \xrightarrow{\sim} M(10^1)$.
%\end{itemize}
\end{itemize}
\end{ex}

We now give a concrete description of the functor $(-)^\X_\tors\colon \C \to \cf{0}$ on objects. Recall from \cref{rem:1tors} that $X_\tors^\X \simeq \1_\tors^\X \otimes X$ for all $X \in \C$. As such, it is enough to understand $\1_\tors^\X$. 

\begin{prop}\label{identify1tors}
For all $A^i \in I_{-}(d)$, we have \[\1_\tors^\X(A^i) \simeq \Sigma^{d-i}\Gammale{i}\1_\ad^\X(A).\] In particular, $\1_\tors^\X(i^i) \simeq \Sigma^{d-i}\Elr{i} \otimes \1_\ad^\X(i)$.
\end{prop}

\begin{proof}
We shall proceed via downward induction on $i$. In the case $i=d$ we have $\1_\tors^\X(A^d) = \1_\ad^\X(A)$ as required, using the fact that $\Gammale{d}$ is the identity. 
For $i = d-1$, by definition of $\cf{d-1}$ we have
\[
\1_{\tors}^\X(A^{d-1}) \simeq \mrm{cof}(\1_\ad^\X (A) \to \1_\ad^\X (A \cup d)) = \mrm{cof}(\1_\ad^\X(A) \to \bigoplus_{x_d\in\X}L^\X_{x_d} \1_\ad^\X(A)),
\]
where the sum in the final term is finite as $\Spc(\C^\omega)$ has finitely many points of maximal dimension since it is Noetherian. We note that $\bigoplus_{x_d} L_{x_d}^\X\1_\ad^\X(A) \simeq \Lge{d}\1_\ad^\X(A)$ by~\cref{maxL}\ref{maxLsplit}, so $\1_\tors^\X(A^{d-1})$ is equivalent to $\Sigma\Gammale{d-1}\1_\ad^\X(A)$ as required.

Now suppose that the result holds for all $j > i$, where we can assume $i \neq d,~d-1$. We have a cofibre sequence
\[
\1_\tors^\X(A^{i+1}) \longrightarrow \1_\tors^\X ((A \cup \{i+1\})^{i+1}) \longrightarrow \1_\tors^\X(A^i)
\]
by definition of $\cf{i}$. By the induction hypothesis we know that this cofibre sequence is equivalent to
\[
\Sigma^{d-i-1} \Gammale{i+1} \1_\ad^\X(A) \longrightarrow \Sigma^{d-i-1} \Gammale{i+1} \1_\ad^\X (A \cup \{i+1\}) \longrightarrow \1_\tors^\X(A^i).
\]
Simplifying the middle term in the sequence, we have
\begin{align*}
\Sigma^{d-i-1} \Gammale{i+1} \1_\ad^\X (A \cup \{i+1\}) &= \Sigma^{d-i-1} \Gammale{i+1} \prod_{x_{i+1}} L^\X_{x_{i+1}}\1_\ad^\X (A) \\
&\simeq \Sigma^{d-i-1} \Elr{i+1} \otimes \1_\ad^\X(A)
\end{align*}
using \cref{epointyproduct}\ref{epointy1}.
Using the above cofibre sequence, we conclude that \[\1_\tors^\X(A^i) \simeq \Sigma^{d-i} \Gammale{i} \Gammale{i+1} \1_\ad^\X(A) \simeq \Sigma^{d-i} \Gammale{i}  \1_\ad^\X(A)\] and the result follows.
\end{proof}

\begin{cor}\label{1torshasproperties}
The object $\1_\tors^\X \in \cf{0}$ lies in the torsion model $\ct$.
\end{cor}
\begin{proof}
This is immediate from \cref{identify1tors}.
\end{proof}

\subsection{Establishing the torsion model}
In this subsection we prove that $\C$ is equivalent to its torsion model $\ct$. In order to do this, we firstly prove the following two lemmas.
%Recall that we denote the extension (resp., restriction) of scalars functor along $\1_\ad^\X(A) \to \1_\ad^\X(B)$ by $\mrm{ext}_A^B$ (resp., $\mrm{res}_A^B$).  

\begin{lem}\label{lem:gen-local}
Let $N$ be a $\1_{\ad}^\X(A \cup d)$-module. Then $N$ is $L_{\geqslant d}$-local and $\mrm{ext}_A^{A \cup d} N \simeq N$.
\end{lem}

\begin{proof}
As $\1_{\ad}^\X(A \cup d)$ is $L_{\geqslant d}$-local, any module over it is $L_{\geqslant d}$-local since $L_{\geqslant d}$ is smashing. Therefore we have \[\mrm{ext}_A^{A \cup d} N\simeq \Lge{d}\1_\ad^\X(A) \otimes_{\1_\ad^\X(A)} N  \simeq L_{\geqslant d} N \simeq N\] where the first equivalence holds by \cref{maxL}\ref{maxLsplit}, and the last equivalence by the above observation. 
\end{proof}

\begin{lem}\label{extendingpurestrata}
Let $i < d$ and $A \in \mc{P}([d{-}1])$ be such that $i \in A$. If $X \in \C$ has mono-dimensional support $i$ for some $i < d$, then $\mrm{ext}_i^{A \cup d} X \simeq 0$.
\end{lem}

\begin{proof}
As $X$ has mono-dimensional support $i$ we have $X \simeq \1_\ad^\X(i) \otimes X$ by \cref{epointyproduct}\ref{epointy2} and so $\mrm{ext}_{i}^{A \cup d}X \simeq \1_\ad^\X(A \cup d) \otimes X $ which still has mono-dimensional support $i$. 
By \cref{lem:gen-local}, $\mrm{ext}_{i}^{A \cup d}X  $ is also $L_{\geqslant d}$-local, thus it must be zero.
\end{proof}

With these lemmas, we are now ready to prove our main theorem.
\begin{thm}\label{torsionmodel}
Let $(\X,\alpha)$ be assembly data for $\C$. The functor $(-)_\tors^\X$ induces an equivalence $\C \xrightarrow{\simeq} \ct$.
\end{thm}
\begin{proof}
We begin by summarising the situation:
\[\xymatrix@C=4em@R=1em{
    \\
    \C  \ar `u[u] `[rr]^{(-)_{\tors}^\X} [rr] \ar@{->}[r]^-{\1_\ad^\X \otimes -}  & \equals{\1_\ad^\X} \ar@<0.7ex>[r]^-{\msf{L}} \ar@{}[r]|-{\simeq} \ar@<-0.7ex>@{<-}[r]_-{\msf{R}} & \cf{0} \\ \\
    \C \ar@{=}[uu] \ar@<0.7ex>[r]^-{\1_\ad^\X \otimes -} \ar@{}[r]|-{\simeq} \ar@<-0.7ex>@{<-}[r]_-{\lim}  & \C^\X_\ad  \ar@{^(->}[uu] & \ct\ar@<0.7ex>@{..>}[l]^{\msf{R}} \ar@<-0.7ex>@{<..}[l]_{\msf{L}} \ar@{^(->}[uu] 
}\]
see \Cref{thm-adelicm} and \Cref{cor:functors}. The bottom right entry of the diagram is the torsion model which we want to prove is equivalent to $\C$. The bottom left equivalence is the equivalence to the adelic model of \cref{thm-adelicm}, so it suffices to construct an equivalence between $\C^\X_\ad$ and $\ct$. If we include these both into the ambient categories of which they are full subcategories then we have the equivalence of categories shown on the top right of the diagram by \cref{cor:functors}. Therefore to prove the statement, it suffices to show that $\msf{L}$ and $\msf{R}$ restrict to functors $\C^\X_\ad \to \ct$ and $\ct \to \C^\X_\ad$ respectively, i.e., that the dotted functors in the above diagram exist. 

Firstly, we show that $\msf{L}$ restricts. Any object $M \in \C^\X_\ad$ is equivalent to $\1_\ad^\X \otimes \mrm{lim}(M)$ by \cref{thm-adelicm}. Therefore 
\begin{align*}
\msf{L}M&\simeq \msf{L}(\1_\ad^\X \otimes \mrm{lim}(M)) & \\
&\simeq (\mrm{lim}(M))_\tors^\X &\text{by definition of $(-)_\tors^\X$} \\
&\simeq \1_\tors^\X \otimes \mrm{lim}(M) &\text{by \cref{rem:1tors}} 
\end{align*} 
which is in $\ct$ by \cref{1torshasproperties}.

We now show that $\msf{R}$ restricts. Recall from \cref{cor:functors} that $\msf{R}$ is defined to be the composite \[\msf{R}\colon \cf{0} \xrightarrow{u_{\geqslant 1}} \cf{1} \to \cdots \to \cf{d-2} \xrightarrow{u_{\geqslant d-1}} \cf{d-1} = \equalsmix{\1^\X_\ad} \xrightarrow{\mrm{fib}^\msf{m}} \equals{\1^\X_\ad}\] where $u_{\geqslant i}$ forgets the part of the diagram of filtration degree $i-1$ as in \cref{rem:justifynotation}, and $\mrm{fib}^\msf{m}$ takes the fibre of maps indexed on $(A \cup d)^d \to A^{d-1}$ as in \cref{puncLtoR}. As such, we see that $\msf{R}(M)(A) = M(A^d)$ if $d \in A$. 

Suppose that $M \in \ct$. To prove that $\msf{R}(M) \in \C^\X_\ad$ we must show that
\[\mrm{ext}_A^{A \cup i} \msf{R}(M)(A) \to \msf{R}(M)(A \cup i)\]
is an equivalence for all $i \not\in A$ as the maps $A \to A \cup i$ generate all maps in $\mc{P}([d])_{\neq\varnothing}$ under composition. We prove this by considering three cases:
\begin{enumerate}[label=(\arabic*)]
\item\label{case1} $d \in A$;
\item\label{case2} $d \not\in A$ and $i \neq d$;
\item\label{case3} $d \not\in A$ and $i = d$.
\end{enumerate}
In case~\ref{case1} this follows from the observation that $\msf{R}(M)(A) = M(A^d)$ and $\msf{R}(M)(A \cup i) = M((A \cup i)^d)$, together with the assumption that $M \in \ct$. 

For the remaining cases we recall that we have a fibre sequence
\begin{equation}\label{fibreR}
\msf{R}(M)(A) \to M((A \cup d)^d) \to M(A^{d-1})
\end{equation}
of $\1_\ad^\X(A)$-modules by the definition of $\msf{R}$. 

For case~\ref{case2}, that is, when $d \not\in A$ and $i \neq d$, by definition of $\msf{R}$ (see (\ref{fibreR})), the map $\mrm{ext}_A^{A \cup i}\msf{R}(M)(A) \to \msf{R}(M)(A \cup i)$ is defined to be the induced map on fibres as shown in the diagram
\[
\begin{tikzcd}
\mrm{ext}_A^{A\cup i}\msf{R}(M)(A) \ar[r] \ar[d] & \mrm{ext}_A^{A\cup i}M((A\cup d)^d) \ar[r] \ar[d] & \mrm{ext}_A^{A\cup i}M(A^{d-1}) \ar[d] \\
\msf{R}(M)(A \cup i) \ar[r] & M((A \cup \{i,d\})^d) \ar[r] & M((A \cup i)^{d-1}).
\end{tikzcd}
\]
The middle vertical map is an equivalence by case~\ref{case1}, and the rightmost vertical map is an equivalence since $M \in \ct$. Therefore the leftmost vertical map is also an equivalence which proves case~\ref{case2}. 

For case~\ref{case3}, we fix $A$ with $d \not\in A$ and $i = d$. By applying $\mrm{ext}_A^{A \cup d}$ to (\ref{fibreR}) we obtain a fibre sequence
\begin{equation}
\mrm{ext}_A^{A \cup d}\msf{R}(M)(A) \to M((A \cup d)^d) \to \mrm{ext}_A^{A \cup d}M(A^{d-1})
\end{equation}
of $\1_\ad^\X(A \cup d)$-modules, as $\mrm{ext}_A^{A \cup d} M((A \cup d)^d) \simeq M((A \cup d)^d)$ by \cref{lem:gen-local}.
Therefore, for this case, it suffices to prove that $\mrm{ext}_A^{A \cup d} M(A^{d-1}) \simeq 0$, since $\msf{R}(M)(A \cup d) = M((A \cup d)^d)$.

Let $j(A)$ be the smallest non-negative integer such that $d-j(A)-1 \in A$. We will prove that \[\mrm{ext}_A^{A \cup d} M(A^{d-j(A)+k}) \simeq 0\] for all $-1 \leqslant k \leqslant j(A)-1$ by induction on $k$. For $k=-1$, we have $d-j(A)+k = d-j(A)-1 \in A$. Therefore \[M(A^{d-j(A)-1}) \simeq \mrm{ext}_{d-j(A)-1}^{A}M((d-j(A)-1)^{d-j(A)-1})\] where $M((d-j(A)-1)^{d-j(A)-1})$ has mono-dimensional support $d-j(A)-1$ as $M \in \ct$. As such the base case follows from \cref{extendingpurestrata}.

As the inductive hypothesis, we now suppose that $\mrm{ext}_A^{A \cup d}M(A^{d-j(A)+k-1}) \simeq 0$. In order to prove that $\mrm{ext}_A^{A \cup d}M(A^{d-j(A)+k}) \simeq 0$, we consider two cases. 

Firstly, if $d-j(A)+k \in A$, then \[\mrm{ext}_A^{A \cup d}M(A^{d-j(A)+k}) \simeq \mrm{ext}_A^{A\cup d}\mrm{ext}_{d-j(A)+k}^{A \cup d-j(A)+k}M((d-j(A)+k)^{d-j(A)+k})\] and $M((d-j(A)+k)^{d-j(A)+k})$ has mono-dimensional support $d-j(A)+k$ as $M \in \ct$. Therefore in this case, $\mrm{ext}_A^{A \cup d}M(A^{d-j(A)+k-1}) \simeq 0$ by \cref{extendingpurestrata}. 

Secondly, if $d-j(A)+k \not\in A$, then there is a fibre sequence
\[\mrm{ext}_A^{A \cup d}M(A^{d-j(A)+k}) \to \mrm{ext}_A^{A \cup d}M((A \cup d-j(A)+k)^{d-j(A)+k}) \to \mrm{ext}_A^{A \cup d}M(A^{d-j(A)+k-1})\] as $M\in \ct$. 
By a similar argument as above the middle term is extended from an object with mono-dimensional support $d-j(A)+k$ as $M \in \ct$, and hence is zero by \cref{extendingpurestrata}. We also have $\mrm{ext}_A^{A \cup d} M(A^{d-j(A)+k-1}) \simeq 0$ by the inductive hypothesis. As such, we conclude that $\mrm{ext}_A^{A \cup d} M(A^{d-j(A)+k}) \simeq 0$ which completes the proof by induction. Taking $k = j(A)-1$ in the proved inductive claim completes the proof of  case~\ref{case3}, and $\msf{R}(M) \in \C^\X_\ad$ as required.
\end{proof}

\section{Examples}\label{sec:examples}

We will now illuminate the torsion model with some examples. 
\subsection{Derived categories of commutative Noetherian rings}
Let $R$ be a commutative Noetherian ring of finite Krull dimension. The derived category $\msf{D}(R)$ satisfies \cref{hyp:hyp} as $\Spc(\msf{D}(R)^\omega)$ is homeomorphic to $\mrm{Spec}(R)$. For any prime $\p \in \Spc(\msf{D}(R)^\omega)$ we can identify the associated $\p$-torsion, $\p$-localization functor, and $\p$-completion functors as the derived $\p$-torsion functor, the usual localization at $\p$, and the derived $\p$-completion functor respectively, see~\cite[\S 5.1]{torsion1} for more details. For more recollections on these functors see~\cite{DwyerGreenlees02, adicL} for instance.

Let us make the adelic rings, and the torsion model for $\msf{D}(R)$ explicit when $R$ is a local integral domain of Krull dimension $2$. The case of Krull dimension $1$ has already been described in \cite[\S 9.3]{torsion1}. In this 2-dimensional setting, we have a unique generic point $\g$, and a unique closed point $\m$, together with infinitely many prime ideals $\p$ of height $1$. For concreteness, one might take $R = \Z_{(p)}\llbracket x\rrbracket$. 

There are two obvious candidates for assembly data here, the finest and the coarsest. The corresponding diagrams of adelic rings are illustrated in \cref{fig:D(R)rings}.

\begin{figure}[h]
\xymatrixcolsep{2ex}\xymatrixrowsep{4ex}\xymatrix{
 &  \displaystyle{\prod_{\fp}} (R_\p^\wedge)_\p   \ar[rr] \ar[dd]|\hole&&   \left(\displaystyle{\prod_{\fp}} (R_\p^\wedge)_\p\right)_\g\ar[dd]  &&  & (\Lambda_{\leqslant 1}R)[\m^{-1}]   \ar[rr] \ar[dd]|\hole&&   (\Lambda_{\leqslant 1}R)_\g \ar[dd]  \\
R \ar[rr] \ar[ur] \ar[dd]&& R_\g \ar[dd] \ar[ur]  &&& R \ar[rr] \ar[ur] \ar[dd]&& R_\g \ar[dd] \ar[ur]  &\\
&   \displaystyle{\prod_{\fp}} (R_\m^\wedge)_\p  \ar[rr]|-\hole&&   \left(\displaystyle{\prod_{\fp}} (R_\m^\wedge)_\p\right)_\g && &   (R_\m^\wedge)[\m^{-1}]  \ar[rr]|-\hole&&   (R_\m^\wedge)_\g \\
R_\m^\wedge \ar[rr] \ar[ur] &&    (R_\m^\wedge)_\g \ar[ur] &&&   R_\m^\wedge \ar[rr] \ar[ur] &&   (R_\m^\wedge)_\g \ar[ur]
}
\caption{The adelic rings associated to the finest and coarsest
  assembly data on $\msf{D}(R)$ for a local Noetherian domain $R$ of
  Krull dimension $2$. For an $R$-module $M$, the chain complex
  $M[\m^{-1}]$ has homology given by the \v{C}ech cohomology of
  punctured affine space, and $\Lambdale{1}R$ can be described as the cofibre of $\Hom_R(R_\g, R) \to R$ (as an object of $\msf{D}(R)$ rather than as a commutative algebra object).} \label{fig:D(R)rings}
\end{figure}

%\begin{figure}[h]
%\scalebox{0.9}{
%\xymatrixcolsep{2ex}\xymatrixrowsep{4ex}\xymatrix{
% &  \displaystyle{\prod_{\fp}} (R_\p^\wedge)_\p   \ar[rr] \ar[dd]|\hole&&   \displaystyle{\prod_{\fp}} (R_\p^\wedge)_\p\ar[dd]  \\
%R \ar[rr] \ar[ur] \ar[dd]&& R_\g \ar[dd] \ar[ur]  &\\
%&   \displaystyle{\prod_{\fp}} (R_\m^\wedge)_\p  \ar[rr]|-\hole&&   \left(\displaystyle{\prod_{\fp}} (R_\m^\wedge)_\p\right)_\g \\
%  R_\m^\wedge \ar[rr] \ar[ur] &&    (R_\m^\wedge)_\g \ar[ur]
% }
%\qquad \qquad \xymatrixcolsep{0ex}\xymatrixrowsep{6.4ex}\xymatrix{
%
%
% & (\Lambda_{\leqslant 1}R)[\m^{-1}]   \ar[rr] \ar[dd]|\hole&&   (\Lambda_{\leqslant 1}R)_\g \ar[dd]  \\
%
%
%R \ar[rr] \ar[ur] \ar[dd]&& R_\g \ar[dd] \ar[ur]  &\\
%
%
%&   (R_\m^\wedge)[\m^{-1}]  \ar[rr]|-\hole&&   (R_\m^\wedge)_\g \\
%
%
%  R_\m^\wedge \ar[rr] \ar[ur] &&   (R_\m^\wedge)_\g \ar[ur]
%}}
%\caption{The adelic rings associated to the finest and coarsest assembly data on $\msf{D}(R)$ for a local Noetherian domain $R$ of Krull dimension $2$. For an $R$-module $M$, we note that $M[\m^{-1}]$ is the \v{C}ech cohomology at $\m$, and $\Lambdale{1}M$ can be described as the cofibre of $\Hom_R(R_\g, R) \to R$ (as an object of $\msf{D}(R)$ rather than as a commutative algebra object).} \label{fig:D(R)rings}
%\end{figure}
 
Recall from \cref{ex:2dhomotopy} that an object $M$ in the torsion model $\msf{D}(R)_\mrm{t}^\X$ can be represented by a diagram of the form
\[
\xymatrix@=1em{
& M(21^2) \ar@{-->}[dd] \ar[rr] && M(1^1) \ar@{-->}[dd] \\
 M(2^2) \ar@{-->}[ur] \ar@{-->}[dd] & & & & \\
 & M(210^2) \ar[rr] && M(10^1) \ar[rr] && M(0^0) \\
 M(20^2)  \ar@{-->}[ur] \ar[rr] && M(0^1) \ar@{-->}[ur]
}
\]
where the solid maps represent maps after restricting scalars in the domain, the dashed maps represent maps after restricting scalars in the codomain, each $M(A^k)$ is a $\1_\ad^\X(A)$-module, and the following conditions are satisfied:
\begin{itemize}
\item[(i)] the sequence $M(0^1) \to M(10^1) \to M(0^0)$ is a cofibre sequence;
\item[(ii)] the adjoint structure maps corresponding to the dashed arrows in the above diagram are equivalences;
\item[(iii)] \begin{enumerate}[label=(\alph*)]
\item $M(1^1) \simeq \bigoplus_\p M(\p)$ where each $M(\p)$ is derived $\p$-torsion;
\item $M(0^0)$ is derived $\m$-torsion.
\end{enumerate}
\end{itemize}

\subsection{Chromatic homotopy theory}

In this section we will consider the torsion model in the realm of chromatic homotopy theory. As the category of spectra has an infinite dimensional, non-Noetherian Balmer spectrum, we consider the category $L_{E(n)}\Sp$ of $E(n)$-local spectra at some fixed prime $p$, for some height $n \geqslant 0$. There is a homeomorphism
\begin{align*}
[n] &\xrightarrow{\simeq} \Spc(L_{E(n)}\Sp^\omega) \\
i &\mapsto \mrm{ker}(K(n-i) \otimes -)
\end{align*}
where $[n]$ is equipped with the Alexandroff topology. We emphasize that the chromatic height filtration corresponds to the \emph{co}dimension in the Balmer ordering. To avoid confusion and conflict with the localization $L_i$ in the sense of \cref{Lpnotation}\ref{nota2}, we will never write $L_i$ to mean the localization at $E(i)$, and instead write $L_{E(i)}$.

As the Balmer spectrum is linear, there is only one choice of assembly data. The localization $L_i$ is given by $L_{E(n-i)}$, and the completion $\Lambda_i$ is $L_{K(n-i)}$. Therefore the adelic rings take the form \[\1_\ad(A) = L_{E(n-\max(A))}L_{K(n-\min(A))}S^0\] for any non-empty subset $A$ of $[n]$. The functor $\Gamma_iL_i$ is the $(n-i)$-monochromatic layer functor \[M_{n-i}X := \mrm{fib}(L_{E(n-i)}X \to L_{E(n-i-1)}X).\] Therefore in the usual chromatic terminology, an $E(n)$-local spectrum $X$ with mono-dimensional support $i$ is called \emph{monochromatic} of height $n-i$, that is, $X$ is $E(n-i)$-local and the natural map $M_{n-i}X \to X$ is an equivalence. Moreover, we can now see the torsion model as a categorification of reconstructing an $E(n)$-local spectrum from its monochromatic pieces.

By definition, every vertex in an object in the torsion model is extended from the vertices of the form $i^i$, that is, from the pieces with mono-dimensional support. For any $X \in L_{E(n)}\Sp$, we can give an explicit description of these mono-dimensional pieces in the corresponding object $X_\tors$ of the torsion model. We have $X_\tors \simeq S^0_\tors \otimes X$ by \cref{rem:1tors}, and by appealing to \cref{identify1tors} we see that \[X_\tors(i^i) \simeq \Sigma^{n-i}M_{n-i}X\] so that the torsion model does indeed give the aforementioned categorification.

\subsection{Rational torus-equivariant spectra}\label{subsec:torusspec}  
In this section we discuss the main example arising from the theory of 
rational $G$-spectra where $G$ is a torus of arbitrary rank. We relate
this back to the algebraic model constructed in~\cite{GreenleesShipley18}, and discuss the prospect of an algebraic torsion model.

Recall that a   {\em family} $\mc{F}$ of subgroups of $G$ is
a set of subgroups of $G$ which is closed under conjugation and taking
subgroups. When we wish to emphasise the ambient group we will call
$\mc{F}$ a $G$-family. A {\em cofamily} is a set of subgroups of $G$
closed under conjugation and passage to larger subgroups. 
Associated to a $G$-family $\mc{F}$ are two $G$-spaces $E\mc{F}_+$ and $\widetilde{E}\mc{F}$. These are determined up to equivariant weak equivalence by their fixed points:
\[(E\mc{F}_+)^H = \begin{cases}
S^0 & \text{if $H \in \mc{F}$} \\
\ast & \text{if $H \not\in \mc{F}$} 
\end{cases} \qquad \text{and} \qquad \widetilde{E}\mc{F}^H = \begin{cases}
S^0 & \text{if $H \not\in \mc{F}$} \\
\ast & \text{if $H \in \mc{F}$} 
\end{cases}\]

We write $\Sp_G$ for the category of rational $G$-equivariant spectra, and omit the rationalization from the notation for brevity. Recall from~\cite{Greenlees19} (also see \cref{assemblyexamples}\ref{rationalassembly}) that the Balmer spectrum of $\Sp_G^\omega$ is given by the set of conjugacy classes of closed subgroups of $G$, with the poset structure given by cotoral inclusions. The support function associated to this is the \emph{geometric isotropy} defined by $\mrm{supp}(X) = \{H \mid \Phi^HX \not\simeq 0\}$ where $\Phi^H$ is the geometric $H$-fixed points functor.

If one considers the finest assembly data
(\cref{assemblyexamples}\ref{finestassembly}), the adelic rings  in
equivariant homotopy theory do not generally have good multiplicative
properties. Instead, writing $\msf{C}$ for the set of connected subgroups of $G$, we consider the assembly data $(\msf{C}, \mrm{conn}(-))$ of \cref{assemblyexamples}\ref{rationalassembly} given by the function $\mrm{conn}(-)$ which takes the connected component of the identity. We note that in $\msf{C}$ the cotoral ordering coincides with the ordering given by subgroup inclusion.

For a connected subgroup $H$ of $G$, we write $\downcl_\mrm{sb}(H)$
for the family of $\mrm{Sub}(G)$ consisting of  $K$
subconjugate to $H$, and similarly write $\upcl_\mrm{sb}(H)$ for the cofamily
consisting of  $K$ in which $H$ is subconjugate. We write $S^{\infty V(H)}$
for the $G$-space \[S^{\infty V(H)} := \bigcup_{V^H = 0} S^V.\] The
$K$-fixed points of this space are $S^0$ whenever $H \leqslant K$, and, since $G$ is a torus, are
contractible otherwise. In other words, $S^{\infty V(H)}$ is a couniversal space for the
cofamily $\upcl_\mrm{sb}(H)$.

\begin{prop}\label{Lagrees}
Let $H$ be a connected subgroup of $G$. Then \[L_H^\msf{C}(X) \simeq S^{\infty V(H)} \otimes X\] for all $X \in \Sp_G$.
\end{prop}
\begin{proof}
Recall from~\cite[\S 3]{Greenlees01b} that for any family of subgroups
$\mc{F}$, we have $L_{\mc{F}^c}X \simeq \widetilde{E}\mc{F} \otimes
X$. Since $H$ is connected, 
$\mathrm{conn}^{-1}(\upcl_{\msf{C}}(H))=\upcl_{\mrm{sb}}(H)$ which completes the proof.
\end{proof}

%\begin{rem}
%We note that a containment of identity components does not imply a
%containment, so that $\mrm{conn}^{-1}(\downcl_\msf{C}(H)) \neq
%\downcl_{\mrm{sb}}(H)$. 
%indeed, $H \times F$ for $F$ a finite non-trivial abelian group is in the former but not %the latter. As such we do not obtain analogous characterizations of $\Gamma_H^\msf{%C}$ and $\Lambda_H^\msf{C}$. We will give descriptions of these functors below in \c%ref{GammaLambda}, but their forms are more complicated than for $L_H^\msf{C}$.
%\end{rem}

We now work towards giving a description of the adelic cube
$\1_\ad^\msf{C}$. For a connected subgroup $H$ of $G$, we write
$\mc{F}/H$ for the family of subgroups of $G$ which are finite mod
$H$. This is the family $p_*^{-1}\cF_{G/H}$, where $\cF_{G/H}$ is the family
of finite subgroups of $G/H$, $p\colon G\rightarrow G/H$ is the projection, and $p_*$ is the induced map on subgroups.

\begin{lem}\label{GammaLambda}
Let $H$ be a connected subgroup of $G$. 
We have
 \[\Gamma_H^\msf{C}X \simeq E\mc{F}/H_+ \otimes X 
\qquad \text{and} \qquad 
\Lambda_H^\msf{C}X \simeq \uHom(E\mc{F}/H_+, X)\]
for all $X \in \Sp_G$.
\end{lem}
\begin{proof}
By~\cite[\S 3]{Greenlees01b}, for any $G$-family of subgroups $\mc{F}$
we have $\Gamma_\mc{F}X \simeq E\mc{F}_+ \otimes X$ and
$\Lambda_\mc{F} \simeq \uHom(E\mc{F}_+,-)$. So it suffices to
observe that $\mrm{conn}^{-1}(\downcl_{\msf{C}}(H)) = \mc{F}/H$. 
% So it suffices to
%observe that $\mrm{conn}^{-1}(\downcl_{\msf{C}}(H)) =
%\mrm{inf}_{G/H}^G\mc{F}/H$. 
%This amounts to showing that \[\{K \leqslant G \mid \mrm{conn}(K) \leqslant H\} = \{K \leqslant G \mid KH/H \text{ finite}\}.\] For any subgroup $K$ of $G$, we have that $KH/H = K/(H \cap K)$ if $H \cap K \neq \{e\}$, and $KH/H = (K \times H)/H = K$ otherwise, as $G$ is abelian. We split the proof into cases:
%\begin{enumerate}[label=(\alph*)]
%\item \label{inc1} Suppose $K \leqslant G$ such that $\{e\} = \mrm{conn}(K) \leqslant H$. In this case $K$ is finite, so $KH/H$ is finite as it is a quotient of $K$. 
%\item \label{inc2} Suppose $K \leqslant G$ such that $\{e\} \neq \mrm{conn}(K) \leqslant H$. In this case, $H \cap K \neq \{e\}$ so that $KH/H = K/(H\cap K)$ which is a quotient of $K/\mrm{conn}(K)$ which is finite.
%\item \label{inc3} Suppose $K \leqslant G$ with $KH/H$ finite such that $H \cap K \neq \{e\}$. Here $\mrm{conn}(K) \leqslant H \cap K$.
%\item \label{inc4} Suppose $K \leqslant G$ with $KH/H$ finite such that $H \cap K = \{e\}$. In this case, $KH/H = K$ so $\mrm{conn}(K) = \{e\}$.
%\end{enumerate}
%Cases \ref{inc1} and \ref{inc2} give the forward inclusion, while \ref{inc3} and \ref{inc4} give the reverse inclusion. This completes the proof.
\end{proof}

%
%To understand the form of the\todo{from here on this example needs rethinking and reselling} adelic diagram, it suffices to understand $L_H^\msf{C} \Lambda_H^\msf{C} S^0$ for $H$ connected. It turns out that in the presence of $L_H^\msf{C}$, one can commute the inflation with the internal hom in the description of $\Lambda_H^\msf{C}$ above to obtain the following. We write $D_G$ for the functional dual in $\Sp_G$ and $D_{G/H}$ for the functional dual in $\Sp_{G/H}$. 
%\begin{cor}
%We have \[L_H^\msf{C}\Lambda_H^\msf{C}S^0 \simeq S^{\infty V(H)} \otimes \mrm{inf}_{G/H}^G D_{G/H}(E\mc{F}/H_+).\]
%\end{cor}
%\begin{proof}
%In light of \cref{Lagrees} and \cref{GammaLambda}, it suffices to prove that \[S^{\infty V(H)} \otimes \mrm{inf}_{G/H}^G D_{G/H}(E\mc{F}/H_+) \simeq S^{\infty V(H)} \otimes D_{G}(\mrm{inf}_{G/H}^GE\mc{F}/H_+).\]
%\textcolor{red}{Add this proof}
%\end{proof}

Using the previous result, we see that the punctured $(\mrm{rank}(G)+1)$-cube $\1_\ad^\msf{C}$ has vertices
\begin{equation}\label{eqn:torus}
\1_\ad^\msf{C}(A) = \hspace{-1mm}\prod_{H_n \in \msf{C}} S^{\infty V(H_n)} \otimes \hspace{-3mm}\prod_{H_{n-1} \in \msf{C}} S^{\infty V(H_{n-1})} \otimes {\cdots} \otimes \hspace{-1.7mm}\prod_{H_0 \in \msf{C}} S^{\infty V(H_0)} \otimes  D (E(\mc{F}/{H_0})_+) \quad
\end{equation}
for $A = \{i_0 < i_1 < \cdots < i_n\}$ where we follow the convention that $\mrm{rank}(H_j) = i_j$, and $D = \underline{\mrm{Hom}}(-, S^0)$ denotes the functional dual in $\Sp_G$.

%\todo[inline]{John: From here on we need to rewrite to address the problem with containments in the definition of the adelic rings}

In~\cite[\S 6]{GreenleesShipley18}, the authors define a punctured $(\mrm{rank}(G)+1)$-cube $\widetilde{R}$ which is used as the first step towards constructing an algebraic model for rational $G$-spectra. Our punctured cube $\1_\ad^\msf{C}$ differs from $\widetilde{R}$ in two ways:
\begin{enumerate}
\item we index over all subgroups of the correct rank,
  whereas~\cite{GreenleesShipley18} indexes over those of
  the correct rank which are moreover contained in the subgroup fixed
  at the previous step (the indexing in \cite{GreenleesShipley18} is
  not correct in ranks $\geqslant 3$);  
\item the model in \cite{GreenleesShipley18} replaces the $D
  (E(\mc{F}/{H_0})_+)$ appearing in
  \cref{eqn:torus} with $\mrm{inf}_{G/H_0}^G
  D_{G/H_{0}}((E\mc{F}_{G/H_0})_+)$ where $D_{G/H_{0}}(-)$ denotes the
  functional dual in the category of $G/H_0$-spectra. The point of
  doing this is to give objects whose homotopy groups are easy to
  calculate and geometrically relevant.
\end{enumerate}

The model presented here provides a counterpart of the model of
\cite{GreenleesShipley18} in terms that apply to any
tensor-triangulated category satisfying \cref{hyp:hyp}. In other words, we have distilled the key features of the filtration on the Balmer spectrum of rational $G$-spectra in general terms. The language of assembly data enables us
to discuss different variants in a common framework. In particular
this applies to the models from  \cite{adelicm} and 
from \cite{GreenleesShipley18},  
either indexing over connected subgroups or over all subgroups according
to the needs of the context (see also \cite{AGs}).  
%As such, up to the different indexing, we recover the first step in the algebraic model of~\cite{GreenleesShipley18} via our \cref{thm-adelicm}. Moreover, this clarifies the relationship between the model of~\cite{GreenleesShipley18} and the results of~\cite{adelicm}: the former is based on the assembly data $(\msf{C}, \mrm{conn})$ whereas the latter uses the finest assembly data.

Finally, we discuss the prospect of constructing an algebraic torsion model for rational torus-equivariant spectra.
In~\cite{Greenlees23}, the second-named author 
defines a graded abelian category $\mc{A}_t(G)$ and uses this to 
construct an Adams spectral sequence converging to the homotopy 
classes of maps between rational $G$-spectra. The objects of 
$\mc{A}_t(G)$ are complexes of modules subject to some modules being 
annihilated by Euler classes of appropriate representations. 

It seems natural to expect that differential graded
objects of $\mc{A}_t(G)$ provide  a model for rational $G$-spectra,
which is an equivalence $\Sp_G \simeq DG\text{-}\mc{A}_t(G)$ where the latter denotes the category of objects of $\mc{A}_t(G)$ equipped with a differential
which is compatible with the internal grading on $\mc{A}_t(G)$. 
There are two obvious approaches to proving this. 
 The first approach is to use an  adelic
model (standard, separated or complete) to give an algebraic model and then establish a cellular
equivalence with an algebraic torison model. This method may be
easier, but is unsatisfying because it is indirect. 

The more direct approach would be to start with the torsion model
$(\Sp_G)_\mrm{t}^\msf{C}$, and 
then  apply $G$-fixed points to obtain a diagram category of modules
over non-equivariant spectra.  Shipley's algebraicization
theorem~\cite{Shipley07} would enable one to consider this as a
diagram of DGAs. Since the rings in questions are obtained from
polynomial rings by localizations and products, one may hope to prove
the diagram is formal, which would give a category built from a diagram of
graded rings. From this point, one can recognize a simple relationship
between objects with mono-dimensional support and Euler torsion
modules and hence recognize $\mc{A}_t(G)$ as a cellular skeleton.

\bibliographystyle{abbrv}
\bibliography{torsion}
\end{document}